\documentclass[letterpaper,11pt,reqno,twoside]{article}
\usepackage[
  backend=biber,
  style=alphabetic,
  natbib=true,
]{biblatex}

\usepackage[english]{babel}
\usepackage{amsmath,amsfonts,amssymb,amsthm,mathabx}
\usepackage[utf8]{inputenc}
\usepackage{newunicodechar}
\newunicodechar{́}{\'}
\usepackage[margin=2cm]{geometry}
\usepackage[scr=boondoxo]{mathalpha}
\usepackage{sidenotes}
\usepackage{multirow}
\usepackage{graphicx}
\graphicspath{{images/}}
\usepackage{caption}
\usepackage{xcolor}

\usepackage{orcidlink}
\usepackage{titlecaps}
\Addlcwords{a is with of for on in are and its to the by under around}

\newcommand{\mytitle}{\titlecap{%
Numerical computation of quasiperiodic reducible saddle-node bifurcations: 
a parameterization method approach
}}
\newcommand{\myshorttitle}{\titlecap{%
Quasiperiodic reducible  saddle-node bifurcations
}}

\usepackage{hyperref}
\hypersetup{
hidelinks,
bookmarksnumbered=true,
bookmarksopen=true,
pdftitle={Quasiperiodic reducible saddle-node bifurcations},
pdfauthor={J.-Ll. Figueras, J. Gimeno, and J. Parker},
}
\usepackage{cleveref}

\newcommand{\Z}{\mathbb{Z}}
\newcommand{\R}{\mathbb{R}}
\newcommand{\C}{\mathbb{C}}
\newcommand{\T}{\mathbb{T}}
\newcommand{\I}{\boldsymbol{i}}

\DeclareMathOperator{\Span}{Span}
\newcommand{\dop}{\mathrm{D}} % derivative operator
\newcommand{\Id}{Id}
\DeclareMathOperator{\diag}{diag}

\newcommand{\bydef}{\,\stackrel{\mbox{\tiny\textnormal{\raisebox{0ex}[0ex][0ex]{def}}}}{=}\,}

\newcommand{\x}{z}
\newcommand{\vf}{F}
\newcommand{\li}{\mathscr{L}}
\let\th\relax \newcommand{\th}{\theta}
\newcommand{\omg}{\omega}
\newcommand{\kk}{K}
\newcommand{\prm}{\mu}
\newcommand{\ee}{E}
\newcommand{\f}{P}
\newcommand{\tf}{L}
\newcommand{\nf}{N}
\newcommand{\nfbif}{W}
\newcommand{\lamb}{\Lambda}

\newcommand{\seed}[1]{#1_0}
\newcommand{\kko}{\seed \kk}
\newcommand{\prmo}{\seed \prm}
\newcommand{\eeo}{\seed \ee}
\newcommand{\fo}{\seed \f}
\newcommand{\tfo}{\seed L}
\newcommand{\nfo}{\seed \nf}
\newcommand{\nfbifo}{\seed \nfbif}
\newcommand{\lambo}{\seed \lamb}

\newcommand{\tol}{\texttt{tol}}

\newcommand{\correction}[1]{\Delta{#1}}
\newcommand{\hlamb}{\correction \lamb}

\newcommand{\hkk}{\correction \kk}
\newcommand{\homg}{\correction \omg}
\newcommand{\hprm}{\correction \prm}

\newcommand{\hnf}{\correction{\nf}}
\newcommand{\hLambda}{\correction{\lamb}}
\newcommand{\hnfbif}{\correction{\nfbif}}

\newcommand{\eetor}{\ee _{\texttt{tor}}}
\newcommand{\eered}{\ee _{\texttt{red}}}

\renewcommand{\Re}{\mathrm{Re}\,}

\newcommand{\fou}[1]{\widehat{#1}}
\newcommand{\avg}[1]{\langle #1 \rangle}

\newcommand{\bifprm}{\vartheta}

\newcommand{\bifu}[1]{{#1}^{\texttt{bif}}}
\newcommand{\lambbif}{\bifu \lamb}
\newcommand{\fbif}{\bifu \f}
\newcommand{\kkbif}{\bifu \kk}
\newcommand{\prmbif}{\bifu \prm}

\newcommand{\bifprmbif}{\bifu \bifprm}
\newcommand{\bifprmo}{\seed \bifprm}
\newcommand{\hbifprm}{\correction \bifprm}

\newcommand{\nfvbif}{v ^c}
\newcommand{\nfvbifo}{\seed \nfvbif}
\newcommand{\hnfvbif}{\correction \nfvbif}
\newcommand{\nfvbifbif}{\bifu{(\nfvbif)}}

\newcommand{\lambdabif}{\lambda ^c}
\newcommand{\lambdabifo}{\seed \lambdabif}
\newcommand{\hlambdabif}{\correction \lambdabif}
\newcommand{\lambdabifbif}{\bifu{(\lambdabif)}}

\newcommand{\helpprm}{\varsigma}
\newcommand{\helpprmo}{\seed\varsigma}
\newcommand{\hhelpprm}{\correction\varsigma}

\newcommand{\tortuple}{ \mathscr{T} }
\newcommand{\tortupleo}{ \seed {\tortuple} }

\newcommand{\torbiftuple}{ \mathscr{B}  }
\newcommand{\torbiftupleo}{ \seed {\torbiftuple} }
\newcommand{\torbiftuplebif}{ \bifu {\torbiftuple} }
\newcommand{\htorbiftuple}{ \correction {\torbiftuple} }

\newtheorem{thm}{Theorem}

\newtheorem{lem}[thm]{Lemma}

\newtheorem{alg}{Algorithm}
\newcommand{\algoendsymbol}{\ensuremath{\Diamond}} % or \square

\theoremstyle{definition}
\newtheorem{dfn}{Definition}

\theoremstyle{remark}
\newtheorem{rmk}{Remark}

\title{\mytitle{}}
\author{%
Jordi-Llu\'is Figueras\,\orcidlink{0000-0002-0535-4137}\textsuperscript{(1)} \and 
Joan Gimeno\,\orcidlink{0000-0002-8707-6379}\textsuperscript{(2)} \and 
Jeremy Parker\,\orcidlink{0000-0003-2066-072X}\textsuperscript{(3,\textasteriskcentered)}}
\date{\today}
\begin{document}

\maketitle
{\small
\begin{itemize}
\renewcommand{\itemsep}{.5pt}
\item [(1)] Department of Mathematics, %
Uppsala University, Box 480, 751 06 Uppsala, Sweden, \verb+figueras@math.uu.se+
\item [(2)] Departament de Matem\`atiques i Inform\`atica, %
Universitat de Barcelona, %
Gran Via de les Corts Catalanes, 585, 08007 Barcelona, %
Spain, \verb+joan@maia.ub.es+
%\item [(3)] School of Mathematics, %
%Georgia Institute of Technology, %
%686 Cherry St., Atlanta GA. 30332-0160, %
%USA, \verb+rafael.delallave@math.gatech.edu+
\item [(3)] Division of Mathematics, %
University of Dundee, %
Dundee, DD1 4HN, Scotland, \verb+jparker002@dundee.ac.uk+
\item[\textsuperscript\textasteriskcentered] Corresponding Author
\end{itemize}
}

\begin{abstract}
We present a method for computing reducible, normally hyperbolic, invariant tori with internal quasiperiodic dynamics in autonomous ordinary differential equation systems. The approach is based on the parameterization method of KAM theory; thus, it is a Newton scheme with small divisors. Since the inner dynamics of the torus is prescribed, the corresponding system parameters for which such a torus exists are simultaneously determined. The method is amenable to a form of pseudo-arclength continuation, enabling the traversal and computation of saddle-node bifurcations. We give explicit algorithms for the methods and demonstrate their applicability with two numerical examples.
\end{abstract}

{\footnotesize {\it 2020 Mathematics Subject Classification: %
 37M20% Computational methods for dynamical systems.
, 34C45% Invariant manifolds
, 34C23%Bifurcation theory for ordinary differential equations
, 37C55%Periodic and quasiperiodic flows and diffeomorphisms
, 65P30%Numerical bifurcation problems
} 

}

\newpage
{\small\tableofcontents
}
\newpage

\pagestyle{myheadings}
\markboth{\myshorttitle{}}{J.Ll. Figueras, J. Gimeno, and J. Parker}% \qquad draft, please do not circulate}

\section{Introduction}

Quasiperiodicity, dynamics generated by two or more incommensurate frequencies, is a central phenomenon in nonlinear dynamics. In phase space, such a motion is supported on invariant tori. For example, in Hamiltonian systems, quasiperiodic tori are abundant (full measure) in the integrable limit, and KAM theory explains their persistence under perturbation \citep{Kolmogorov1954,Arnold1963Proof,Moser1962,BroerHuitemaSevryuk1996,DeLaLlave2001KAM,Chierchia2003,Haro2016Parameterization}. This motivated the development of accurate numerical methods to compute and continue invariant tori; see for example \citep{Jorba2001Numerical, SchilderOsingaVogt2005,HaroDeLaLlave2006PartII,Haro2016Parameterization,CallejaCellettiGimenoLlave2022,CallejaCellettiGimenoLlave2024}.

In dissipative (non-conservative) systems, where the forward flow contracts phase-space volume, quasiperiodicity is not, \emph{a priori}, a generic phenomenon. Instead, computational efforts have often focused on finding periodic orbits\footnote{Periodic orbits \textit{are} invariant tori: they are invariant $1$-tori embedded in phase space. However, with only one frequency (rather than the two or more required for quasiperiodicity), the mathematics and numerics are much simpler; in particular, they are not subject to small-divisor issues or phase locking. Therefore, when we say \emph{invariant tori}, we mean $d$-tori with $d\geq 2$.} \citep{viswanath2003symbolic,figueras2017numerical}, which, under certain assumptions on the system, are known to be dense on the attractor and therefore sufficient for understanding the flow \citep{Cvitanovic91, ChaosBook}. Nevertheless, periodic orbits can undergo \textit{Neimark--Sacker bifurcations}, from which a normally stable or hyperbolic invariant torus is created.
Numerical evidence suggests that many dissipative systems do indeed exhibit structurally and dynamically stable invariant tori with either quasiperiodic or phase-locked internal dynamics \citep{albers2006routes}. In particular, in fluid dynamics, the breakdown of an invariant torus arising from a Neimark-Sacker bifurcation has been hypothesized \citep{RuelleTakens1971,newhouse1978occurrence} and has been observed in experiments and simulations \citep{swinney1978transition,mainieri1989two,van2005quasi} to be a dominant route to turbulence in many configurations of the Navier--Stokes equations. Furthermore, when a system is \textit{hyperchaotic}, meaning that the attractor has more than one expanding direction, it becomes possible for hyperbolic invariant tori to be embedded within the attractor \citep{parker2022invariant}, and this quasiperiodic behaviour is believed to be relevant when attempting to describe the physics underlying high-dimensional chaos \citep{doohan2022state,song2026multiscale}. It has been speculated that in hyperchaotic scenarios, quasiperiodic invariant tori should be considered analogously to periodic orbits in low-dimensional chaos, and taken into account when using invariant solutions to compute statistics for the system \citep{cvitanovic2007continuous, parker2023predicting}. To date, the main focus of this effort has been on the special class of quasiperiodic solutions which arise when a system has a continuous symmetry, in which case the resulting invariant $2$-tori are known as ``relative periodic orbits'' within this subfield of dynamical systems. With a continuous symmetry, the computation of the quasiperiodic solutions can be reduced to the problem of finding periodic solutions in the quotient of the system obtained by identifying all symmetry-related states \citep{lopez2005relative,budanur2015periodic,parker2022variational}.

In the present work, we will not make any assumptions about the existence of a continuous symmetry, or any other special features of the system. Unlike much previous work, we do not assume the system is conservative.
The main contributions of this paper are:
\begin{enumerate}
\item a parameterization-based Newton scheme for computing normally hyperbolic quasiperiodic invariant tori in autonomous ODEs, together with their normal bundles and corrected system parameters;

\item an adapted continuation strategy with an unfolding parameter that allows the computation and traversal of saddle-node bifurcations of invariant tori; and

\item numerical demonstrations on benchmark models illustrating convergence, continuation behavior, and computational performance.
\end{enumerate}

Some previous work has focused on computing invariant tori in dissipative flows. Two broad approaches are possible: either one parameterizes the full torus in the state space of the flow, or one uses a (generalized) Poincar\'e section to transform the flow into a discrete-time map and then applies existing techniques for computing invariant circles or fixed points of maps. The former approach, if applied na\"ively, quickly encounters numerical conditioning issues \citep{Jorba2001Numerical}, so many authors use the latter approach \citep{kaas1985computation,lan2006newton,JorbaO09,sanchez2010computation}. Nevertheless, there are significant advantages to working directly with the vector field via the so-called \textit{parameterization method} \citep{Haro2016Parameterization}. First, the vector-field approach avoids the high computational cost and potential numerical instabilities associated with accurately integrating trajectories and their derivatives between sections, which is particularly relevant for high-dimensional systems or very stiff dynamics. Second, working with the flow preserves the continuous-time symmetry of the system and allows for a more natural treatment of the tangent bundle, which always contains the direction of the flow. Finally, implementing KAM-like schemes directly on the vector field avoids the need to find a suitable global transverse section, which can be non-trivial for complex invariant tori or when parameters vary through bifurcations. Moreover, parameterizing the torus naturally facilitates rigorous validation of the numerics without requiring a rigorous timestepper \cite{FiguerasHaroLuque2017}; we do not pursue this direction in the present work.

%The approach of this paper is to simultaneously parameterize both the full invariant torus together with its hyperbolic directions. This avoids inverting the large, poorly conditioned matrices that arise from a na\"ive application of Newton's method to this problem \citep{CabreFontichDeLaLlave2003, HaroDeLaLlave2006PartI}. 
%Particularly relevant previous numerical efforts with similar parameterization methods include \citep{HaroDeLaLlave2006PartII}, which focuses on invariant tori in maps \red{quasiperiodically forced??} (the discrete-time analogue of the present work), \red{finish}. These works are all underpinned by variations of KAM theory for each respective situation \red{citations}

%In this work, consider \red{todo}

The approach of this paper is to simultaneously converge, with a Newton method, parameterizations of both the full invariant torus together with its hyperbolic directions. This avoids inverting the large matrices that arise from a na\"ive application of Newton's method to a parameterized quasiperiodic torus, which is necessarily ill-conditioned because of the existence of small divisors. 
This fits into a broad family of methods which started with the work of \citet{broer1997} and \citep{cabre2003i,cabre2003ii,cabre2005} who described general parameterization methods for invariant manifolds of discrete-time dynamical systems.
Similar methods have been extended, adapted and applied to discrete-time and continuous-time `skew-product' systems \citep{HdlLL1,HdlLL2,HdlLL3,JorbaO09,figueras2016parameterization,GimenoJZ22}; partially integrable Hamiltonian systems \citep{haro2019, figueras2024}; and periodically- and quasiperiodically-forced Hamiltonian systems \citep{calleja2025}. For an introduction to the method for normally hyperbolic invariant tori in discrete-time systems, we recommend \citet{canadell2016newton}, whose method is directly analogous to our own. 

Our parameterization method is amenable to branch continuation, in which a converged solution at one choice of system parameters is used to initialize the Newton method at different nearby parameters. Furthermore, we can modify the approach to allow \textit{pseudo-arclength continuation}, in which the continuation parameter is now, to leading order, the arclength along the curve instead of any external parameter of the system. We hold fixed the $d$ frequencies of the quasiperiodic internal dynamics, in order to avoid problems of continuation associated with crossing Arnold tongues. This means that $d+1$ parameters of the system are corrected at each iteration of the continuation procedure, as finite steps are made along the curve. Pseudo-arclength continuation has been combined with different forms of the parameterization method previously. \Citet{vitolo2011quasi} studied invariant circles and their bifurcations for dissipative maps; the present work is the continuous-time analogue of their study of saddle-node bifurcations, and we do not restrict to the case $d=2$.

Bifurcations of tori, including the saddle-node bifurcations we study here and more complicated scenarios, are subtle but well understood theoretically \citep{BroerHanssmannYou2005,HanssmannVanDerMeer2005,Hanssmann2005a,Hanssmann2006,Hanssmann1998,Hanssmann2004,IoossLos1988,Los1988,SekikawaInaba2016,KamiyamaInabaSekikawaEndo2014,KuznetsovSedova2016,chenciner1979bifurcations, chenciner1979persistance,Chenciner1985}. Of particular relevance is the work of Gonz\'alez, Haro and de la Llave \citep{gonzalez2014singularity,gonzalez2022efficient}, who give both in-depth analysis and practical computational methods for \textit{non-twist} tori, which are an important special case in Hamiltonian systems.
Compared to existing numerical methods for bifurcations of quasiperiodic invariant tori, our contributions are that they can be applied far from \emph{normal form} and avoid the use of change of variables that would put them near normal form. Although the algorithm does not require a preliminary reduction of the vector field to normal form, the computed objects provide such a reduction \emph{a posteriori}, in local coordinates near the torus. More precisely, after the torus, the distinguished normal direction, and the remaining reducible normal bundle have been computed, to leading-order the dynamics can be represented locally in coordinates $(\theta,h,z)$ as 
\[
  \dot \theta = \omega,\qquad
  \dot h = h^2-\vartheta,\qquad
  \dot z = \Lambda(\vartheta)z,
\]
up to higher-order terms, where $h$ is the coordinate along the distinguished direction and $z$ denotes the remaining hyperbolic normal coordinates. The continuation parameter $\varsigma$ is then used as a regular coordinate along the branch: for each prescribed value of $\varsigma$, the algorithm determines the system parameters $\mu$ and $\vartheta$ for which the corresponding reducible torus exists. In this sense, the method avoids putting the original system into normal form as an input, but recovers the saddle-node normal-form structure from the computed parameterization.

The a posteriori theorem underlying the numerical approach developed here, together with the precise analytic assumptions needed for validation of the computed objects, is proved in the companion paper \cite{FiguerasGLP2027}. The emphasis of the present paper is instead on the derivation, implementation, and numerical performance of the algorithms.

The concepts and basic algorithm for computing an invariant torus, given a sufficiently good initial guess, are discussed in \cref{sec:computation}.
In \cref{sec:bifurcation} we then give a modified version of this algorithm that incorporates a form of pseudo-arclength continuation. %, which allows us to continue a torus converged at one value of a system parameter to the corresponding torus at a different parameter value. %In particular, this makes it possible to traverse saddle-node (fold) bifurcations in which the stability of the torus changes. Without this modification to the algorithm, converging a torus close to a bifurcation becomes numerically very poorly conditioned, as the assumption of normal hyperbolicity begins to break down \jg{for me this type of sentence is for a previous paragraph itself}.
In \cref{sec:numerical}, we demonstrate our methods on two example systems, a five-dimensional ODE and three-dimensional ODE, in both cases successfully continuing the invariant torus around a saddle-node bifurcation.  Concluding remarks are given in \cref{sec:discussion}.

%\jg{Thinks that may have not stated clearly in the intro. parameter dependence, non-conservative vs conservative, advice of eigenvalues assumptions. Still not mentioning of the theorem statement}

\section{Computation of normally hyperbolic invariant tori}% \jg{in a non-conservative system?}}
\label{sec:computation}
This section is organized into four subsections: {\sl i)} the geometric setting and invariance equations; {\sl ii)} almost-invariant objects and linearized corrections; {\sl iii)} decoupled correction equations for torus/parameters and normal bundle; and {\sl iv)} explicit algorithms.

\paragraph{Notation.} 
Given an approximate torus $(\kko,\prmo)$, we use $\eetor$ for the torus defect, $\eered$ for the reducibility defect, $\fo=(\tf\ \nf)$ for the frame, and $(\xi,\eta)$ for correction/error coordinates in this frame.

\subsection{Hyperbolic tori with a fixed basic frequency}

Normally hyperbolic invariant manifolds are robust under perturbations \citep{HirschPughShub1977}. For quasiperiodic tori, however, prescribing a specific Diophantine internal frequency typically requires parameter correction. This is precisely why the Newton--KAM strategy below simultaneously updates the embedding and selected system parameters.

We consider ordinary differential equation systems defined on a subset of $\R ^n\times \R ^{d}$ ($n$-dimensional phase space, $d$-dimensional parameter space) by
\begin{equation} \label{eq.model}
 \dot \x(t) = \vf(\x(t); \prm).
\end{equation}
where we will always assume that the vector field $\vf$ is smooth enough to ensure the validity of the second-order Taylor expansions used in our Newton schemes. 
Here $\prm$ plays a parameter role. 

A torus (or, in general, a manifold) on phase space is invariant under the flow if, at any point on it, its tangent space contains the vector field. In particular, given an embedding $\kk\colon\T^d\to \R^n$, the points of the torus are given by $\kk(\th)$, $\th\in\T^d$, and the tangent space is spanned by the column vectors of $\dop \kk(\th)$. Thus, the invariance condition translates to 
\[\operatorname{rank}\Bigl( \dop \kk(\th), \vf \bigl(\kk(\th); \prm \bigr) \Bigr)=d,\] 
or, more explicitly, there exists $f\colon \T^d\to\R^d$ such that 
\[\dop \kk(\th)f(\th)=\vf(\kk(\th); \prm),\]
where $f$ is the pullback vector field onto the ideal torus $\T^d$.
That the inner dynamics on the torus is quasiperiodic is equivalent to $f(\th)=\omg^\top$, with  $\omg \in \R ^{d}$ ergodic (i.e. $k ^\top \omega \neq 0$ for all $k \in \Z ^{d} \setminus\{0\}$).

If $\omg$ is ergodic, $\kk\colon \T^d\to \R^n$ and $\prm $ are said to satisfy the invariance equation when
\begin{equation} \label{eq.inv}
 \li _\omg [\kk] (\th) + \vf (\kk(\th); \prm) = 0 \qquad \text{for all } \th \in \T ^{d},
\end{equation}
where $\li _\omg [\kk](\th) \bydef - \dop \kk(\th)\omg$. The orbits on the torus are given by $\x(t) = \kk(\th + t\omg)$, $\th\in\T^d$. In this paper we are interested in the case that $\omg $ is Diophantine (i.e. there are constants $\nu > 0$ and $\tau \ge d$ such that 
\[
|k ^\top \omg| \ge \nu |k| _1 ^{-\tau} \qquad \text{for all } k \in \Z ^{d}\setminus\{0\},
\] 
where $|k| _1 \bydef |k_1| + \dotsb + |k_d|$).

Of interest in this paper are \emph{normally hyperbolic invariant tori}, i.e. invariant tori such that the linearized dynamics in the normal directions is hyperbolic and dominates the dynamics in the tangent directions. This translates into the existence of a splitting of the tangent space at each point of the torus into stable, unstable, and tangent bundles, which are invariant under the linearized dynamics, and such that the contraction/expansion rates in the normal directions dominate those in the tangent direction. In the case of ODEs, the tangent bundle is always neutral since it contains the direction of the flow. The normal hyperbolicity condition translates into the existence of a bundle where the linearized dynamics leave it invariant and is hyperbolic. In this paper we will only tackle \emph{reducible normally hyperbolic invariant tori}: the ODE on the bundle is conjugate to a \textbf{constant} coefficient ODE with hyperbolic matrix.

% \small\jg{Why there is a small here?}

% Throughout this paper, $M(n, d, \R)$ denotes the space of $n \times d$ real matrices \jg{In my opinion this notation is inconsistently applied through out the paper}. 
The tangent bundle to the torus is given by $\tf \bydef \dop \kko\colon \T ^{d} \to \R ^{n\times d}$ and its normal bundle by $\nf \colon \T ^{d} \to \R ^{n\times (n-d)}$.
Notice that $\tf$ is invariant under the linearized dynamics: it 
satisfies the invariance equation
\begin{equation} \label{eq.tfinveq}
\li _\omg [\tf](\th)  + \dop _\x \vf (\kko(\th) ; \prmo) \tf(\th)  = 0, \quad \text{ for all } \th \in \T ^{d}.
\end{equation}

On the normal bundle is where all infinitesimal dynamics happen. In particular, since we are dealing with reducible tori there exists a constant matrix $\lamb_{\nf}$ such that $\nf$ satisfies the invariance equation
\begin{equation} \label{eq.normalinveq}
 \li _\omg [\nf](\th) + \dop _\x \vf(\kko(\th); \prmo) \nf(\th) - \nf(\th) \lamb _{\nf} = 0, \quad \text{ for all } \th \in \T ^{d}.
\end{equation}

\subsection{Almost invariant tori}

% In practice, a computer may only get an approximation of the equations \eqref{eq.inv}, \eqref{eq.tfinveq}, and \eqref{eq.normalinveq} (i.e. the equations are zero up to some number of digits). That combined by suitable a-posteriori theorems one can ensure the existence in a quantifiable neighborhood around the given numerical approximation.   The notion of ``almost'' invariant tori captures this idea of approximation.

In practice, the computer only produces approximations of the invariance and reducibility equations, i.e. \eqref{eq.inv} and  \eqref{eq.tfinveq} are only zero up to some finite number of digits. When these defects are sufficiently small and the required nondegeneracy and small-divisor conditions are satisfied, an a posteriori theorem can be used to validate the existence of a true invariant torus near the numerical approximation; the relevant theorem and its quantitative hypotheses are given in the companion paper \cite{FiguerasGLP2027}. The notion of ``almost'' invariant tori captures this idea of approximation.

\begin{dfn}
  A pair $(\kko, \prmo)$ is called an \emph{almost invariant torus} for the system \eqref{eq.model} and $\eetor \colon \T ^{d} \to \R ^n$ is called \emph{error function} if they satisfy
\begin{equation} \label{eq.almostinveq}
 \li _\omg [\kko] (\th) + \vf (\kko(\th); \prmo) = \eetor(\th) \qquad \text{for all } \th \in \T ^{d}.
\end{equation}
\end{dfn}

\begin{rmk}
  Note that if $\eetor = 0$, then the associated pair is an invariant torus. Also note that we abuse notation and call $(\kko, \prmo)$ an almost invariant torus, although $\prmo$ is just a parameter.
\end{rmk}

One goal in this paper is that, given an almost invariant torus $(\kko, \prmo)$, we want to find corrections for $(\kk, \prm)$ such that the error function is smaller than a given tolerance.
The corrections $(\hkk, \hprm)$ satisfy an additive relationship
\begin{equation}\label{eq.tormucor}
    (\kk, \prm) = (\kko, \prmo) + (\hkk, \hprm).
\end{equation} 

As is customary, to compute the corrections $(\hkk, \hprm)$ we are going to use a first order approximation of a neighborhood of $(\kko, \prmo)$. In such a neighborhood we are going to establish local coordinates given by a local basis
\begin{equation*}
 \Span (\fo(\th)) \cong \R ^n, \qquad \fo \bydef 
 \begin{pmatrix}
  \tf & \nf
 \end{pmatrix} \quad \text{ for all } \th \in \T ^{d},
\end{equation*}
where $\tf \bydef \dop \kko\colon \T ^{d} \to \R ^{n\times d}$ is the tangent and $\nf \colon \T ^{d} \to \R ^{n\times (n-d)}$ the normal bundle. 
Notice that if $(\kko, \prmo)$ is an invariant torus, then $\tf$ is invariant under the linearized dynamics: it 
satisfies the invariance equation \eqref{eq.tfinveq}.

Within an iteration correction procedure, the frame $\fo$ changes as well since it depends on $(\kko, \prmo)$. Therefore, to proceed with a Newton process we will also need to correct the frame $\fo$. A first lemma says if the error $\eetor$ is small, then $\tf$ is almost invariant:
\begin{lem} \label{lem.tangentredee}
 If $(\kko, \prmo)$ is an almost invariant torus of \eqref{eq.model}, then $\li _\omg [\tf] + \dop _\x \vf (\kko; \prmo) \tf = \dop \eetor$.
\end{lem}
\begin{proof}
 By assumption, \eqref{eq.almostinveq} holds. Now just take derivative w.r.t. $\th$ and note that $\dop \li _\omg[\kko] = \li _\omg[\tf]$.
\end{proof}
\begin{rmk}
  The size of $\dop \eetor$ is comparable to the size of $\eetor$ in most norms, so $\tf$ is an almost invariant bundle. For example, in the analytic norm $\|\cdot\|_\rho$ defined on a complex strip of length $\rho > 0$, we have $\|\dop \eetor\|_{\rho-\varepsilon} \le \varepsilon ^{-1} \|\eetor\|_\rho$ for some $\varepsilon \in (0,1)$.
\end{rmk}

Thus, an almost invariant torus has a first order error $\eered \colon \T ^d \to \R ^{n \times (n-d)}$, called \emph{reducibility error}, given by the normal bundle of $\kko$, and it is defined as  
\begin{equation} \label{eq.normalredee}
 \li _\omg [\nf](\th) + \dop _\x \vf(\kko(\th); \prmo) \nf(\th) - \nf(\th) \lamb _{\nf} = \eered(\th), \quad \text{ for all } \th \in \T ^{d},
\end{equation}
where $\lamb _{\nf}$ is an $(n-d)$-by-$(n-d)$ matrix. In this paper we study tori under the following assumptions:
\begin{enumerate}
\renewcommand{\theenumi}{\bf A\arabic{enumi}}
\renewcommand{\labelenumi}{{\theenumi})}
    \item \label{a.1} The matrix $\lamb _{\nf}$ is hyperbolic with $n _s$ stable components and $n _u$ unstable ones (s.t. $n _s + n _u = n-d$).
    
    \item \label{a.2} The matrix $\lamb _{\nf}$ admits, after a possible change of coordinates, a diagonal-block matrix
    \begin{equation*}
        \lamb _{\nf} = 
        \begin{pmatrix}
            \lamb ^s \\ & \lamb ^u
        \end{pmatrix},
    \end{equation*}
    with $\lamb ^s = \diag(\lambda ^s_i)$ and $\lamb ^u = \diag(\lambda ^u_j)$ for $i = 1, \dotsc, n _s$ and $j = 1, \dotsc, n _u$.
    
    \item \label{a.pairdiff} All diagonal entries for $\lamb ^s$ and $\lamb ^u$ are pairwise different respectively.
    
    \item \label{a.lambreal} The matrix $\lamb _{\nf}$ is real and so the normal bundle $\nf$ too.
\end{enumerate}

Assumption~\ref{a.1} involves the dynamics splitting hyperbolic directions in stable and unstable. \ref{a.2} changes the frame $\fo$ and it allows specific algorithmic treatments based on these attracting and repelling directions. 
Assumption~\ref{a.pairdiff} ensures the solvability of cohomological equations that appear when finding the corrections $\hkk$. Finally, \ref{a.lambreal} simplifies the computational algorithms avoiding complex number arithmetic. Insights of works avoiding \ref{a.lambreal} can be found in \cite{BarcelonaGJ2026} and references therein which uses Hermitian relationships and hypergeometric transformations.

\subsection{Torus and parameter corrections} \label{sec.torcorrection}
At each Newton step, we project the defect equation onto tangent and normal coordinates of the moving frame $\f$. This avoids solving one large poorly-conditioned system directly in ambient coordinates and yields cohomological equations that can be solved mode-by-mode in Fourier space.

After a first-order Taylor expansion of \eqref{eq.almostinveq} at \eqref{eq.tormucor},
\begin{equation} \label{eq.linalmostinveq}
\eetor + \li _\omg[\hkk] + \dop _\x \vf (\kko; \prmo) \hkk + \dop _\prm \vf (\kko; \prmo) \hprm + T[\kk, \prm] = 0,
\end{equation}
where 
\begin{equation*}
 T[\kk,\prm](\th) \bydef \vf(\kk(\th); \prm) - \vf (\kko(\th); \prmo) -  \dop _\x \vf (\kko(\th); \prmo) \hkk(\th) - \dop _\prm \vf (\kko(\th); \prmo) \hprm,
\end{equation*}
contains the higher Taylor order terms. We use the frame $\fo$ to express the given error $\eetor$ and the unknown correction $\hkk$ in coordinates. That is, pairs $(\xi ^{\tf}, \xi ^{\nf})$ and $(\eta ^{\tf},\eta ^{\nf})$ such that
\begin{equation} \label{eq.coordhkk}
 \hkk(\th) = \tf(\th) \xi ^{\tf}(\th) + \nf (\th)\xi ^{\nf}(\th) \quad \text{ and } \quad
 \eetor(\th) = \tf(\th) \eta ^{\tf}(\th) + \nf (\th)\eta ^{\nf}(\th),
\end{equation}
where $(\bullet) ^{\tf} \colon \T ^{d} \to \R ^{d}$ and $(\bullet) ^{\nf} \colon \T ^{d} \to \R ^{n-d}$ for $(\bullet) \in \{\xi, \eta\}$.

Note that given $\eetor$, we get $\eta = (\eta ^{\tf},\eta ^{\nf})$ by computing $\eta (\th) = \fo(\th) ^{-1} \eetor(\th) $. Similarly, once $\xi= (\xi ^{\tf}, \xi ^{\nf})$ is discovered, then the correction  will be $\hkk(\th) = \fo(\th) \xi(\th)$.

\medskip

Going back to \eqref{eq.linalmostinveq}, neglecting quadratic error term $T$, plugging \eqref{eq.coordhkk} into \eqref{eq.linalmostinveq}, and using the reducibility error \eqref{eq.normalredee}, we obtain
\begin{equation} \label{eq.projinveq}
  \tf \eta ^{\tf} +  \nf \eta ^{\nf} + 
 \tf \li _\omg[\xi ^{\tf}] + \nf \li _\omg [\xi ^{\nf}] + 
 \dop \eetor \xi ^{\tf} + 
 (\nf \lamb _{\nf} + \eered)\xi ^{\nf} + 
 \dop _\prm \vf(\kko; \prmo) \hprm = 0.
\end{equation}
We also neglect the terms $\dop \eetor \xi ^{\tf}$ and $\eered\xi ^{\nf}$ in \eqref{eq.projinveq} since they belong to higher order terms.%. However, in the rigorous part these terms will be required to be bounded.

By means of the frame $\fo$, we obtain coordinates $b = (b ^{\tf}, b ^{\nf})$
\begin{equation*}
 \f(\th)^{-1}\dop _\prm \vf (\kko(\th); \prmo) \hprm =  \tf(\th) b ^{\tf}(\th) \hprm + \nf(\th) b ^{\nf}(\th) \hprm,
\end{equation*}
where $b ^{\tf}\colon \T ^{d} \to \R ^{d\times d}$ and $ b ^{\nf} \colon \T ^{d} \to \R ^{(n-d)\times d}$. 

\medskip

Thus, coordinate-wise \eqref{eq.projinveq} is equivalent to the equations
\begin{align}
 \li _\omg [\xi ^{\tf}](\th) + b ^{\tf}(\th) \hprm + \eta ^{\tf}(\th) &= 0, \label{eq.xiLcohom} \\
 \lamb _{\nf} \xi^{\nf} (\th) + \li _\omg [\xi ^{\nf}](\th) + b ^{\nf}(\th) \hprm + \eta ^{\nf}(\th) &= 0, \label{eq.xiNcohom}
\end{align}
which are solvable under small divisor conditions.

\medskip

The tangent component \eqref{eq.xiLcohom} is solvable in terms of Fourier transformations. Indeed, let $\fou \eta _k ^{\tf}, \fou \xi _k ^{\tf} \in \C^{d}$ and $\fou b _k ^{\tf} \in \C ^{d \times d}$ be Fourier coefficients so that for all $k \in \Z ^{d}$, 
\begin{equation*}
 \begin{split}
  \fou \eta _0^{\tf} + \fou b _0 ^{\tf} \hprm &=0, \qquad \text{for }|k| = 0\text{, } \hprm \text{ is solved}, \\
  \fou \xi _0^{\tf} &=0, \qquad \text{normalization condition.}\\
  -\I (k \cdot \omg) \fou \xi _k^{\tf} + \fou \eta _k^{\tf} + \fou b _k ^{\tf} \hprm &=0, \qquad \text{for }|k| \ne 0\text{, } \xi _k^{\tf} \text{ is solved},
 \end{split}
\end{equation*}
is solvable as long as $\fou b _0 ^{\tf}$ is invertible.

\smallskip

The normal component \eqref{eq.xiNcohom} is solvable in Fourier transformations. Indeed, let $\fou \eta _k ^{\nf}, \fou \xi _k ^{\nf} \in \C^{n-d}$ and $\fou b _k ^{\nf} \in \C ^{(n-d) \times d}$ be Fourier coefficients so that,
\begin{equation} \label{eq.xiNsol}
  (\lamb _{\nf} - \I (k \cdot \omg) )\fou \xi _k^{\nf} + \fou \eta _k^{\nf} + \fou b _k ^{\nf} \hprm =0, \qquad \text{for all  }k \in \Z ^{d}.
\end{equation}
Under assumption \ref{a.1}, $\lamb _{\nf}$ is hyperbolic (it has no eigenvalues in $\I \R$), \eqref{eq.xiNsol} is solvable in terms of $\fou \xi _k^{\nf}$.

%\begin{lem}\marginpar{J-Ll: I suggest to remove this lemma}
%\jg{for a more technical part}
% $\lamb _{\nf} \xi^{\nf} (\th) + \li _\omg [\xi ^{\nf}](\th) + b ^{\nf}(\th) \hprm + \eta ^{\nf}(\th) + \eered \fou \xi _k^{\nf} =0$ is solvable and moreover
% \[ \| \xi ^{\nf}\| _\rho \leq C \| \eta ^{\nf} + b ^{\nf} \hprm\| _\rho \]
%\end{lem}

\subsection{Normal bundle correction}
The torus and parameter corrections described in Section~\ref{sec.torcorrection} use the frame $\fo$. The tangent is already updated with the torus, but the normal bundle $\nf$ must also be corrected to reduce the reducibility error $\eered$. The resulting equations have the same Fourier-solvable structure and preserve the stable/unstable block decomposition.

The reducibility condition \eqref{eq.normalredee} provides an error for $\nf$ and $\lamb _{\nf}$. Let us consider (unknown) corrections $\hnf$ and $\hLambda$ respectively. Moreover, $\hLambda$ is diagonal due to \ref{a.2}, and  
\begin{equation*}
    \hnf \bydef \tf Q ^{\tf} +  \nf Q ^{\nf},
\end{equation*}
for some $Q ^{\tf} \colon \T ^{d} \to \R ^{d\times (n-d)}$ and $Q ^{\nf} \colon \T ^{d} \to \R ^{(n-d)\times (n-d)}$ to be determined. Plugging these corrections into \eqref{eq.normalredee} and neglecting second order terms we derive
\begin{equation} \label{eq.normalredee1}
\begin{split} 
 0&=\eered + 
 \li _\omg [\tf] Q ^{\tf} + \li _\omg [\nf] Q ^{\nf}  + 
 \tf \li _\omg [Q ^{\tf}] + \nf \li _\omg [Q ^{\nf}] + \\ & \qquad
 \dop _\x \vf(\kko; \prmo) (\tf Q^{\tf} + \nf Q^{\nf}) - 
 \nf \hLambda - (\tf Q^{\tf} + \nf Q^{\nf}) \lamb _{\nf} - (\tf Q^{\tf} + \nf Q^{\nf}) \hLambda \\
 &= \eered + \tf (\li _\omg [Q ^{\tf}] - Q ^{\tf} \lamb _{\nf}) + \nf (\li _\omg[Q^{\nf}] - Q ^{\nf}\lamb _{\nf} + \lamb _{\nf} Q^{\nf} - \hLambda) + h.o.t.
%  &=\eered + \eered Q ^{\nf} + \nf \lamb _{\nf} Q + \nf \li _\omg [Q] - \nf \hLambda - \nf Q \lamb _{\nf} - \nf Q \hLambda 
\end{split}
\end{equation}

\smallskip

Using the frame $\fo$ to express the error $\eered$ in $\eta _{\texttt{red}} = (\eta ^{\tf}_{\texttt{red}}, \eta ^{\nf}_{\texttt{red}})$, i.e.
\[
 \eta(\th) = \f(\th) ^{-1}\eered(\th)  =  \tf(\th) \eta _{\texttt{red}}^{\tf}(\th) + \nf(\th) \eta _{\texttt{red}}^{\nf}(\th) .
\]
where $\eta _{\texttt{red}}^{\tf} \colon \T ^{d} \to \R ^{d \times (n-d)} $ and $\eta _{\texttt{red}}^{\nf} \colon \T ^{d} \to \R ^{(n-d)\times (n-d)} $.  Thus, component-wise \eqref{eq.normalredee1} must satisfy
\begin{align}
 \li _\omg [Q ^{\tf}](\th) - Q ^{\tf}(\th) \lamb _{\nf} &= -\eta _{\texttt{red}}^{\tf}(\th) , \label{QL} \\
 \li _\omg [Q ^{\nf}](\th) + \lamb _{\nf} Q ^{\nf}(\th)  -  Q ^{\nf}(\th)\lamb _{\nf} - \hLambda &= -\eta _{\texttt{red}}^{\nf}(\th), \label{QN}
\end{align}
which are solvable as long as extra assumptions on the $\lamb _{\nf}$ are verified.

\smallskip

To solve \eqref{QL}, we solve by Fourier, for all $k \in \Z ^{d}$,
\begin{equation*} \label{QLfousolve}
  (-\lambda _j - \I k \cdot \omg)(\fou Q ^{\tf}_{i,j}) _k =-((\fou \eta ^{\tf}_{\texttt{red}})_{i,j})_k.
\end{equation*}
To solve \eqref{QN}, we use \ref{a.1} and \ref{a.2} to write that
\begin{equation*}
 Q ^{\nf} \bydef 
 \begin{pmatrix}
  Q ^{ss} & Q ^{su}   \\
  Q ^{us}   & Q ^{uu}
 \end{pmatrix}, \qquad
 \eta _{\texttt{red}}^{\nf} = 
 \begin{pmatrix}
  \eta ^{ss} _{\texttt{red}} & \eta ^{su} _{\texttt{red}} \\ \eta ^{us} _{\texttt{red}} & \eta ^{uu} _{\texttt{red}} 
 \end{pmatrix},
\end{equation*}
and $\lamb ^s = \diag(\lambda _1^s, \dotsc, \lambda _{s}^s)$ and  $\lamb ^u = \diag(\lambda _{1}^u, \dotsc, \lambda _{u}^u)$, which implies that $\hLambda ^s$ and $\hLambda ^u$ are diagonal too (note that \ref{a.lambreal} keeps the correction in real numbers). Thus, the equations to solve are 
\begin{align*}
 \li _\omg [ Q ^{ss}] (\th) + \lamb ^s Q ^{ss}(\th) - Q ^{ss}(\th) \lamb ^s 
 - \hLambda ^s &= -\eta ^{ss} _{\texttt{red}}(\th) \\
 \li _\omg [ Q ^{su}] (\th) + \lamb ^s Q ^{su}(\th) - Q ^{su}(\th) \lamb ^u  &= -\eta ^{su} _{\texttt{red}}(\th) \\
 \li _\omg [ Q ^{us}] (\th) + \lamb ^u Q ^{us}(\th) - Q ^{us}(\th) \lamb ^s  &= -\eta ^{us} _{\texttt{red}}(\th) \\ 
 \li _\omg [ Q ^{uu}] (\th) + \lamb ^u Q ^{uu}(\th) - Q ^{uu}(\th) \lamb ^u 
 - \hLambda ^u &= -\eta ^{uu} _{\texttt{red}}(\th).
\end{align*}
In Fourier coefficients, for all $k \in \Z ^{d}$,
\begin{align*}
  - (\hLambda ^s)_{i,i}& =-((\fou \eta ^{ss}_{\texttt{red}})_{i,i})_k & |k|&= 0  \\
  - (\hLambda ^u)_{i,i}& =-((\fou \eta ^{uu}_{\texttt{red}})_{i,i})_k & |k|&= 0  \\
  (\fou Q ^{ss}_{i,i})_0 = (\fou Q ^{uu}_{i,i})_0 &=0 \\
  (\lambda _i^s - \lambda _j^s)(\fou Q ^{ss}_{i,j}) _0& =-((\fou \eta ^{ss}_{\texttt{red}})_{i,j})_0 & |k| &= 0\text{, } i\ne j\\
  (\lambda _i^u - \lambda _j^u)(\fou Q ^{uu}_{i,j}) _0& =-((\fou \eta ^{uu}_{\texttt{red}})_{i,j})_0  & |k| &= 0\text{, } i\ne j\\
  (\lambda _i^s - \lambda _j^s - \I k \cdot \omg)(\fou Q ^{ss}_{i,j}) _k& =-((\fou \eta ^{ss}_{\texttt{red}})_{i,j})_k & |k| &\ne 0 \\
  (\lambda _i^u - \lambda _j^u - \I k \cdot \omg)(\fou Q ^{uu}_{i,j}) _k& =-((\fou \eta ^{uu}_{\texttt{red}})_{i,j})_k  & |k| &\ne 0
\end{align*}
and
 \begin{align*}
  (\lambda _i^s - \lambda _j^u - \I k \cdot \omg)(\fou Q ^{su}_{i,j}) _k& =-((\fou \eta ^{su}_{\texttt{red}})_{i,j})_k  \\
  (\lambda _i^u - \lambda _j^s - \I k \cdot \omg)(\fou Q ^{us}_{i,j}) _k& =-((\fou \eta ^{us}_{\texttt{red}})_{i,j})_k .
\end{align*}

% \begingroup
% \color{blue}
Detailed algorithms derived from these are given in Algorithm~\ref{sec:algorithms appendix}.
\Cref{alg.Kmu} alternates two correction stages per iteration: {\sl i)} torus/parameter correction from the projected invariance defect; and {\sl ii)} normal-bundle/eigenvalue correction from the reducibility defect. \Cref{alg.Kmuomg} follows the same structure but substitutes a parameter correction by one frequency component.
% among the corrected unknowns and considers.
% \endgroup

\section{Computation of saddle-node bifurcations}
\label{sec:bifurcation}
\subsection{Setup}
We first state the bifurcation setting and why the baseline algorithm degenerates near a saddle-node. Then we introduce the unfolding parameter, derive the modified correction equations, and finally provide an explicit algorithmic procedure.
Let $\dot \x = \vf (\x; \prm, \bifprm)$ be an ODE system with $(\x,\prm,\bifprm)\in \R ^{n+d+1}$ and let $\omg  \in \R ^d$ be a basic frequency vector. Let $\torbiftupleo\bydef (\kko, \prmo, \bifprmo, \fo, \lambo; \omg)$ be a tuple such that for all $\th \in \T ^d$
\begin{equation}
\begin{split}
 \li _\omg[\kko](\th) + \vf (\kko(\th); \prmo, \bifprmo) &= \eetor[\kko, \prmo](\th),\\
 \li _\omg[\fo](\th) + \dop _\x \vf (\kk(\th); \prm, \bifprmo) \fo(\th) - \fo(\th) \lambo &= \ee[\fo,\lambo](\th),
\end{split}
\end{equation}
where $\kko \colon \T ^d \to \R ^n$ is a torus parametrization, $\prmo \in \R ^d$ dissipative parameter, $\bifprmo\in \R$ continuation parameter, $\fo \colon \T ^d \to \R ^{n\times n}$ a frame, and $\lambo \in \R ^{n\times n}$ a diagonal matrix.

The functions $\eetor[\kko, \prmo]\colon \T ^d \to \R ^n$ and $\ee[\fo,\lambo]\colon \T ^d \to \R ^{n\times n}$ denote the error functions on the torus and the frame respectively, which eventually become smaller than a given tolerance, say $\tol$. Notice that $\eetor[\kko,\prmo]$ does not depend on the frame and, for construction, the error in the frame $\ee[\fo,\lambo]$ does not depend on the torus $\kko$ and the dissipative parameter $\prmo$  because we use the ``corrected'' $\kk$ and $\prm$ when computing the frame error.

\medskip

We now assume a distinguished entry on $\lamb$, say $\lambdabif$, so that slight changes of the bifurcation parameter $\bifprm$ makes $\Re (\lambdabif)$ cross the value $0$. Thus, $\lambo$ has the following form:
\begin{equation} \label{eq.lambobif}
 \lambo \bydef 
 \begin{pmatrix}
  \boldsymbol{0} \\ & \lambdabifo \\ & & \lamb _{\nfbifo} 
 \end{pmatrix}, \qquad 
 \lamb _{\nfbifo} \bydef \diag((\seed \lambda ^s) _1, \dotsc, (\seed \lambda ^s)_{n _s}, (\seed \lambda ^u)_1, \dotsc, (\seed \lambda ^u)_{n _u}),
\end{equation}
where, for simplicity we have considered $\lamb$ to be real, hence $\boldsymbol{0} \in \R ^{d\times d}$ corresponds to (already corrected) tangent direction, $\lambdabifo \in \R$ to the eigenvalue that will bifurcate under small variation of the bifurcating parameter $\bifprm$, and $\lamb _{\nfbifo} \in \R ^{(n - d-1) \times (n-d-1)}$ a diagonal matrix that splits in $n _s$ stable (i.e. $\Re ((\seed \lambda ^s) _i ) < 0$ for $i = 1, \dotsc, n _s$) and $n _u$ unstable such that $n _s + n _u = n-d-1$.

The entry order in \eqref{eq.lambobif} also fixes the order of the frame columns. Thus, $\fo$ has the form:
\begin{equation*}
    \fo (\th) \bydef 
    \begin{pmatrix}
        \dop \kk(\th) & \nfvbifo  (\th) & \nfbifo (\th)
    \end{pmatrix}
\end{equation*}
where $\dop \kk \colon \T ^d \to \R ^{n\times d}$ is (the corrected) tangent direction, $\nfvbifo \colon \T ^d \to \R ^{n\times 1}$ is the distinguished direction associated to $\lambdabifo$ that bifurcates w.r.t. the parameter $\bifprm$, and $\nfbifo\colon \T ^d \to \R ^{n\times (n-d-1)}$ the normal directions containing stable and unstable ones (and with real part different from zero).

\medskip

The target event is a saddle-node of invariant tori, detected when the distinguished normal rate crosses zero. Near this point, the standard formulation loses conditioning; therefore, we reformulate the correction step with an unfolding parameter so continuation remains regular through the turning point.
At the bifurcation, Algorithm~\ref{alg.Kmu} will fail, making it impossible to exactly get the tuple $\torbiftuplebif \bydef (\kkbif, \prmbif, \bifprmbif, \fbif, \lambbif)$. The reason is that Algorithm~\ref{alg.Kmu} assumes $\lamb _{\nfbifo}$ to have non-zero (and pairwise distinct) diagonal entries.

A generic saddle-node bifurcation arises due to a non-suitable parametrization of the solution with respect to the bifurcation parameter, see Figure~\ref{fig.snbif}.
If $\bifprm(\lambdabif)$ were considered instead of $\lambdabif(\bifprm)$, then no failure would happen in Algorithm~\ref{alg.Kmu} and $\torbiftuplebif$ would be able to be computed.
This simple idea does not easily translate into a new procedure to obtain $\torbiftuplebif$. Indeed, in such a case, we would search for a correction $\hbifprm$ of $\bifprm$ (also corrections for $\kko$, $\prmo$, $\nfbifo$, and $\lamb _{\nfbifo}$) so that the error functions are small and $\lambdabifbif = 0$ at $\bifprmo + \hbifprm$. Because $\lambdabif$ and $\nfvbif$ depend on each other, it forces to also known $\nfvbif$ at the bifurcation value, i.e. $\nfvbifbif$, which ends up with an under-determined system (more unknowns than equations to be satisfied).
\begin{figure}[ht]
    \centering
    \caption{Saddle-node bifurcation where $\Re(\lambdabifbif) = 0$.}
    \label{fig.snbif}
    \includegraphics[scale=.8]{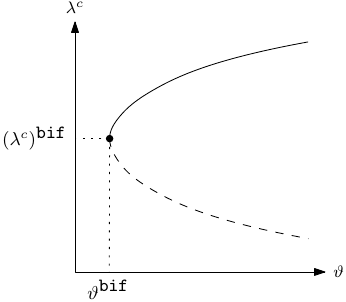}
\end{figure}

Instead of that idea, to be able to find $\torbiftuplebif$, we are going to adapt the general Algorithm~\ref{alg.Kmu} such that the continuation w.r.t. $\bifprm$ will not suffer from a saddle-node bifurcation. After that, we are going to find the parameter $\bifprm$ at the bifurcation,  $\bifprmbif$, by a simple root-finding procedure of the continuation parameter $\bifprm$ on this new algorithm. Hence, we will be able to continue $\torbiftuple$ through and also to get the information at the bifurcation point $\torbiftuplebif$.

It is important to stress that the new Algorithm~\ref{alg.Kmubif} assumes an a priori guessing of which diagonal entry in $\lamb$ will have real part crossing the zero value when $\bifprm$ is continued. Therefore,

\paragraph{Steps overview:}
Based on the illustrative saddle-node bifurcation in Figure~\ref{fig.snbif} and discussed the issues of being $\lambdabif$ the $y$-axis. We introduce a new independent parameter, say $\helpprm$, that will allow $\Re(\lambdabifbif) = 0$ to happen without an algorithm failure. The steps are:
\begin{enumerate}
    \renewcommand{\theenumi}{\arabic{enumi}}
    \renewcommand{\labelenumi}{\theenumi)}
    \item Given an initial guess tuple $(\kko, \prmo, \nfo, \lamb _{\nfo})$, use Algorithm~\ref{alg.Kmu} to compute $ \tortuple = (\kk, \prm, \nf, \lamb _{\nf})$.
    \item Perform standard continuation w.r.t. a parameter $\bifprm$ of the system
    \item Detect $\bifprmo$ when a diagonal entry in $\lamb _{\nf}$, say $\lambdabif$, has real part approaching zero.
    \item If so, compute a value $\helpprm$  from $\tortuple$ at a $\bifprm=\bifprmo$ close to that detection.
    \begin{enumerate}
    \renewcommand{\theenumii}{\theenumi.\arabic{enumii}}
    \renewcommand{\labelenumii}{\theenumii)}
        \item Use Algorithm~\ref{alg.Kmubif} to obtain $\torbiftuple = (\kk, \prm, \bifprm, \nfvbif, \nfbif, \lambdabif, \lamb _{\nfbif})$.
        \item Perform continuation w.r.t. $\helpprm$.
        \item Apply a root-finding method to get $\bifprm(\helpprm)$ at the bifurcation value, i.e. $\bifprmbif = \bifprm(\helpprm_\star)$.
    \end{enumerate}
\end{enumerate}

The parameter $\helpprm$ is a new artificial parameter coming from a normalization condition. This type of parameter is sometimes called unfolding parameter or pseudo-arclength parameter.

\begin{rmk}
 We can directly consider $\bifprm$ to be a pseudo-arclength parameter and we would not need to introduce $\helpprm$. A possible drawback is that we would need to compute the derivative w.r.t. $\bifprm$ from the original system $\vf$.
 % \jg{I think a condition like \begin{equation}
  % \| \kk - \kk _{\texttt{old}} \| _{L ^2} ^2 + \| \prm - \prm _{\texttt{old}} \| _2^2 + | \bifprm - \bifprm _{\texttt{old}} |^2 = \hbifprm
 % \end{equation}
 % i.e. a ball of radius $\hbifprm$, is able to deal with turning points without needing to use the distinguished $\lambdabif$ and $\nfvbif$. And we could apply the root-finding for $\bifprm$ after checking which is about to cross}
\end{rmk}

\subsection{Algorithm derivation}
We reuse the notation and correction philosophy of Section~\ref{sec:computation}, and only modify the parts affected by the distinguished near-neutral normal direction.
Let $\tortupleo = (\kko, \prmo, \nfo, \lamb _{\nfo})$ be an initial tuple for a parameter of the system $\bifprmo$. That is, $\kko\colon \T ^d \to \R ^n$ initial torus parametrization, $\prmo \in \R ^d$ dissipative parameters, $\nfo \colon \R ^d \to \R ^{n \times (n-d)}$ normal bundle containing stable and unstable bundles, and $\lamb _{\nfo} \in \R ^{(n-d)\times (n-d)}$ diagonal matrix.

Let us assume that there is a distinguished entry in $\lamb _{\nfo}$, say $\lambdabifo$, and let $\nfvbifo$ be the corresponding column in $\nfo$. Thus, there are $\nfbifo \in \R ^{n\times (n-d-1)}$ and $\lamb _{\nfbifo} \in \R ^{(n-d-1)\times (n-d-1)}$ such that
\begin{equation*}
    \nfo = 
    \begin{pmatrix}
        \nfvbifo & \nfbifo
    \end{pmatrix} \qquad \text{and} \qquad 
    \lamb _{\nfo} = 
    \begin{pmatrix}
        \lambdabifo \\
        & \lamb _{\nfbifo}
    \end{pmatrix}.
\end{equation*}
This $\lambdabifo$ can initially be either stable or unstable.

To construct a continuation procedure of $\tortuple$ w.r.t. the parameter $\bifprm \in \R$ of the system allowing saddle-node bifurcations, we introduce an unfolding parameter $\helpprm$ so that the continuation is performed with this other (equivalent) parameter that does not suffer from that bifurcation. Informally, the parameter $\helpprm$ will play the role of the $y$-axis in Figure~\ref{fig.snbif}, that under a continuation process is able to cross and compute the solution at the bifurcation point, say $\bifu\helpprm$, or equivalently $\bifu\bifprm  = \bifprm(\bifu\helpprm) $.
Thus, the tuple $\tortupleo$ produces a new tuple $\torbiftupleo = (\helpprmo, \kko, \prmo, \bifprmo, \nfvbifo, \nfbifo, \lambdabifo, \lamb _{\nfbifo})$, where $\helpprmo$ is chosen by the unfolding condition using data from $\tortupleo$. This specific condition is motivated by the geometry of the saddle-node bifurcation in the parameterization space. At the bifurcation point, the linearized operator becomes singular precisely along the direction of the distinguished neutral fiber $\nfvbif$. By imposing a fixed average projection of the torus embedding onto this fiber, we effectively ``unfold'' the singularity. This condition serves to anchor the parameterization in a way that remains transversal to the bifurcation manifold, allowing the Newton scheme to converge even when the physical bifurcation parameter $\bifprm$ reaches a turning point. In practice, this choice ensures that the continuation proceeds along the "arc" of the solution branch, providing a well-defined coordinate even when the Jacobian with respect to $\bifprm$ is singular. We choose this condition to be
\begin{equation} \label{eq.helpprmo}
    \helpprmo \bydef \avg{\kko ^\top \nfvbifo},
\end{equation}
where $\avg{\cdot}$ is the average of a function defined on $\T ^d$. 
Hence, based on the initial tuple $\torbiftupleo$, we have associated initial error functions given by:
\begin{align}
    % \begin{split}
 \li _\omg[\kko](\th) + \vf (\kko(\th); \prmo, \bifprmo) &= \eetor(\th), \label{eq.eeotorsn}\\
 \li _\omg[\nfvbifo](\th) + \dop _\x \vf (\kko(\th); \prmo, \bifprmo) \nfvbifo(\th) - \nfvbifo(\th) \lambdabifo &= \ee _{\nfvbif}(\th), \label{eq.eeonfvsn}\\
 \li _\omg[\nfbifo](\th) + \dop _\x \vf (\kko(\th); \prmo, \bifprmo) \nfbifo(\th) - \nfbifo(\th) \lamb _{\nfbifo} &= \ee _{\nfbif}(\th), \label{eq.eeonfsn}\\
 \avg{\kko^\top \nfvbifo} - \helpprmo &= \ee _{\helpprm}.\label{eq.eeohelpsn}
    % \end{split}
\end{align}
Notice that considering the frame $\fo (\th) = 
\begin{pmatrix}
    \dop \kko(\th) & \nfvbifo(\th) & \nfbifo(\th)
\end{pmatrix}$, the error functions in \eqref{eq.eeonfvsn}--\eqref{eq.eeonfsn} are included in
\begin{equation}
\label{eq.eeo.f.bif}
    \li _\omg[\fo](\th) + \dop _\x \vf (\kko(\th); \prmo, \bifprmo) \fo(\th) - \fo(\th) \lambo = \ee _{\f}(\th),
\end{equation}
where $\lambo$ is as in \eqref{eq.lambobif}.

We will use the frame $\fo$ to obtain the error functions in that new system. To fix notation, let
\begin{equation} \label{eq.eeo.bif.coord}
    \begin{split}         
    \R ^{n} \ni \eetor(\th) &\bydef \fo(\th) \eta _{\kk}(\th) = \tfo(\th) \eta _{\kk} ^{\tf}(\th) + \nfvbifo(\th) \eta _{\kk}^{\nfvbif}(\th) + \nfbifo(\th) \eta _{\kk}^{\nfbif}(\th), \\
    \R ^{n} \ni \ee _{\nfvbif}(\th) &\bydef \fo(\th) \eta _{\nfvbif}(\th) = \tfo(\th) \eta _{\nfvbif} ^{\tf}(\th) + \nfvbifo(\th) \eta _{\nfvbif}^{\nfvbif}(\th) + \nfbifo(\th) \eta _{\nfvbif}^{\nfbif}(\th), \\
    \R ^{n\times (n-d-1)} \ni \ee _{\nfbif}(\th) &\bydef \fo(\th) \eta _{\nfbif}(\th) = \tfo(\th) \eta _{\nfbif} ^{\tf}(\th) + \nfvbifo(\th) \eta _{\nfbif}^{\nfvbif}(\th) + \nfbifo(\th) \eta _{\nfbif}^{\nfbif}(\th).
    \end{split}
\end{equation}

\medskip

The construction of the algorithm consists in finding corrections of all the $\torbiftupleo$ elements when $\helpprmo$ is updated to $\helpprm = \helpprmo + \hhelpprm$. That is,
\begin{equation*}
    \htorbiftuple = (\hhelpprm, \hkk, \hprm, \hbifprm, \hnfvbif, \hnfbif, \hlambdabif, \hlamb _{\nfbif})
\end{equation*}
subject to a given continuation step $\hhelpprm$. The tuple $\torbiftuple = \torbiftupleo + \htorbiftuple$ has error functions $\ee$ (similar to those $\eeo$ in \eqref{eq.eeotorsn}--\eqref{eq.eeohelpsn}) smaller than a given tolerance $\tol$ from where we can derive the algorithm.

We assume that $\hkk$, $\hnfvbif$, and $\hnfbif$ are of the form:
\begin{equation}
\label{eq.corr.coord.bif}
\begin{split}
    \hkk(\th) &\bydef \fo(\th) \xi _{\kk}(\th)  = \tfo(\th) \xi _{\kk} ^{\tf}(\th) + \nfvbifo(\th) \xi _{\kk}^{\nfvbif}(\th) + \nfbifo(\th) \xi _{\kk}^{\nfbif}(\th) ,\\
    \hnfvbif(\th) &\bydef \fo(\th) \xi _{\nfvbif}(\th) = \tfo(\th) \xi _{\nfvbif} ^{\tf}(\th) + \nfvbifo(\th) \xi _{\nfvbif}^{\nfvbif}(\th) + \nfbifo(\th) \xi _{\nfvbif}^{\nfbif}(\th) , \\
    \hnfbif(\th) &\bydef \fo(\th) \xi _{\nfbif}(\th) = \tfo(\th) \xi _{\nfbif} ^{\tf}(\th) + \nfvbifo(\th) \xi _{\nfbif}^{\nfvbif}(\th) + \nfbifo(\th) \xi _{\nfbif}^{\nfbif}(\th) ,
\end{split}
\end{equation}
where $\tfo \bydef \dop \kko$.

\subsubsection{Torus correction}
The corrected form applied to \eqref{eq.eeotorsn} using \eqref{eq.corr.coord.bif} yields the equation in $\R ^n$ for all elements in $\T ^d$, that is,
\begin{multline}
\label{eq.corr.tor.bif}
\li _\omg[\fo] \xi _{\kk}  + \fo \li _\omg [\xi _{\kk}] + \dop _\x \vf (\kko; \prmo, \bifprmo) \fo \xi _{\kk} +
\dop _\prm \vf (\kko; \prmo, \bifprmo) \hprm \\ + \dop _\bifprm \vf (\kko; \prmo, \bifprmo) \hbifprm + \eetor = 0.    
\end{multline}
where we have neglected the Taylor error term of order $2$, more precisely,
\begin{multline*}
    T _{\kk}(\th) \bydef \vf (\kko + \hkk; \prmo + \hprm, \bifprmo + \hbifprm) - \vf (\kko; \prmo, \bifprmo) \\ - \dop _\x \vf (\kko; \prmo, \bifprmo) \hkk - \dop _\prm \vf (\kko; \prmo, \bifprmo) \hprm - \dop _\bifprm \vf (\kko; \prmo, \bifprmo) \hbifprm.
\end{multline*}
Because $\fo(\th)$ is invertible for all $\th \in \T ^d$, let us write
\begin{equation}
\label{eq.b.bif}
    \begin{split}
    \R ^{n\times d} \ni \dop _\prm \vf (\kko(\th); \prmo, \bifprmo) &\bydef \fo(\th) b _{\prm}(\th) = \tfo(\th) b _{\prm} ^{\tf}(\th) + \nfvbifo(\th) b _{\prm}^{\nfvbif}(\th) + \nfbifo(\th) b _{\prm}^{\nfbif}(\th) , \\
    \R ^{n} \ni \dop _\bifprm \vf (\kko(\th); \prmo, \bifprmo)  &\bydef \fo(\th) b _{\bifprm}(\th) = \tfo(\th) b _{\bifprm} ^{\tf}(\th) + \nfvbifo(\th) b _{\prm}^{\nfvbif}(\th) + \nfbifo(\th) b _{\bifprm}^{\nfbif}(\th).
    \end{split}
\end{equation}
Using \eqref{eq.eeo.f.bif} and \eqref{eq.b.bif} in \eqref{eq.corr.tor.bif}, we derive $\li _\omg [\xi _{\kk}] + \lambo \xi _{\kk} + b _\prm \hprm + b _\bifprm \hbifprm + \eta _{\kk} = 0$ after neglecting the second order error term $\ee _{\f} \xi _\kk$. In coordinates the previous equation reads as
\begin{equation}
\label{eq.corr.tor.bif.2}
\begin{split}
    \li _\omg [\xi _{\kk} ^{\tf}] + b _\prm ^{\tf} \hprm + b _\bifprm ^{\tf} \hbifprm + \eta _{\kk} ^{\tf} &= 0, \\
    \li _\omg [\xi _{\kk}^{\nfvbif}] + \lambdabifo \xi _{\kk} ^{\nfvbif} + b _\prm ^{\nfvbif} \hprm + b _\bifprm ^{\nfvbif} \hbifprm + \eta _{\kk} ^{\nfvbif} &= 0, \\
    \li _\omg [\xi _{\kk}^{\nfbif}] + \lamb _{\nfbifo} \xi _{\kk}^{\nfbif} + b _\prm ^{\nfbif} \hprm + b _\bifprm ^{\nfbif} \hbifprm + \eta _{\kk} ^{\nfbif} &= 0
\end{split}
\end{equation}
where the unknowns are $\xi _{\kk} = (\xi _{\kk} ^{\tf}, \xi _{\kk} ^{\nfvbif}, \xi _{\kk} ^{\nfbif})$, $\hprm$, and $\hbifprm$.

\smallskip

\newcommand{\symb}{s}
For \eqref{eq.corr.tor.bif.2} to be solvable in the unknowns, it is required to have a zero average, that is,
\begin{align}
% \begin{equation}
% \label{eq.corr.tor.bif.2}
% \begin{split}
    \avg{b _\prm ^{\tf}} \hprm + \avg{b _\bifprm ^{\tf}} \hbifprm + \avg{\eta _{\kk} ^{\tf}} &= 0,  \label{eq.avg.hprm} \\
    \lambdabifo \avg{\xi _{\kk} ^{\nfvbif}} + \avg{b _\prm ^{\nfvbif}} \hprm + \avg{b _\bifprm ^{\nfvbif}} \hbifprm + \avg{\eta _{\kk} ^{\nfvbif}} &= 0,  \label{eq.avg.xi.nfv.hbifprm} \\
    \lamb _{\nfbifo} \avg{\xi _{\kk}^{\nfbif}} + \avg{b _\prm ^{\nfbif}} \hprm + \avg{b _\bifprm ^{\nfbif}} \hbifprm + \avg{\eta _{\kk} ^{\nfbif}} &= 0, \label{eq.avg.xi.kk.nfbif}
% \end{split}
% \end{equation}
\end{align}
which is underdetermined, since it has more unknowns than equations.

Let us solve \eqref{eq.avg.hprm}--\eqref{eq.avg.xi.nfv.hbifprm} parametrizing the solution (provided an invertible $d+1$ matrix condition) by a symbol $\symb$, i.e. 
\begin{equation*}
    \begin{pmatrix}
        \avg{b _\prm ^{\tf}} & \avg{b _\bifprm ^{\tf}} \\
        \avg{b _\prm ^{\nfvbif}}  & \avg{b _\bifprm ^{\nfvbif}}
    \end{pmatrix}
    \begin{pmatrix}
        (\hprm )_0 & (\hprm) _1\\
        (\hbifprm) _0 & (\hbifprm )_1
    \end{pmatrix} = -
    \begin{pmatrix}
        \avg{\eta _{\kk} ^{\tf}}  & \boldsymbol 0 \\
        \avg{\eta _{\kk} ^{\nfvbif}} & \lambdabifo
    \end{pmatrix}, \qquad (\text{here } \boldsymbol 0 \in \R ^d)
\end{equation*}
with $\hprm(\symb) \bydef (\hprm) _0 + (\hprm) _1 \symb$ and $\hbifprm(\symb) \bydef (\hbifprm) _0 + (\hbifprm) _1 \symb$.
The symbol $\symb$ will later be subject to a condition that will fix its value, and it will correspond to $\avg{\xi _{\kk}^{\nfvbif}}$.
Notice that at the bifurcation point $\Re\lambdabifo=0$ and because it is assumed to be real, $(\hprm )_1 = 0 \in \R ^{d}$ and $(\hbifprm) _1 = 0 \in \R$.

\smallskip

Thus, from \eqref{eq.avg.xi.kk.nfbif} we derive the solution of the average of $\xi _{\kk} ^{\nfbif}(\symb) = (\xi _{\kk} ^{\nfbif}) _0 + (\xi _{\kk} ^{\nfbif}) _1 \symb$, explicitly, by solving the (diagonal) linear systems
\begin{equation} \label{eq.avg.xi.symb.kk.nfbif}
    \lamb _{\nfbifo}\avg{(\xi _{\kk} ^{\nfbif}) _j} = - \avg{\eta _{\kk} ^{\nfbif}} - \avg{ b _{\prm}^{\nfbif}} (\hprm) _j - \avg{b _{\bifprm} ^{\nfbif}} (\hbifprm) _j  , \qquad j = 0,1.
\end{equation}

Having parametrized the unknowns $\hprm$ and $\hbifprm$ in terms of a symbol $\symb$, the unknown $\xi _{\kk}(\symb)$ can also be determined from \eqref{eq.corr.tor.bif.2} in terms of Fourier coefficients.

\subsubsection{Distinguished direction correction}
Using \eqref{eq.corr.coord.bif} in \eqref{eq.eeonfvsn} yields
\begin{equation} \label{eq.corr.nfv}
    \li _\omg [\fo] \xi _{\nfvbif} + \fo \li _\omg [\xi _{\nfvbif}] + \dop _\x \vf (\kko; \prmo, \bifprmo) \fo \xi _{\nfvbif} -\nfvbifo \hlambdabif - \fo \xi _{\nfvbif} \lambdabifo + \ee _{\nfvbif} + \gamma(\symb) = 0,
\end{equation}
where 
\begin{multline*}
    [\gamma(\th)](\symb) = \bigl[\dop _{\x\x}^2 \vf (\kko(\th); \prmo, \bifprmo) \fo(\th) [\xi _{\kk}(\th)](\symb) + \dop _{\prm\x}^2 \vf (\kko(\th); \prmo, \bifprmo) \hprm(\symb) \\ + \dop _{\bifprm\x}^2 \vf (\kko(\th); \prmo, \bifprmo) \hbifprm (\symb)\bigr] \nfvbifo(\th)
\end{multline*}
and we have neglected the error terms of order $2$ such as $T _{\nfvbif}(\th) (\nfvbif(\th) + \hnfvbif(\th))$ with
\begin{multline*}
    T _{\nfvbif}(\th) \bydef \dop _\x \vf \bigl(\kko(\th) + \hkk(\th); \prmo + \hprm, \bifprmo + \hbifprm \bigr) - \dop _\x \vf (\kko(\th); \prmo, \bifprmo) \\ - \dop _{\x\x}^2 \vf (\kko(\th); \prmo, \bifprmo) \hkk(\th) - \dop _{\prm\x}^2 \vf (\kko(\th); \prmo, \bifprmo) \hprm - \dop _{\bifprm\x}^2 \vf (\kko(\th); \prmo, \bifprmo) \hbifprm.
\end{multline*}
Using \eqref{eq.eeo.f.bif} in \eqref{eq.corr.nfv} and neglecting the second order error term $\ee _{\f}  \xi _{\nfvbif}$, we derive
\begin{equation}\label{eq.corr.nfv.2}
    \fo \li _\omg [\xi _{\nfvbif}] + \fo \lambo \xi _{\nfvbif} -\nfvbifo \hlambdabif - \fo \xi _{\nfvbif} \lambdabifo + \ee _{\nfvbif} + \gamma(\symb) = 0,
\end{equation}
where $[\xi _{\nfvbif}(\th)](\symb) \bydef (\xi _{\nfvbif}(\th)) _0 + (\xi _{\nfvbif}(\th)) _1 \symb$ and $\hlambdabif(\symb) \bydef (\hlambdabif) _0 + (\hlambdabif) _1 \symb$ are the unknowns parametrized by the symbol $\symb$ up to degree $1$. 

Let us now define
\begin{equation*}
    [\gamma(\th)](\symb)\bydef \fo(\th) [c _{\nfvbif}(\th)](\symb) = \fo(\th)\bigl[(c _{\nfvbif}(\th)) _0 + (c _{\nfvbif}(\th)) _1 \symb\bigr]
\end{equation*}
with
\begin{equation*}
    \fo(\th)(c _{\nfvbif}(\th)) _j  = \tfo(\th) (c _{\nfvbif} ^{\tf}(\th)) _j + \nfvbifo(\th) (c _{\nfvbif}^{\nfvbif}(\th)) _j + \nfbifo(\th) (c _{\nfvbif}^{\nfbif}(\th)) _j, \qquad j = 0,1.
\end{equation*}
Thus \eqref{eq.corr.nfv.2} in coordinates reads as 
\begin{align*}
% \begin{equation}
    % \begin{split}
        \li _\omg [\xi _{\nfvbif}^{\tf}(\symb)] - \xi _{\nfvbif}^{\tf}(\symb)\lambdabifo + c _{\nfvbif} ^{\tf}(\symb) + \eta _{\nfvbif}^{\tf}&= 0, \\
        \li _\omg [\xi _{\nfvbif}^{\nfvbif}(\symb)]  - \hlambdabif(\symb) + c _{\nfvbif} ^{\nfvbif}(\symb) + \eta _{\nfvbif}^{\nfvbif}&= 0, \\
        \li _\omg [\xi _{\nfvbif}^{\nfbif}(\symb)]  + (\lamb _{\nfbifo} - \lambdabifo \Id_{n-d-1}) \xi _{\nfvbif} ^{\nfbif}(\symb) + c _{\nfvbif} ^{\nfbif}(\symb)  + \eta _{\nfvbif}^{\nfbif}&= 0.
    % \end{split}
% \end{equation}
\end{align*}
These equations require zero average, which means six equations
\begin{align*}
% \begin{equation}
    % \begin{split}
        - \avg{(\xi _{\nfvbif}^{\tf}) _0}\lambdabifo + \avg{(c _{\nfvbif} ^{\tf}) _0} + \avg{\eta _{\nfvbif}^{\tf}}&= 0, 
        & - \avg{(\xi _{\nfvbif}^{\tf}) _1}\lambdabifo + \avg{(c _{\nfvbif} ^{\tf}) _1}&= 0, \\
        - (\hlambdabif) _0 + \avg{(c _{\nfvbif} ^{\nfvbif}) _0} + \avg{\eta _{\nfvbif}^{\nfvbif}}&= 0, 
        & - (\hlambdabif) _1 + \avg{(c _{\nfvbif} ^{\nfvbif}) _1}&= 0,\\
        (\lamb _{\nfbifo} - \lambdabifo) \avg{(\xi _{\nfvbif} ^{\nfbif}) _0}  + \avg{(c _{\nfvbif} ^{\nfbif}) _0} + \avg{\eta _{\nfvbif}^{\nfbif}}&= 0, 
        & 
        (\lamb _{\nfbifo} - \lambdabifo) \avg{(\xi _{\nfvbif} ^{\nfbif}) _1} + \avg{(c _{\nfvbif} ^{\nfbif}) _1}&= 0.
    % \end{split}
% \end{equation}
\end{align*}
We deduce the normalization conditions $\avg{(\xi _{\nfvbif} ^{\nfvbif}) _0} = \avg{(\xi _{\nfvbif} ^{\nfvbif}) _1} = 0 \in \R$.

\subsubsection{Unfolding condition correction}
Using \eqref{eq.corr.coord.bif} in \eqref{eq.eeohelpsn} yields to
\begin{equation} \label{eq.unfold.sol}
\begin{split}
 0 &= \avg{(\fo \xi _{\kk}(\symb))^\top \nfvbifo} + \avg{\kko^\top \fo\xi _{\nfvbif}(\symb)} + \ee _{\helpprm} - \hhelpprm\\ & =
 \avg{(\fo (\xi _{\kk}) _0)^\top \nfvbifo} + \avg{\kko^\top (\fo(\xi _{\nfvbif}) _0)} + \symb\bigl[\avg{(\fo (\xi _{\kk}) _1)^\top \nfvbifo} + \avg{\kko^\top (\fo(\xi _{\nfvbif}) _1)} \bigr] + \ee _{\helpprm} - \hhelpprm.
\end{split}
\end{equation}
where we have neglected the second order term $\avg{(\hkk(\symb)) ^\top \hnfvbif(\symb)}$. By solving \eqref{eq.unfold.sol}, we find $\symb _\star$ that, by construction, $\symb _\star \bydef \avg{\xi _\kk ^{\nfvbif}}$. This solution is possible to find as long as 
\begin{equation*}
 \avg{(\fo (\xi _{\kk}) _1)^\top \nfvbifo} + \avg{\kko^\top (\fo(\xi _{\nfvbif}) _1)} \ne 0.
\end{equation*}

\begin{rmk}[Quadratic correction]\label{rmk.quadratic-correction}
  Notice that by not disregarding the quadratic terms we obtain a quadratic equation on $\symb$ of the form 
  $\alpha + \beta \symb + \gamma \symb ^2= 0$ with 
\begin{equation*}
 \begin{split}
  \alpha &= \avg{(\fo (\xi _{\kk}) _0)^\top \nfvbifo} + \avg{\kko^\top (\fo(\xi _{\nfvbif}) _0)}  + \ee _{\helpprm} - \hhelpprm\\
  \beta &= \avg{(\fo (\xi _{\kk}) _1)^\top \nfvbifo} + \avg{\kko^\top (\fo(\xi _{\nfvbif}) _1)} +  \avg{(\fo (\xi _{\kk}) _0) ^\top (\fo(\xi _{\nfvbif}) _1)} +  \avg{(\fo (\xi _{\kk}) _1) ^\top (\fo(\xi _{\nfvbif}) _0)} \\
  \gamma &= \avg{(\fo (\xi _{\kk}) _1) ^\top (\fo(\xi _{\nfvbif}) _1)}
 \end{split}
\end{equation*}
  Since, in principle, all $\xi$s are small, the added terms in this new equation are of a smaller order of magnitude, securing the solvability of the equation for $\symb$ as long as \eqref{eq.unfold.sol} is solvable.
\end{rmk}

\subsubsection{Reduced normal bundle correction}
Using \eqref{eq.corr.coord.bif} in \eqref{eq.eeonfsn} yields to
\begin{equation} \label{eq.corr.nf}
    \li _\omg [\fo] \xi _{\nfbif} + \fo \li _\omg [\xi _{\nfbif}] + \dop _\x \vf (\kko; \prmo, \bifprmo) \fo \xi _{\nfbif} - \nfbifo \hlamb _{\nfbifo} - \fo \xi _{\nfbif} \lamb _{\nfbifo} + \ee _{\nfbif} + \sigma(\symb _\star) = 0,
\end{equation}
where 
\begin{multline*}
    [\sigma(\th)](\symb _\star) = \bigl[\dop _{\x\x}^2 \vf (\kko(\th); \prmo, \bifprmo) \fo(\th) [\xi _{\kk}(\th)](\symb _\star) + \dop _{\prm\x}^2 \vf (\kko(\th); \prmo, \bifprmo) \hprm(\symb _\star) \\ + \dop _{\bifprm\x}^2 \vf (\kko(\th); \prmo, \bifprmo) \hbifprm (\symb _\star)\bigr] \nfbifo(\th)
\end{multline*}
and we have neglected the error terms of order $2$ such as $T _{\nfvbif}(\th) (\nfbif(\th) + \hnfbif(\th))$.

Using \eqref{eq.eeo.f.bif} in \eqref{eq.corr.nfv} and neglecting the second order error term $\ee _{\f}  \xi _{\nfbif}$, we derive
\begin{equation}\label{eq.corr.nfbif.2}
    \fo \li _\omg [\xi _{\nfbif}] + \fo \lambo \xi _{\nfbif} -\nfbifo \hlamb _{\nfbifo} - \fo \xi _{\nfbif} \lamb _{\nfbifo} + \ee _{\nfbif} + \sigma(\symb _\star) = 0,
\end{equation}
where the unknowns $\xi _{\nfbif}$ and $\hlamb _{\nfbifo}(\symb)$ do not need to be parametrized by the symbol $\symb$. 

Let us  now define
\begin{equation*}
    [\sigma(\th)](\symb _\star)\bydef \fo(\th) c _{\nfbif}(\th) 
    = \tfo(\th) c _{\nfbif} ^{\tf}(\th) + \nfvbifo(\th) c _{\nfbif}^{\nfvbif}(\th) + \nfbifo(\th) c _{\nfbif}^{\nfbif}(\th).
\end{equation*}
Thus \eqref{eq.corr.nfbif.2} in coordinates reads as 
\begin{align*}
% \begin{equation}
    % \begin{split}
        \li _\omg [\xi _{\nfbif}^{\tf}] - \xi _{\nfbif}^{\tf}\lamb _{\nfbifo} + c _{\nfbif} ^{\tf} + \eta _{\nfbif}^{\tf}&= 0, \\
        \li _\omg [\xi _{\nfbif}^{\nfvbif}] +  \xi _{\nfbif} ^{\nfvbif}(\lambdabifo \Id _{n-d-1} - \lamb _{\nfbifo})+ c _{\nfbif} ^{\nfvbif} + \eta _{\nfbif}^{\nfvbif}&= 0, \\
        \li _\omg [\xi _{\nfbif}^{\nfbif}]  + \lamb _{\nfbifo}\xi _{\nfbif} ^{\nfbif} - \xi _{\nfbif} ^{\nfbif}\lamb _{\nfbifo}  - \hlamb _{\nfbifo}+ c _{\nfbif} ^{\nfbif}  + \eta _{\nfbif}^{\nfbif}&= 0.
    % \end{split}
% \end{equation}
\end{align*}
These equations require zero average, which means three equations
\begin{align*}
% \begin{equation}
    % \begin{split}
        - \avg{\xi _{\nfbif}^{\tf}}\lamb _{\nfbifo} + \avg{c _{\nfbif} ^{\tf}} + \avg{\eta _{\nfbif}^{\tf}}&= 0, \\
        \avg{\xi _{\nfbif} ^{\nfvbif}}(\lambdabifo \Id _{n-d-1} - \lamb _{\nfbifo})  + \avg{c _{\nfbif} ^{\nfvbif}} + \avg{\eta _{\nfbif}^{\nfvbif}}&= 0,\\        
        - \hlamb _{\nfbifo}  + \avg{c _{\nfbif} ^{\nfbif}} + \avg{\eta _{\nfbif}^{\nfbif}}&= 0,
    % \end{split}
% \end{equation}
\end{align*}
where we use that $\lamb _{\nfbifo}$ is a diagonal matrix to solve it. In particular, we deduce the normalization conditions $\avg{\xi _{\nfbif} ^{\nfbif}} =\boldsymbol 0 \in \R ^{(n-d-1)\times (n-d-1)}$.

Algorithm~\ref{alg.Kmubif} mirrors Algorithm~\ref{alg.Kmu} but augments the unknowns with bifurcation and unfolding variables. Each iteration computes torus and bundle corrections in frame coordinates, fixes the unfolding scalar from the normalization condition, and updates the distinguished and reduced normal directions consistently.
\begin{alg}[Steps to correct $(\kk, \prm, \bifprm, \nfvbif, \nfbif, \lambdabif, \lamb _{\nfbif})$ -- procedure with distinguished normal direction] \label{alg.Kmubif}
    \small \
\begin{enumerate}
 {\setlength{\itemsep}{0pt}
 \item [$\star$] \texttt{Input:} ODE like $\dot \x = \vf (\x;\prm, \bifprm)$ and ergodic frequency $\omg\in \R ^{d}$. Initial guesses of embedding $\kko\colon \T ^{d}\to \R ^n$, system parameters $\prmo \in \R ^{d}$, $\bifprmo \in \R$, normal bundle $\nfo = 
 \begin{pmatrix}
  \nfvbifo & \nfbifo
 \end{pmatrix}\colon \T ^{d} \to \R ^{n\times (1 + (n-d-1))}$, matrix $\lamb _{\nfo} = \lambdabifo \oplus \lamb _{\nfbifo} \in \R ^{(n-d)\times (n-d)}$, and initial pseudo-arclength parameter $\helpprmo = \avg{\kko ^\top \nfvbifo}$.
 % , and its step $\hhelpprm$.  
 \item [$\star$] \texttt{Output:} $\kk$, $\prm $, $\bifprm$, $\nfvbif$, $\nfbif$, $\lambdabif$, $\lamb _{\nfbif}$ s.t. \eqref{eq.eeotorsn}--\eqref{eq.eeo.bif.coord} holds with error norms smaller than a given tolerance $\tol$
 \item [$\star$]
 \texttt{Assumption:} $\Re \lambdabifo$ close to zero
 \item [$\star$] \texttt{Notation:} $\li _\omg[f](\th)\bydef -\dop f(\th) \omg$ and  $\langle \eta \rangle \bydef \int _{\T ^{d}} \eta (\th) \, d\th$
 }
 % \item $\helpprmo \gets \helpprmo + \hhelpprm$
 
 \item\label{alg.Kmubifiterstep} $\fo(\th) \gets 
 \begin{pmatrix}
  \dop \kko(\th) & \nfvbifo(\th) & \nfbifo(\th)
 \end{pmatrix}$ \hfill {\footnotesize$\rhd \fo \colon \T ^{d} \to \R ^{n\times n}$}
 
 \item $\eetor(\th) \gets \li _\omg[\kko](\th) + \vf (\kko(\th); \prmo, \bifprmo)$ \hfill {\footnotesize$\rhd \eetor \colon \T ^{d} \to \R ^n$}
 
 \item $(\eta _{\kk}^{\tf}, \eta _{\kk}^{\nfvbif}, \eta _{\kk}^{\nfbif}) \gets \fo (\th)^{-1} \eetor(\th)$ 
 \hfill {\footnotesize$\rhd (\eta _{\kk} ^{\tf}, \eta _{\kk} ^{\nfvbif}, \eta _{\kk} ^{\nfbif}) \colon \T ^{d} \to \R ^{(d + 1 + (n-d-1))\times 1}$}
 
 \item $(b _{\prm} ^{\tf}, b _{\prm} ^{\nfvbif}, b _{\prm} ^{\nfbif}) \gets \fo (\th)^{-1}\dop _{\prm} \vf (\kko(\th); \prmo, \bifprmo)$ \hfill {\footnotesize$\rhd (b _{\prm} ^{\tf}, b _{\prm} ^{\nfvbif}, b _{\prm} ^{\nfbif}) \colon \T ^{d} \to \R ^{(d + 1 + (n-d-1))\times d}$}
 
 \item $(b _{\bifprm} ^{\tf}, b _{\bifprm} ^{\nfvbif}, b _{\bifprm} ^{\nfbif}) \gets \fo (\th)^{-1}\dop _{\bifprm} \vf (\kko(\th); \prmo, \bifprmo)$ \hfill {\footnotesize$\rhd (b _{\bifprm} ^{\tf}, b _{\bifprm} ^{\nfvbif}, b _{\bifprm} ^{\nfbif}) \colon \T ^{d} \to \R ^{(d + 1 + (n-d-1))\times 1}$}
 
 \item Average condition to find $\hprm (\symb) = (\hprm) _0 +(\hprm) _1 \symb$, $\hbifprm (\symb) = (\hbifprm) _0 +(\hbifprm) _1 \symb$ parametrized by $\symb$: solve
 \begin{equation*}
    \begin{pmatrix}
        \avg{b _\prm ^{\tf}} & \avg{b _\bifprm ^{\tf}} \\
        \avg{b _\prm ^{\nfvbif}}  & \avg{b _\bifprm ^{\nfvbif}}
    \end{pmatrix}
    \begin{pmatrix}
        (\hprm) _0 & (\hprm) _1\\
        (\hbifprm) _0 & (\hbifprm) _1
    \end{pmatrix} = -
    \begin{pmatrix}
        \avg{\eta _{\kk} ^{\tf}}  & \boldsymbol 0 \\
        \avg{\eta _{\kk} ^{\nfvbif}} & \lambdabifo
    \end{pmatrix}, \qquad (\text{here } \boldsymbol 0 \in \R ^d)
 \end{equation*}
 \item Fourier step of $\xi _{\kk} ^{\tf}(\symb) = (\xi _{\kk} ^{\tf}) _0 + (\xi _{\kk} ^{\tf}) _1 \symb$ to obtain a parametrized solution: for all $k \in \Z ^{d}$ and $j = 0,1$
\begin{equation*}
 \begin{split}
  \fou {(\xi _{\kk}^{\tf})} _{j,k} &=0, \qquad \text{for }|k| = 0 \text{ (normalization condition),}\\
  -\I (k \cdot \omg) \fou {(\xi _{\kk}^{\tf})} _{j,k}  + \fou {b _{\prm}^{\tf}} _k (\hprm) _j + (1-j) \fou {\eta _{\kk}^{\tf}} _k + \fou {b _{\bifprm}^{\tf}} _k (\hbifprm) _j &=0, \qquad \text{for }|k| \ne 0\text{, } \fou {(\xi _{\kk}^{\tf})} _{j,k} \text{ is solved}
 \end{split}
\end{equation*}
 \item Fourier step of $\xi _{\kk} ^{\nfvbif}(\symb) = (\xi _{\kk} ^{\nfvbif}) _0 + (\xi _{\kk} ^{\nfvbif}) _1 \symb$, $\xi _{\kk} ^{\nfbif}(\symb) = (\xi _{\kk} ^{\nfbif}) _0 + (\xi _{\kk} ^{\nfbif}) _1 \symb$: for all $k \in \Z ^{d}$ and $j = 0,1$
\begin{equation*} 
\begin{split} 
  \bigl(\lambdabifo - \I (k \cdot \omg) \bigr)\fou {(\xi _{\kk}^{\nfvbif})} _{j,k} + (1-j) \fou {\eta _{\kk}^{\nfvbif}} _k + \fou {b _{\prm}^{\nfvbif}} _k (\hprm) _j + \fou {b _{\bifprm}^{\nfvbif}} _k (\hbifprm) _j &=0\\
  \bigl(\lamb _{\nfbifo} - \I (k \cdot \omg) \bigr)\fou {(\xi _{\kk}^{\nfbif})} _{j,k} + (1-j) \fou {\eta _{\kk}^{\nfbif}} _k + \fou {b _{\prm}^{\nfbif}} _k (\hprm) _j + \fou {b _{\bifprm}^{\nfbif}} _k (\hbifprm) _j &=0
\end{split}
\end{equation*}
 
 \item $[\gamma (\symb)](\th) = (\gamma) _0(\th) + (\gamma) _1(\th) \symb$ where for $j=0,1$, \hfill {\footnotesize$\rhd (\gamma) _j \colon \T ^d \to \R ^{n} $}
 \begin{equation*}
  \begin{split}
   (\gamma) _j(\th) \gets \biggl[&
   \dop _{\x\x}^2 \vf (\kko(\th); \prmo, \bifprmo) \fo(\th)
   \begin{pmatrix}
    (\xi _{\kk} ^{\tf}) _j(\th) \\ 
    (\xi _{\kk} ^{\nfvbif}) _j(\th) \\
    (\xi _{\kk} ^{\nfbif}) _j(\th)
   \end{pmatrix} + \\ &
   \dop _{\prm\x}^2 \vf (\kko(\th); \prmo, \bifprmo) (\hprm) _j  + 
   \dop _{\bifprm\x}^2 \vf (\kko(\th); \prmo, \bifprmo) (\hbifprm) _j \biggr] \nfvbifo(\th)
  \end{split}
 \end{equation*}
 
 \item $ \ee _{\nfvbif}(\th) \gets 
 \li _\omg[\nfvbifo](\th) + \dop _\x \vf (\kko(\th); \prmo, \bifprmo) \nfvbifo(\th) - \nfvbifo(\th) \lambdabifo $\hfill {\footnotesize$\rhd \ee _{\nfvbif}\colon \T ^{d} \to \R ^{n}$}
 
%  \item $(\eta _{\nfvbif}^{\tf}, \eta _{\nfvbif}^{\nfvbif}, \eta _{\nfvbif}^{\nfbif}) \gets \fo (\th)^{-1} \ee _{\nfvbif}(\th)$  \hfill $(\eta _{\nfvbif} ^{\tf}, \eta _{\nfvbif} ^{\nfvbif}, \eta _{\nfvbif} ^{\nfbif}) \colon \T ^{d} \to \R ^{(d + 1 + (n-d-1))\times 1}$
 
 \item $((c _{\nfvbif} ^{\tf}) _0, (c _{\nfvbif} ^{\nfvbif}) _0, (c _{\nfvbif} ^{\nfbif}) _0) \gets \fo (\th)^{-1}[(\gamma) _0(\th) + \ee _{\nfvbif}(\th)]$ \hfill {\footnotesize $\rhd ((c _{\nfvbif} ^{\tf}) _0, (c _{\nfvbif} ^{\nfvbif}) _0, (c _{\nfvbif} ^{\nfbif}) _0) \colon \T ^{d} \to \R ^{(d + 1 + (n-d-1))\times 1}$}
 
 $((c _{\nfvbif} ^{\tf}) _1, (c _{\nfvbif} ^{\nfvbif}) _1, (c _{\nfvbif} ^{\nfbif}) _1) \gets \fo (\th)^{-1}(\gamma) _1(\th)$ \hfill {\footnotesize$\rhd ((c _{\nfvbif} ^{\tf}) _1, (c _{\nfvbif} ^{\nfvbif}) _1, (c _{\nfvbif} ^{\nfbif}) _1) \colon \T ^{d} \to \R ^{(d + 1 + (n-d-1))\times 1}$}

 %$\eta _{\nfvbif} ^{\tf} \colon \T ^{d} \to \R ^{d}$, $\eta _{\nfvbif}^{\nfvbif} \colon \T ^{d} \to \R$, and $\eta _{\nfvbif} ^{\nfbif} \colon \T ^{d} \to \R ^{n-d-1}$

 \item $\hlambdabif(\symb) \gets \avg{(c _{\nfvbif} ^{\nfvbif}) _0} + \avg{(c _{\nfvbif} ^{\nfvbif}) _1} \symb$
  
 \item Fourier step of $\xi _{\nfvbif} ^{\tf}(\symb) = (\xi _{\nfvbif} ^{\tf}) _0 + (\xi _{\nfvbif} ^{\tf}) _1 \symb$, $\xi _{\nfvbif} ^{\nfvbif}(\symb) = (\xi _{\nfvbif} ^{\nfvbif}) _0 + (\xi _{\nfvbif} ^{\nfvbif}) _1 \symb$, and $\xi _{\nfvbif} ^{\nfbif}(\symb) = (\xi _{\nfvbif} ^{\nfbif}) _0 + (\xi _{\nfvbif} ^{\nfbif}) _1 \symb$: for all $k \in \Z ^{d}$ and $j = 0,1$
 \begin{equation*}
  \begin{split}
  \bigl(- \lambdabifo - \I (k \cdot \omg)\bigr) \fou {(\xi _{\nfvbif} ^{\tf})} _{j,k} + \fou {(c _{\nfvbif} ^{\tf})} _{j,k} &=0 \\
  \fou {\xi _{\nfvbif} ^{\nfvbif}} _{j,k} &= 0 \qquad \text{ for } |k|=0 \text{ normalization} \\
  - \I (k \cdot \omg) \fou {(\xi _{\nfvbif} ^{\nfvbif})} _{j,k} + \fou {(c _{\nfvbif} ^{\nfvbif})} _{j,k} &=0\qquad  \text{ for } |k|\ne 0\\ 
  \bigl(\lambda _i - \lambdabifo - \I (k \cdot \omg)\bigr) \fou {(\xi _{\nfvbif} ^{\nfbif})} _{j,k} + \fou {(c _{\nfvbif} ^{\nfbif})} _{j,k} &=0 \qquad \text{assuming }\lamb _{\nfbifo} = \diag(\lambda _i) 
  \end{split}
 \end{equation*}
 
 \item $\ee _{\helpprm} \gets \avg{\kko ^\top \nfvbifo} - \helpprmo$ \hfill {\footnotesize $\rhd \ee _{\helpprm} \in \R$}
 \item Let $\symb _\star$ be a solution of  \hfill {\footnotesize$\rhd \symb _\star \in \R$}
 \begin{equation*}  
 \begin{split}
  &\avg{(\fo(\xi _{\kk})_1)^\top (\fo(\xi_{\nfvbif})_1)} \symb^2 + \\ &\avg{\kko^\top (\fo(\xi _{\nfvbif}) _1) +(\fo (\xi _{\kk}) _1)^\top \nfvbifo +  (\fo(\xi _{\kk})_0)^\top (\fo(\xi_{\nfvbif})_1) + (\fo(\xi _{\kk})_1)^\top (\fo(\xi_{\nfvbif})_0)} \symb + \\
  &\avg{\kko^\top (\fo(\xi _{\nfvbif}) _0)+(\fo (\xi _{\kk}) _0)^\top \nfvbifo +  (\fo(\xi _{\kk})_0)^\top (\fo(\xi_{\nfvbif})_0)} + \ee _{\helpprm} =0
  \end{split}
 \end{equation*}

 \item $ \ee _{\nfbif}(\th) \gets 
 \li _\omg[\nfbifo](\th) + 
 \dop _\x \vf (\kko(\th); \prmo, \bifprmo) \nfbifo(\th) - 
 \nfbifo(\th) \lamb _{\lambdabifo}$\hfill {\footnotesize$\rhd \ee _{\nfbif}\colon \T ^{d} \to \R ^{n\times (n-d-1)}$}
 
 \item $(\eta _{\nfbif}^{\tf}, \eta _{\nfbif}^{\nfvbif}, \eta _{\nfbif}^{\nfbif}) \gets \fo (\th)^{-1} \ee _{\nfbif}(\th)$ 
 \hfill {\footnotesize$\rhd (\eta _{\nfbif} ^{\tf},\eta _{\nfbif}^{\nfvbif},\eta _{\nfbif} ^{\nfbif}) \colon \T ^{d} \to \R^{(d + 1 + (n-d-1)) \times (n-d-1)}$}
 
 \item Compute $\sigma \colon \T ^d \to \R ^{n \times (n-d-1)}$ given by
 \begin{align*}
   \sigma(\th) \gets \biggl[
   &\dop _{\x\x}^2 \vf (\kko(\th); \prmo, \bifprmo) \fo(\th)
   \begin{pmatrix}
    \xi _{\kk} ^{\tf}(\symb _\star)(\th) \\ 
    \xi _{\kk} ^{\nfvbif}(\symb _\star)(\th) \\
    \xi _{\kk} ^{\nfbif}(\symb _\star)(\th)
   \end{pmatrix} +  \\
   &\dop _{\prm\x}^2 \vf (\kko(\th); \prmo, \bifprmo) \hprm (\symb _\star)  + 
   \dop _{\bifprm\x}^2 \vf (\kko(\th); \prmo, \bifprmo) \hbifprm(\symb _\star) \biggr] \nfbifo(\th)
  \end{align*}
  
 \item $(c _{\nfbif} ^{\tf}, c _{\nfbif} ^{\nfvbif}, c _{\nfbif} ^{\nfbif}) \gets \fo (\th)^{-1}\sigma(\th)$ \hfill {\footnotesize$\rhd (c _{\nfbif} ^{\tf}, c _{\nfbif} ^{\nfvbif}, c _{\nfbif} ^{\nfbif}) \colon \T ^{d} \to \R ^{(d + 1 + (n-d-1))\times (n-d-1)}$}
 
 \item Solve in Fourier for $\xi _{\nfbif} ^{\tf}$ and $\xi _{\nfbif} ^{\nfvbif}$. For all $k \in \Z ^{d}$, assume here $\lamb _{\nfbifo} = \diag(\lambda _l) _{l = 1}^{n-d-1}$
\begin{align*}
% \begin{equation}
    % \begin{split}
        \bigl(\lambda _l - \I (k \cdot \omg)\bigr) \fou {\xi _{\nfbif}^{\tf}} _k  + \fou { c _{\nfbif} ^{\tf}} _k + \fou {\eta _{\nfbif}^{\tf}} _k&= 0, \\
        \bigl(\lambdabifo - \lambda _l - \I (k \cdot \omg)\bigr) \fou {\xi _{\nfbif}^{\nfvbif}} + \fou{c _{\nfbif} ^{\nfvbif}} _k + \fou{\eta _{\nfbif}^{\nfvbif}} _k&= 0
    % \end{split}
% \end{equation}
\end{align*}
\item Consider block-matrix views of $\xi _{\nfbif} ^{\nfbif}(\th)$, $c _{\nfbif} ^{\nfbif}(\th)$, $\eta _{\nfbif} ^{\nfbif}(\th) \in \R ^{(n-d-1) \times (n-d-1)}$,
\begin{equation*}
 \xi _{\nfbif} ^{\nfbif}(\th) \gets 
 \begin{pmatrix}
  \xi ^{ss}(\th) & \xi ^{su} (\th)\\
  \xi ^{us}(\th) & \xi ^{uu} (\th)
 \end{pmatrix}, \qquad 
 c _{\nfbif} ^{\nfbif}(\th) \gets 
 \begin{pmatrix}
  c ^{ss} (\th) & c ^{su}(\th) \\
  c ^{us} (\th) & c ^{uu}(\th) 
 \end{pmatrix}, \qquad 
 \eta _{\nfbif} ^{\nfbif}(\th) \gets 
 \begin{pmatrix}
  \eta ^{ss} (\th) & \eta ^{su} (\th)\\
  \eta ^{us} (\th) & \eta ^{uu} (\th)
 \end{pmatrix}
\end{equation*}

\item Solve in Fourier $\xi _{\nfbif} ^{\nfbif}(\th) $. For all $k \in \Z ^d$, assume here  $\lamb _{\nfbifo} = \diag(\lambda _i ^s)_{i=1}^{n _s} \oplus \diag(\lambda _j ^u)_{j=1}^{n _u}$ \hfill {\footnotesize$\rhd n-d-1 = n _s + n _u$}
\begin{align*}
% \begin{equation}
    % \begin{split}
    \fou {(\xi ^{ss}_{i,i})}_0 = \fou {(\xi ^{uu}_{i,i})}_0 &= 0 \qquad \text{normalization condition,}\\
    (\lambda _i ^s - \lambda _j ^s) \fou {(\xi ^{ss}_{i,j})} _0 + \fou {(c ^{ss}_{i,j})} _0 + \fou {(\eta ^{ss}_{i,j})} _0 &= 0 \qquad |k| = 0 \text{, }i \ne j,\\
    (\lambda _i ^u - \lambda _j ^u) \fou {(\xi ^{uu}_{i,j})} _0 + \fou {(c ^{uu}_{i,j})} _0 + \fou {(\eta ^{uu}_{i,j})} _0&= 0 \qquad |k| = 0 \text{, }i \ne j,\\
    \bigl(\lambda _i ^s - \lambda _j ^s -\I(k \cdot \omg)\bigr) \fou {(\xi ^{ss}_{i,j})} _k + \fou {(c ^{ss}_{i,j})} _k + \fou {(\eta ^{ss}_{i,j})} _k &= 0 \qquad |k| \ne 0, \\
    \bigl(\lambda _i ^s - \lambda _j ^u -\I(k \cdot \omg)\bigr) \fou {(\xi ^{su}_{i,j})} _k + \fou {(c ^{su}_{i,j})} _k + \fou {(\eta ^{su}_{i,j})} _k&= 0, \\
    \bigl(\lambda _i ^u - \lambda _j ^s -\I(k \cdot \omg)\bigr) \fou {(\xi ^{us}_{i,i})} _k + \fou {(c ^{us}_{i,j})} _k + \fou {(\eta ^{us}_{i,j})} _k&= 0, \\
    \bigl(\lambda _i ^u - \lambda _j ^u -\I(k \cdot \omg)\bigr) \fou {(\xi ^{uu}_{i,j})} _k + \fou {(c ^{uu}_{i,j})} _k + \fou {(\eta ^{uu}_{i,j})} _k&= 0 \qquad |k| \ne 0
    % \end{split}
% \end{equation}
\end{align*}

 \item $\kko(\th) \gets \kko(\th) +  \fo(\th)
 \begin{pmatrix}
  \xi _{\kk}^{\tf}(\symb _\star)(\th) \\ \xi _{\kk}^{\nfvbif}(\symb _\star)(\th) \\ \xi _{\kk}^{\nfbif}(\symb _\star)(\th)
 \end{pmatrix}$,  $\prmo \gets \prmo + \hprm(\symb _\star)$, and $\bifprmo \gets \bifprmo + \hbifprm(\symb _\star)$
 
 \item $\nfvbifo(\th) \gets \nfvbifo(\th) +  \fo(\th)
 \begin{pmatrix}
  \xi _{\nfvbif}^{\tf}(\symb _\star)(\th) \\ \xi _{\nfvbif}^{\nfvbif}(\symb _\star)(\th) \\ \xi _{\nfvbif}^{\nfbif}(\symb _\star)(\th)
 \end{pmatrix}$, and $\lambdabifo \gets \lambdabifo + \hlambdabif(\symb _\star)$
 
 \item $\nfbifo(\th) \gets \nfbifo(\th) +  \fo(\th)
 \begin{pmatrix}
  \xi _{\nfbif}^{\tf}(\th) \\ \xi _{\nfbif}^{\nfvbif}(\th) \\ \xi _{\nfbif}^{\nfbif}(\th)
 \end{pmatrix}$ and $\lamb _{\nfbifo} \gets \lamb _{\nfbifo}  +
 \begin{pmatrix}
  \lambda _i^s + \fou {(\eta ^{ss}_{i,i})}_0 + \fou{(c ^{ss}_{i,i})}_0  & 0 \\ 
  0 & \lambda _j^u + \fou{(\eta ^{uu}_{j,j})} _0 + \fou {(c ^{uu}_{j,j})} _0
 \end{pmatrix}$
 \item Iterate from step \ref{alg.Kmubifiterstep} until norms of $\eetor$, $\ee _{\nfvbif}$, $\ee _{\helpprm}$, and $\ee _{\nfbif}$ are smaller than a given $\tol$\hfill \algoendsymbol
\end{enumerate} 
\end{alg}

\subsection{Implementation details} \label{sec.implementation-details}
In this subsection, we describe the discretization choices, arithmetic backends, stopping criteria, and continuation policy used in the numerical experiments.

  The algorithms were implemented in C++ using a common code path for the torus correction, the reducibility correction, and the saddle-node correction. The unknown functions $\kk$, $
  \nfvbif$, and $\nfbif$ are represented as grid functions on a uniform tensor-product mesh of $\T ^d$. Derivatives with respect to the torus variables are computed spectrally. The
  cohomological equations are solved in Fourier space: the grid data are transformed to Fourier coefficients, the equations are solved mode-by-mode as in \eqref{eq.xiLcohom},
  \eqref{eq.xiNcohom}, and the Fourier steps of Algorithm~\ref{alg.Kmubif}, and the result is transformed back to physical space. The zero Fourier modes are treated separately, since they
  determine the parameter corrections, eigenvalue corrections, and normalization conditions.

  At each Newton step the moving frame is assembled pointwise on the mesh. In the standard case this is the frame used in \eqref{eq.projinveq}; in the saddle-node case it is the frame
  introduced in \eqref{eq.eeo.f.bif}. Defects and parameter derivatives are projected onto this frame by solving pointwise linear systems, as in \eqref{eq.eeo.bif.coord}. This avoids
  forming a single large Newton matrix in the discretized ambient space. Moreover, the pointwise LU factorizations are independent for each $\th$ and we have parallelized them using OpenMP. After this
  projection, the remaining global computations are the small averaged systems for the tangent equations and the scalar or diagonal Fourier solves for the non-zero modes.

  The model-dependent quantities $\vf$, $\dop_\x\vf$, and $\dop_\prm\vf$ appearing in \eqref{eq.linalmostinveq} are evaluated by callbacks generated from the parser and jet-arithmetic tools
  of \cite{JorbaZ05,GimenoJZ22}. For Algorithm~\ref{alg.Kmubif}, the callback also evaluates $\dop_\bifprm\vf$, as required in \eqref{eq.b.bif}. A second callback evaluates the second-order
  directional actions that enter the correction of the distinguished direction and of the reduced normal bundle. These are the terms denoted by $\gamma$ in \eqref{eq.corr.nfv} and by $
  \sigma$ in \eqref{eq.corr.nf}. This callback structure separates the Newton--KAM solver from the specific vector field while keeping all model derivatives analytic.

  The implementation of Algorithm~\ref{alg.Kmubif} assumes that the normal direction involved in the saddle-node mechanism has been selected in advance. In the numerical code this direction
  is stored as the first column of the normal bundle, using the splitting in \eqref{eq.lambobif}. The scalar-normal case $n-d=1$ is treated separately and is used in the model
  \eqref{eq.saddle3d}. The case $n-d>1$ uses the distinguished-direction correction \eqref{eq.corr.nfv.2} and then corrects the remaining columns through the reduced-bundle equation
  \eqref{eq.corr.nfbif.2}. The present implementation assumes that the number of corrected dissipative parameters is $d$, so that the averaged tangent equations determine $\hprm$.

  The unfolding parameter used in the continuation runs is the scalar defined in \eqref{eq.helpprmo}. During pseudo-arclength continuation its target value is prescribed.
  Algorithm~\ref{alg.Kmubif} then corrects $\kk$, $\prm$, $\bifprm$, $\nfvbif$, $\nfbif$, $\lambdabif$, and $\lamb_{\nfbif}$ simultaneously so that the errors in
  \eqref{eq.eeotorsn}--\eqref{eq.eeohelpsn} are reduced. In the code, the remaining scalar freedom in the affine correction is fixed using the quadratic version of the unfolding equation
  described in Remark~\ref{rmk.quadratic-correction}; the selected root is the one closest to zero.

  Newton iterations are stopped when the relevant residual norms are smaller than the prescribed tolerance $\tol$. For Algorithm~\ref{alg.Kmu}, these residuals are the torus invariance
  defect \eqref{eq.almostinveq} and the reducibility defect \eqref{eq.normalredee}. For Algorithm~\ref{alg.Kmubif}, we monitor the residuals in \eqref{eq.eeotorsn}--\eqref{eq.eeohelpsn}.
  The norms reported in the numerical experiments are discrete maximum norms over all mesh points and components.

  The code supports both double precision and multiprecision arithmetic through the ongoing \texttt{TorKam} library, MPFR, and MPFI. The high-accuracy runs in the numerical experiments use multiprecision
  arithmetic; the precise working precision and tolerances are stated together with each experiment. Continuation data are written after accepted steps and include the corrected parameters,
  $\helpprm$, the distinguished normal rate $\lambdabif$, the remaining normal rates when present, and the final residuals. A sign change of $\lambdabif$ along a pseudo-arclength branch is
  used as the numerical indicator that the saddle-node bifurcation has been crossed.

  For direct continuation in a physical parameter, the previously converged solution is used as the initial guess for the next parameter value. Near a fold this strategy becomes ill-conditioned because $\bifprm$ is no longer a regular coordinate on the solution branch; this is the situation represented by the setup leading to Algorithm~\ref{alg.Kmubif}. In pseudo-arclength mode, the continuation step is instead taken in $\helpprm$. If Newton fails, the last accepted state is restored and the step is reduced without going below a minimum. If Newton is accepted in few iterations, the step is increased without exceeding a maximum.

% In this subsection, we explain the discretization choices, working precision, stopping criteria, and continuation step policy used in the numerical experiments.

\section{Numerical Experiments}

We now illustrate the algorithms on two model problems. The first one is a higher-dimensional toy model for which an explicit unperturbed torus and normal bundle are available. This
  example is used to test Algorithm~\ref{alg.Kmubif} in the case $n-d>1$, where one normal direction is distinguished and the remaining normal directions form a reduced hyperbolic bundle.
  The second one is the three-dimensional saddle-node model \eqref{eq.saddle3d}, where $n-d=1$ and the distinguished direction is the whole normal bundle.

  The purpose of the experiments is threefold. First, we verify the Newton convergence of the correction equations by monitoring the residuals associated with the torus and bundle
  invariance equations. Second, we show that the unfolding parameter $\helpprm$ allows continuation through values where direct continuation in the physical bifurcation parameter becomes
  ill-conditioned. Third, we record the corrected parameters and normal rates, in particular $\lambdabif$, whose sign change provides the numerical signature of crossing the saddle-node
  bifurcation.

  All computations use the implementation described in Section~\ref{sec.implementation-details}. Unless otherwise stated, the tori are discretized on uniform Fourier grids, the vector-field
  derivatives are evaluated analytically through the generated callbacks, and Newton iterations are stopped when the residuals fall below the tolerances specified in each experiment.

The algorithms are expected to be quadratically convergent once the initial guess is sufficiently close to a true solution. For the model \eqref{eq.saddle3d}, Figure~\ref{fig.s3d-errors}
  shows this behavior for several mesh sizes by plotting the residuals of the torus equation and the distinguished-direction equation during the first Newton correction.

\begin{figure}[ht]
  \centering
  \caption{Using 211-digits arithmetic and a model \eqref{eq.saddle3d}, residual errors in different meshes of $(\th _1, \th _2)$ of the torus $\ee$ and the distinguished direction $\ee _\nfvbif$. The last panel shows raw CPU time with no parallelization}
  \label{fig.s3d-errors}
  \setlength{\tabcolsep}{.5em} % horizontal padding between columns  
  \begin{tabular}{@{}ccc@{}}
    \parbox[t]{.24\textwidth}{\includegraphics[width=\linewidth,page=1]{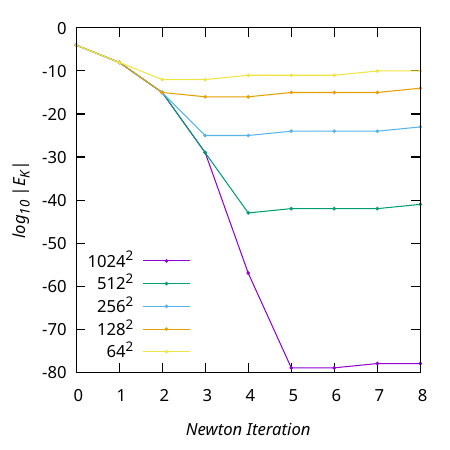}} &
    \parbox[t]{.24\textwidth}{\includegraphics[width=\linewidth,page=2]{out-s3d-702_errors.pdf}} &
    \parbox[t]{.25\textwidth}{\includegraphics[width=\linewidth,page=3]{out-s3d-702_errors.pdf}} 
  \end{tabular}
\end{figure}

\label{sec:numerical}
\subsection{Toy saddle-node model}
To assess the accuracy and convergence of the proposed algorithms, we first consider a synthetic model for
which the invariant torus and its normal fibers are available in closed form. This example primarily serves
to illustrate the practical implementation of Algorithm~\ref{alg.Kmubif}. The corresponding
code has been written from scratch in C/C++ and relies on \verb|mpfr| for multiprecision arithmetic. Vector-field
evaluations are generated via the automatic parser of the \verb|taylor| package~\citep{GimenoJZ22}, which provides
a suitable output for parallelization and supports the generation of source code targeting different data
types.

The model is constructed as a perturbation of an exactly solvable system. Since the exact invariant objects
are known, this setting allows us to quantify the numerical error precisely and to verify the expected quadratic
convergence of the Newton scheme before turning to more demanding applications. The vector field is
\begin{equation} \label{eq.toy-test}
\begin{aligned}
 \dot h   &= h^2-9+\vartheta + \varepsilon (x _1 + x _3) \\
 \dot x _1&= -7(1-r_{1,2}) x _1 + \mu _1  x _2 \tilde \omega _1 + \varepsilon \cos(h) & r _{i,j} \bydef \sqrt{x_i^2+x_j^2}\\
 \dot x _2&= -7(1-r_{1,2}) x _2 - \mu _1  x _1 \tilde \omega _1 + \varepsilon \sin(h) \\
 \dot x _3&= -5(1-r_{3,4}) x _3 + \mu _2  x _4 \tilde \omega _2 + \varepsilon \sin(h) \\
 \dot x _4&= -5(1-r_{3,4}) x _4 - \mu _2  x _3 \tilde \omega _2 + \varepsilon \cos(h).
\end{aligned}
\end{equation}
Here $\vartheta$ is the bifurcation parameter associated with the $h$-direction. For $\varepsilon=0$ and
$\vartheta = 0$, \eqref{eq.toy-test} admits the explicit solution.
\begin{equation} \label{eq.toy-init-guess}
\begin{gathered}
 \kko(\th _1, \th _2) =
 \begin{pmatrix}
  3 \\ \cos (\th_1) \\ -\sin(\th _1) \\ \cos (\th_2) \\ -\sin(\th _2)
 \end{pmatrix}, \qquad
 \nfo(\th _1, \th _2) =
 \begin{pmatrix}
  1 & 0 & 0 \\
  0 &\cos \th _1 &  0 \\
  0 &-\sin \th _1 &  0 \\
  0  & 0& \cos \th _2 \\
  0  & 0& -\sin \th _2 \\
 \end{pmatrix},
 \lambo =
 \begin{pmatrix}
  6 & 0 & 0\\
  0  & 7 & 0 \\
  0  & 0 & 5
 \end{pmatrix}.
\end{gathered}
\end{equation}
These functions satisfy the invariance equations \eqref{eq.inv} and \eqref{eq.normalinveq} with frequency vector
$(\omega_1,\omega_2)=(\tilde\omega_1,\tilde\omega_2)$. Throughout the tests we fix
\[
\tilde\omega_1=2,
\qquad
\tilde\omega_2=\frac{1}{\tilde \omega _1}\frac{\sqrt{5}-1}{2},
\qquad
\mu_1=\mu_2=1,
\qquad
\vartheta = 0.
\]
The first column of $\nfo$ is the distinguished normal direction associated with the saddle-node mechanism. At the
explicit torus we have $\partial _h(h^2-9+\vartheta)=2h=6$, which explains the first normal rate in \eqref{eq.toy-init-guess}.
The remaining two normal rates are uniformly hyperbolic and correspond to the radial directions of the two oscillatory
subsystems.

The perturbations in \eqref{eq.toy-test} are deliberately chosen to be simple: the components $\dot x_1$ to $\dot x_4$
share a common dependence on $h$, making the effect of the coupling visually apparent in the distinguished direction,
while the scalar equation for $h$ provides a codimension-one bifurcation parameter that can be tracked explicitly.

All experiments for \eqref{eq.toy-test} were performed on a standard laptop (Intel i5 @ 1.80\,GHz, 4 CPUs,
8\,GB RAM).

\subsubsection{Computation using Algorithm~\ref{alg.Kmubif}}

We apply Algorithm~\ref{alg.Kmubif} with $n = 5$ and $d = 2$, treating $\vartheta$ as the bifurcation parameter and the
first normal column of $\nfo$ as the distinguished direction. The numerical code initializes the torus and the bundle with
\eqref{eq.toy-init-guess}, performs a Newton correction at fixed $\varepsilon$, and then monitors the quantities
$\vartheta$, $\lambda _c$, and the remaining normal rates $\Lambda = (\lambda _2, \lambda _3)$ produced by the bifurcation step. In the implementation,
the unfolding scalar is chosen as the pseudo-arclength quantity $\helpprm = \langle K, v_c\rangle$, where $v_c$ denotes the
distinguished normal direction.

All runs use multiprecision arithmetic and a Fourier discretization on a uniform mesh of $\T ^2$. For small values of
$\varepsilon$, the Newton method converges rapidly from \eqref{eq.toy-init-guess}, and the resulting torus tuple provides a
controlled test case for the distinguished-direction correction equations derived in Section~\ref{sec:bifurcation}.
For instance, with $\varepsilon = 10^{-2}$ and using the double-precision driver, the corrected initial state satisfies
approximately $h = 3.0002919423$, $\vartheta = 0$, $\mu _1 = 1.0000150926$, $\mu _2 = 1.0005239031$,
$\lambda _c = 5.9999723341$, and $(\lambda _2,\lambda _3) = (6.9999915112, 5.0000007209)$, with residuals
$\max |\eetor| \approx 5.7 \times 10^{-11}$ and $\max |\eered| \approx 6.1 \times 10^{-10}$. These values are
consistent with the explicit unperturbed rates $(6,7,5)$ and illustrate that the distinguished eigenvalue remains close to its
reference value before continuation starts.

\subsubsection{Pseudo-arclength continuation in \texorpdfstring{$h=\langle K,v_c\rangle$}{h=<K,vc>}}

For this toy saddle-node model, the relevant continuation parameter is not $\varepsilon$ but the unfolding scalar
$\helpprm = \langle K, v_c\rangle$. The continuation driver therefore predicts a new state by changing the target value of
$\helpprm$ and then applies Algorithm~\ref{alg.Kmubif} to correct simultaneously the torus, the distinguished normal
direction, the reduced normal bundle, and the bifurcation parameter $\vartheta$.

This continuation strategy is precisely the one needed near the turning point, because it remains well-conditioned when
the distinguished normal rate $\lambda _c$ approaches zero. Numerically, the saddle-node crossing is detected by following
the corrected value of $\lambda _c$ along the continuation branch while the algorithm updates $\vartheta$ so that the
normalization constraint is satisfied at each step.
In the same run, a first pseudo-arclength step with target increment $\Delta h = 10^{-3}$ produces an accepted state with
$h = 3.0012902766$, $\vartheta = -5.994482657 \times 10^{-3}$, and $\lambda _c = 6.0019702291$, while the remaining
normal rates stay close to $(7,5)$. This illustrates the role of the unfolding scalar: the continuation updates $\vartheta$
and the torus simultaneously to enforce the constraint $\langle K,v_c\rangle = h$ along the branch.

\subsection{3D saddle-node model} \label{sec.3dsn}
\iffalse
\subsection{Synthetic model}

The paradigm system with a saddle-node bifurcation is the ODE $\dot x = x^2 + \bifprm$. Let us consider the perturbed system with $z \bydef (x,y) \in \R \times \R ^{n-1}$, with $\varepsilon \ll 1$,
\begin{equation}
% \begin{split}
 \vf (x,y) = 
 \begin{pmatrix}
  x ^2 + \bifprm + \varepsilon  g _1(x,y, \bifprm) \\
  y + \varepsilon  g _2 (x,y, \bifprm)
 \end{pmatrix}.
% \end{split}
\end{equation}
\fi

Finally, we consider a three-dimensional model specifically designed to exhibit a quasiperiodic saddle-node bifurcation. This system allows us to demonstrate the primary contribution of this paper: the use of the unfolding parameter $\helpprm$ and the adapted Algorithm~\ref{alg.Kmubif} to compute solutions precisely crossing the bifurcation point and continue the family of tori along the turning point.
Consider a system
\begin{equation} \label{eq.saddle3d}
\begin{aligned}
    \frac{dx}{dt} &= -\omega_1 y - \omega_2 \frac{x}{\rho}z + \frac{x}{\rho}(\rho-1)\left(\mu-\left(\sigma-\frac{1}{2}\right)^2\right)\\
    \frac{dy}{dt} &= \omega_1 x - \omega_2 \frac{y}{\rho}z + \frac{y}{\rho}(\rho-1)\left(\mu-\left(\sigma-\frac{1}{2}\right)^2\right)\\
    \frac{dz}{dt} &= \omega_2\rho-\omega_2 +z\left(\mu-\left(\sigma-\frac{1}{2}\right)^2\right)
\end{aligned}
\end{equation}
where $\omega_1$, $\omega_2$ and $\mu$ are real parameters and $\rho=\sqrt{x^2+y^2}$, $\sigma=\sqrt{(\rho-1)^2+z^2}$.
When $\mu = 0$, $\sigma=1/2$ gives a non-hyperbolic invariant torus -- this is a saddle-node bifurcation of the torus, from which two branches continue to $\mu > 0$.

\iffalse
Then
\[
\frac{d\sigma}{dt} = \sigma\left(\mu-\left(\sigma-\frac{1}{2}\right)^2\right)
\]

At $x=y=0$ so $\rho=0$, there is an invariant line.

At $\rho=1$ and $z=0$, there is a periodic orbit with period $2\pi/\omega$.

When $\mu<0$ there are no other invariant structures.

When $\mu = 0$, $\sigma=1/2$ gives a non-hyperbolic invariant torus -- this is a saddle-node bifurcation of the torus.

When $0<\mu<1/4$, there is a stable torus at $\sigma=1/2+\sqrt{\mu}$ and an unstable torus at $\sigma=1/2-\sqrt{\mu}$.

When $\mu=1/4$ I believe there's a Hopf bifurcation between the inner torus and the 
periodic orbit, and presumably for $\mu>1 /4$ there is only the stable torus and an unstable periodic orbit.
\marginpar{J-Ll: I don't know who wrote this but we can not have sentences like: "I believe..." We must 
fix this.}
\fi

Consider the regime $0<\mu<1/4$, and parameterize the torus\[\sigma=\sigma_0:=1/2+\sqrt{\mu}\] by
\begin{equation*}
\begin{aligned}
    x &= \left(1+\sigma_0\cos\phi\right)\cos\theta,\\
    y &= \left(1+\sigma_0\cos\phi\right)\sin\theta,\\
    z &= \sigma_0 \sin \phi
\end{aligned}
\end{equation*}
for $\phi,\theta\in[0,2\pi)$. The dynamics reduce to
\begin{equation*}
\begin{aligned}
     \frac{dx}{dt} &= -\omega_1 y - \omega_2 \frac{x}{\rho}z,\\
     \frac{dy}{dt} &= \omega_1 x - \omega_2 \frac{y}{\rho}z,\\
     \frac{dz}{dt} &= \omega_2(\rho-1),
\end{aligned}
\end{equation*}
so $\frac{d\theta}{dt} = \omega_1$ and $\frac{d\phi}{dt} = \omega_2$.
Observe that this is purely quasiperiodic motion, with no possibility of phase locking. This makes it a poor model of the situation in most physical systems. To remedy this we could add additional terms to the equations with small parameters.

The normal vector to the $\sigma = \sigma_0$ torus at local coordinates $(\phi,\theta)$ is given by $\nabla\sigma = [\cos\theta \cos \phi, \sin\theta \sin\phi, \sin\phi]$, and since $\frac{d\sigma}{dt}$ depends only on $\sigma$, the normal dynamics are trivial.

\subsubsection{Model and initial conditions for Algorithm~\ref{alg.Kmubif}}
To apply Algorithm~\ref{alg.Kmubif} we first translate the description of the model \eqref{eq.saddle3d} into the setting in Section~\ref{sec:bifurcation}. Let us consider a vector field with an $\varepsilon$-perturbation equivalent to \eqref{eq.saddle3d} when $\varepsilon=0$, that is,
\begin{equation} \label{eq.saddle3d-run}
    \begin{aligned}
    F_1(z;\mu,C) &= -A z_2 - B \frac{z_1}{\sqrt{z_1^2+z_2^2}}z_3 + \frac{z_1}{\sqrt{z_1^2+z_2^2}}(\sqrt{z_1^2+z_2^2}-1)\left(C-\left(\sqrt{(\sqrt{z_1^2+z_2^2}-1)^2+z_3^2}-\frac{1}{2}\right)^2\right)\\
    F_2(z;\mu,C) &= A z_1 - B \frac{z_2}{\sqrt{z_1^2+z_2^2}}z_3 + \frac{z_2}{\sqrt{z_1^2+z_2^2}}(\sqrt{z_1^2+z_2^2}-1)\left(C-\left(\sqrt{(\sqrt{z_1^2+z_2^2}-1)^2+z_3^2}-\frac{1}{2}\right)^2\right)\\
    F_3(z;\mu,C) &= B\sqrt{z_1^2+z_2^2} - B +z_3\left(C-\left(\sqrt{(\sqrt{z_1^2+z_2^2}-1)^2+z_3^2}-\frac{1}{2}\right)^2\right) + \varepsilon z _2 ^2
    \end{aligned}
\end{equation}
where $d = 2$ and $n = 3$. For $\varepsilon=0$, the system has an explicit solution for a frequency $\omega = (A,B) \bydef (1,\frac{\sqrt{5}-1}{2})$, given by 
\begin{equation} \label{eq.saddle3d-init}
\renewcommand{\arraystretch}{2}
\begin{array}{l @{\hspace{2em}} l @{\hspace{2em}} l}
\kk_1(\theta) = \left(1 +\left(\frac{1}{2} + \sqrt{C _0}\right)\cos\theta_2\right)\cos\theta_1 & \nf_{11}(\theta) = \cos\theta_1 \cos\theta_2 & \lambdabif = -2 \left(\frac{1}{2} + \sqrt{C _0}\right)\sqrt{C _0} \\
    \kk_2(\theta) = \left(1 +\left(\frac{1}{2} + \sqrt{C _0}\right)\cos\theta_2\right)\sin\theta_1 & \nf_{21}(\theta) = \sin\theta_1 \cos\theta_2 & \mu = (A ,B ) \\
    \kk_3(\theta) = \left(\frac{1}{2} + \sqrt{C _0}\right)\sin\theta_2 &  \nf_{31}(\theta) = \sin\theta_2
\end{array}
\end{equation}
The initial value of $\mu$ is $(1,\frac{\sqrt{5}-1}{2})$ and $C _0 = 2 \times 10^{-3}$ (hence, $\lambdabifo \approx \mathtt{-0.04872135}$). We use $\varepsilon$ as continuation parameter and use $C$ as $\bifprm$ in Algorithm~\ref{alg.Kmubif}.

\subsubsection{Continuation of Algorithm~\ref{alg.Kmubif}}

The experiment for this example has been performed on an Intel(R) Xeon(R) w5-3433, 16 CPUs, and 62 GB RAM. We used a $512\times512$ mesh for $\theta$, 211-digit precision, $\mathtt{tol}=10^{-50}$ for the initial Newton solver, and $\mathtt{tolc}=10^{-16}$ for Newton iterations during continuation (reset to $\mathtt{tol}$ upon reaching the final continuation value).
The size of the continuation step was adapted by multiplying it by $1.1$ if Newton is successful in fewer than 3 iterates and by $0.8$ when Newton fails. In the event of 4 consecutive failures, the continuation method is considered unsuccessful.

Figure~\ref{fig.sn-conti} shows the values of the stability and the corrected parameter $\prm _2$ vs the continuation and pseudo-arclength continuation.  Figure~\ref{fig.sn-conti-plots} shows, for the same continuation, torus, and distinguished direction solution plots.
\begin{figure}[h]
\centering
 \caption{Continuation of the eigenvalue $\lambdabif$ and $\prm _2$ w.r.t. $\helpprm$ and $\varepsilon$. Starting from $\bifprmo=0.1$ and $\helpprmo=0.0002$} \label{fig.sn-conti}
  \begin{tabular}{@{}cccc@{}}
    \parbox[t]{.23\textwidth}{\includegraphics[width=\linewidth,page=1]{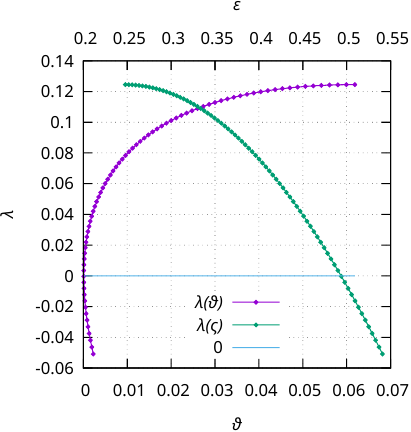}} &
    \parbox[t]{.23\textwidth}{\includegraphics[width=\linewidth,page=2]{plot-vaps-crop.pdf}} &
    \parbox[t]{.23\textwidth}{\includegraphics[width=\linewidth,page=3]{plot-vaps-crop.pdf}} &
    \parbox[t]{.23\textwidth}{\includegraphics[width=\linewidth,page=4]{plot-vaps-crop.pdf}}
  \end{tabular}
\end{figure}

% \begin{figure}[h]
% \centering
%  \includegraphics[scale=.7,page=1]{bkps/images/simple3d-cnt-1e-4.pdf} \quad
%  \includegraphics[scale=.7,page=2]{bkps/images/simple3d-cnt-1e-4.pdf}
%  \caption{\red{todo}}
% \end{figure}

\begin{figure}[h]
\centering
 \caption{Top panels initial: (1,1) solution after the first Newton step; (1,2) distinguished solution. Bottom panels: torus sliced of continuation solutions.} \label{fig.sn-conti-plots}
  \begin{tabular}{@{}cc@{}}
    \parbox[t]{.35\textwidth}{\includegraphics[width=\linewidth,page=1]{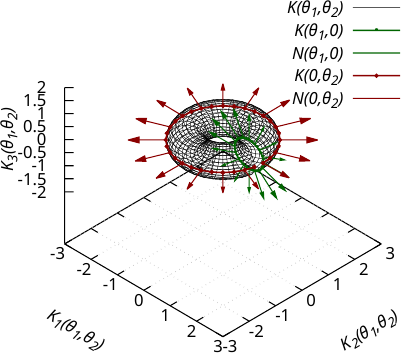}} &
    \parbox[t]{.35\textwidth}{\includegraphics[width=\linewidth,page=2]{plot_tor_cat-crop.pdf}}
  \end{tabular}
  \\
  \begin{tabular}{@{}ccc@{}}
    \parbox[t]{.3\textwidth}{\includegraphics[width=\linewidth,page=3]{plot_tor_cat-crop.pdf}} &
    \parbox[t]{.3\textwidth}{\includegraphics[width=\linewidth,page=4]{plot_tor_cat-crop.pdf}} &
    \parbox[t]{.3\textwidth}{\includegraphics[width=\linewidth,page=5]{plot_tor_cat-crop.pdf}}
  \end{tabular}
\end{figure}

% \begin{figure}[h]
% \centering
%  \includegraphics[scale=.62,page=1]{simple3d-solcnt-1e-4.pdf} 
%  \includegraphics[scale=.62,page=2]{simple3d-solcnt-1e-4.pdf}
%  \caption{\red{todo}}
% \end{figure}

% \begin{figure}[h]
% \centering
%  \includegraphics[scale=.4,page=3]{simple3d-solcnt-1e-4.pdf} 
%  \includegraphics[scale=.4,page=4]{simple3d-solcnt-1e-4.pdf}
%  \includegraphics[scale=.4,page=5]{simple3d-solcnt-1e-4.pdf}
%  \caption{\red{todo}}
% \end{figure}

\section{Discussion and Conclusions}
\label{sec:discussion}
In this paper, we have presented a robust numerical framework for the computation and continuation of normally hyperbolic invariant tori in dissipative autonomous systems. By leveraging the parameterization method, we developed a Newton-KAM scheme that avoids the inversion of large, poorly conditioned matrices, making it suitable for high-dimensional phase spaces. A key contribution of this work is the adaptation of the algorithm to handle quasiperiodic saddle-node bifurcations. By introducing an artificial unfolding parameter $\helpprm$ based on a normalization condition, we were able to transform the singular problem at the bifurcation point into a regular continuation problem.

Our numerical experiments on synthetic and benchmark models confirm the efficiency and reliability of the method. The test models demonstrate quadratic convergence of the error functions, while the application to the three-dimensional saddle-node model shows that the algorithm can precisely locate and cross turning points in the parameter space without loss of numerical stability.

Future research directions include the extension of this methodology to non-autonomous systems and quasiperiodically forced oscillators, where the interaction between multiple frequencies can lead to more complex bifurcation scenarios. Furthermore, it would be of great interest to develop algorithms that can handle hyperbolic subspaces that do not decompose into one-dimensional fibers. In principle, this is possible since such cases merely add contraction-expansion on the normal directions without changing the fundamental structure of the bundles, but there is currently a lack of robust algorithms dealing with this specific scenario. Finally, exploring the scalability of the method to high-dimensional partial differential equations (PDEs) through the use of efficient Fourier-spectral implementations remains an area of active interest.

A key practical point is arithmetic precision. In KAM computations intended for rigorous validation, 
the required a posteriori bounds typically involve scales that are beyond reliable double-precision 
resolution; see \citep{FiguerasHaroLuque2017,FiguerasHaro2025}. Although this paper does not 
include a computer-assisted proof, we carried out all computations in multiprecision following 
that validation-oriented philosophy, so that the numerics are compatible with future 
rigorous extensions. The implementation is based on the \texttt{Torkam} library, 
under active development, which provides precision-agnostic infrastructure (from standard 
floating point to arbitrary precision), Fourier/cohomological solvers, and continuation 
tools for torus and normal-bundle equations within a unified code path.

A further important consideration in practical applications of this method is how to find sufficiently accurate initial guesses $K_0$ and $N_0$ so that the Newton method converges. The two examples we presented are perturbations of systems where an exact quasiperiodic solution can be derived, along with its normal bundle. In real-world problems, even when there exists a stable torus, so that $K_0$ can be found accurately with a sufficiently long numerical integration of the governing ODE, finding $N_0$ is not immediate. This is most easily done via covariant Lyapunov vectors, whose computation is known to be challenging \citep{kuptsov2012theory, ginelli2013covariant}.

\section*{Acknowledgments}
JG has been supported by the Spanish
grant PID2021-125535NB-I00 (MICINN/AEI/FEDER, UE), the Catalan grant 2021
SGR 01072, and by the Air Force Office of Scientific
Research under award number FA8655-24-1-7059. The project that led to these results also received
the support of a fellowship from ``la Caixa'' Foundation (ID
100010434), the fellowship code is LCF/BQ/PR23/11980047. Jordi-Llu\'is Figueras has been partially supported by the grant VR Grant 2024-04764.

\section*{Statements and Declarations}
GenAI Codex was used to polish parts of the text and to assist bibliography exploration.

\addcontentsline{toc}{section}{References}
{\small \printbibliography}

@preamble{"\newcommand{\noopsort}[1]{} "}

@string{CHAOS = "Chaos"}

@string{EXPMA = "Exp. Math."}

@string{SIADS = "SIAM J. Appl. Dyn. Syst."}

@inproceedings{DeLaLlave2001KAM,
  title={A tutorial on KAM theory},
  author={De la Llave, Rafael and others},
  booktitle={Proceedings of Symposia in Pure Mathematics},
  volume={69},
  pages={175--296},
  year={2001},
  organization={Providence, RI; American Mathematical Society; 1998}
}

@article{Kolmogorov1954,
  author = {Kolmogorov, A. N.},
  title = {On conservation of conditionally periodic motions for a small change in Hamilton's function},
  journal = {Doklady Akademii Nauk SSSR},
  volume = {98},
  pages = {527--530},
  year = {1954}
}

@article{Arnold1963Proof,
  author = {Arnold, V. I.},
  title = {Proof of a theorem of A. N. Kolmogorov on the preservation of conditionally periodic motions under a small perturbation of the Hamiltonian},
  journal = {Uspekhi Matematicheskikh Nauk},
  volume = {18},
  number = {5(113)},
  pages = {13--40},
  year = {1963}
}

@article{Moser1962,
  author = {Moser, J.},
  title = {On invariant curves of area-preserving mappings of an annulus},
  journal = {Nachrichten der Akademie der Wissenschaften in G{\"o}ttingen, II. Mathematisch-Physikalische Klasse},
  year = {1962},
  pages = {1--20}
}

@article{HaroDeLaLlave2006PartII,
  title={A parameterization method for the computation of invariant tori and their whiskers in quasi-periodic maps: numerical algorithms},
  author={Haro, Alex and de la Llave, Rafael},
  journal={Discrete and Continuous Dynamical Systems Series B},
  volume={6},
  number={6},
  pages={1261},
  year={2006},
  publisher={AIMS PRESS}
}

@article{RuelleTakens1971,
  title={On the nature of turbulence},
  author={Ruelle, David and Takens, Floris},
  journal={Communications in Mathematical Physics},
  volume={20},
  number={3},
  pages={167--192},
  year={1971},
  publisher={Springer}
}

@book{Haro2016Parameterization,
  title={The Parameterization Method for Invariant Manifolds: From Rigorous Results to Effective Computations},
  author={Haro, {\`A}lex and Canadell, Marta and Figueras, Jordi-Lluis and Luque, Alejandro and Mondelo, Josep Maria},
  volume={195},
  year={2016},
  publisher={Springer}
}

@article{Jorba2001Numerical,
  author = {Jorba, {\`A}.},
  title = {Numerical computation of the normal behaviour of invariant curves of n-dimensional maps},
  journal = {Nonlinearity},
  volume = {14},
  number = {5},
  pages = {943--976},
  year = {2001}
}

@article{albers2006routes,
  title={Routes to chaos in high-dimensional dynamical systems: A qualitative numerical study},
  author={Albers, DJ and Sprott, JC},
  journal={Physica D: Nonlinear Phenomena},
  volume={223},
  number={2},
  pages={194--207},
  year={2006},
  publisher={Elsevier}
}

@article{lan2006newton,
  title={Newton’s descent method for the determination of invariant tori},
  author={Lan, Y and Chandre, C and Cvitanovi{\'c}, P},
  journal={Physical Review E—Statistical, Nonlinear, and Soft Matter Physics},
  volume={74},
  number={4},
  pages={046206},
  year={2006},
  publisher={APS}
}

@book{HirschPughShub1977,
  author    = {Hirsch, M. W. and Pugh, C. C. and Shub, M.},
  title     = {Invariant Manifolds},
  series    = {Lecture Notes in Mathematics},
  volume    = {583},
  publisher = {Springer-Verlag},
  address   = {Berlin-New York},
  year      = {1977}
}

@misc{GimenoJZ22,
   AUTHOR = {Gimeno, J. and Jorba, {\`A}. and Zou, M.},
    TITLE = {{T}aylor package, version 2}, 
     YEAR = {2022},
     NOTE = {\url{https://github.com/joang/taylor2-dist}},
}

@article{JorbaZ05,
  author = {Jorba, {\`A}. and Zou, M.},
  title = {A software package for the numerical integration of
           {ODE}s by means of high-order {T}aylor methods},
  year = 2005,
  journal = EXPMA,
  volume = {14},
  number = {1},
  pages = {99--117},
}

@Book{ChaosBook,
  title =     {Chaos: Classical and Quantum},
  publisher = {Niels Bohr Inst.},
  year =      {2016},
  author =    {P. Cvitanovi{\'c} and R. Artuso and R. Mainieri and G. Tanner and G. Vattay},
  address =   {Copenhagen},
  url =       {http://ChaosBook.org/}
}

@article{parker2022invariant,
  title={Invariant tori in dissipative hyperchaos},
  author={Parker, Jeremy P and Schneider, Tobias M},
  journal={Chaos: An Interdisciplinary Journal of Nonlinear Science},
  volume={32},
  number={11},
  year={2022},
  publisher={AIP Publishing}
}

@article{parker2023predicting,
  title={Predicting chaotic statistics with unstable invariant tori},
  author={Parker, Jeremy P and Ashtari, Omid and Schneider, Tobias M},
  journal={Chaos: An Interdisciplinary Journal of Nonlinear Science},
  volume={33},
  number={8},
  year={2023},
  publisher={AIP Publishing}
}

@article{cvitanovic2007continuous,
  title={Continuous symmetry reduced trace formulas},
  author={Cvitanovic, Predrag},
  journal={ChaosBook. org/~ predrag/papers/trace. pdf},
  year={2007}
}

@article{sanchez2010computation,
  title={Computation of invariant tori by Newton--Krylov methods in large-scale dissipative systems},
  author={S{\'a}nchez, J and Net, M and Sim{\'o}, C},
  journal={Physica D: Nonlinear Phenomena},
  volume={239},
  number={3-4},
  pages={123--133},
  year={2010},
  publisher={Elsevier}
}

@article{van2005quasi,
  title={The quasi-periodic doubling cascade in the transition to weak turbulence},
  author={Van Veen, Lennaert},
  journal={Physica D: Nonlinear Phenomena},
  volume={210},
  number={3-4},
  pages={249--261},
  year={2005},
  publisher={Elsevier}
}

@article{budanur2015periodic,
  title={Periodic orbit analysis of a system with continuous symmetry—A tutorial},
  author={Budanur, Nazmi Burak and Borrero-Echeverry, Daniel and Cvitanovi{\'c}, Predrag},
  journal={Chaos: An Interdisciplinary Journal of Nonlinear Science},
  volume={25},
  number={7},
  year={2015},
  publisher={AIP Publishing}
}

@article{lopez2005relative,
  title={Relative Periodic Solutions of the Complex Ginzburg--Landau Equation},
  author={L{\'o}pez, Vanessa and Boyland, Philip and Heath, Michael T and Moser, Robert D},
  journal={SIAM Journal on Applied Dynamical Systems},
  volume={4},
  number={4},
  pages={1042--1075},
  year={2005},
  publisher={SIAM}
}

@article{parker2022variational,
  title={Variational methods for finding periodic orbits in the incompressible Navier--Stokes equations},
  author={Parker, Jeremy P and Schneider, Tobias M},
  journal={Journal of Fluid Mechanics},
  volume={941},
  pages={A17},
  year={2022},
  publisher={Cambridge University Press}
}

@article{newhouse1978occurrence,
  title={{Occurrence of strange axiom A attractors near quasi periodic flows on $T^m, m\geq 3$}},
  author={Newhouse, Sheldon and Ruelle, David and Takens, Floris},
  journal={Communications in Mathematical Physics},
  volume={64},
  number={1},
  pages={35--40},
  year={1978},
  publisher={Springer}
}

@article{swinney1978transition,
  title={The transition to turbulence},
  author={Swinney, Harry L and Gollub, Jerry P},
  journal={Physics today},
  volume={31},
  number={8},
  pages={41--49},
  year={1978},
  publisher={American Institute of Physics}
}

@article{mainieri1989two,
  title={Two-parameter study of the quasiperiodic route to chaos in convecting- 4 3 He mixtures},
  author={Mainieri, Ronnie and Sullivan, Timothy S and Ecke, Robert E},
  journal={Physical review letters},
  volume={63},
  number={21},
  pages={2357},
  year={1989},
  publisher={APS}
}

@article{figueras2017numerical,
  title={Numerical computations and computer assisted proofs of periodic orbits of the Kuramoto--Sivashinsky equation},
  author={Figueras, Jordi-Llu{\'i}s and de la Llave, Rafael},
  journal={SIAM Journal on Applied Dynamical Systems},
  volume={16},
  number={2},
  pages={834--852},
  year={2017},
  publisher={SIAM}
}

@article{viswanath2003symbolic,
  title={Symbolic dynamics and periodic orbits of the Lorenz attractor},
  author={Viswanath, Divakar},
  journal={Nonlinearity},
  volume={16},
  number={3},
  pages={1035},
  year={2003},
  publisher={IOP Publishing}
}

@article{song2026multiscale,
  title={Multiscale quasi time-periodic coherent structures in shear flows},
  author={Song, Runjie and Deguchi, Kengo and Kawahara, Genta and Hwang, Yongyun},
  journal={arXiv preprint arXiv:2601.18023},
  year={2026}
}

@article{doohan2022state,
  title={The state space and travelling-wave solutions in two-scale wall-bounded turbulence},
  author={Doohan, Patrick and Bengana, Yacine and Yang, Qiang and Willis, Ashley P and Hwang, Yongyun},
  journal={Journal of Fluid Mechanics},
  volume={947},
  pages={A41},
  year={2022},
  publisher={Cambridge University Press}
}

@article{broer1997,
   author = {Broer and W., H. and Osinga and M., H. and Vegter and G.},
   title = {Algorithms for computing normally hyperbolic invariant manifolds},
   journal = {Zeitschrift für angewandte Mathematik und Physik},
   volume = {48},
   number = {3},
   pages = {480},
   ISSN = {0044-2275},
   DOI = {10.1007/s000330050044},
   url = {https://dx.doi.org/10.1007/s000330050044},
   year = {1997},
   type = {Journal Article}
}

@article{cabre2003i,
   author = {Cabr{\'e}, Xavier and Fontich, Ernest and de la Llave, Rafael},
   title = {The Parameterization Method for Invariant Manifolds I: Manifolds Associated to Non-resonant Subspaces},
   journal = {Indiana University Mathematics Journal},
   volume = {52},
   number = {2},
   pages = {283–328},
   ISSN = {00222518, 19435258},
   url = {http://www.jstor.org/stable/24902854},
   year = {2003},
   type = {Journal Article}
}

@article{cabre2003ii,
   author = {Cabr{\'e}, Xavier and Fontich, Ernest and de la Llave, Rafael},
   title = {The Parameterization Method for Invariant Manifolds II: Regularity with Respect to Parameters},
   journal = {Indiana University Mathematics Journal},
   volume = {52},
   number = {2},
   pages = {329–360},
   ISSN = {00222518, 19435258},
   url = {http://www.jstor.org/stable/24902855},
   year = {2003},
   type = {Journal Article}
}

@article{cabre2005,
   author = {Cabr{\'e}, Xavier and Fontich, Ernest and De La Llave, Rafael},
   title = {The parameterization method for invariant manifolds III: overview and applications},
   journal = {Journal of Differential Equations},
   volume = {218},
   number = {2},
   pages = {444–515},
   ISSN = {0022-0396},
   DOI = {10.1016/j.jde.2004.12.003},
   url = {https://dx.doi.org/10.1016/j.jde.2004.12.003},
   year = {2005},
   type = {Journal Article}
}

@article{calleja2025,
   author = {Calleja, Renato C. and Haro, Alex and Porras, Pedro},
   title = {Constructive approaches to QP-time-dependent KAM theory for Lagrangian tori in Hamiltonian systems},
   journal = {Journal of Differential Equations},
   volume = {449},
   pages = {113681},
   ISSN = {0022-0396},
   DOI = {10.1016/j.jde.2025.113681},
   url = {https://dx.doi.org/10.1016/j.jde.2025.113681},
   year = {2025},
   type = {Journal Article}
}

@article{figueras2024,
   author = {Figueras, Jordi-Llu{\'i}s and Haro, Alex},
   title = {A modified parameterization method for invariant Lagrangian tori for partially integrable Hamiltonian systems},
   journal = {Physica D: Nonlinear Phenomena},
   volume = {462},
   pages = {134127},
   ISSN = {0167-2789},
   DOI = {10.1016/j.physd.2024.134127},
   url = {https://dx.doi.org/10.1016/j.physd.2024.134127},
   year = {2024},
   type = {Journal Article}
}

@article{haro2019,
   author = {Haro, Alex and Luque, Alejandro},
   title = {A-posteriori KAM theory with optimal estimates for partially integrable systems},
   journal = {Journal of Differential Equations},
   volume = {266},
   number = {2-3},
   pages = {1605–1674},
   ISSN = {0022-0396},
   DOI = {10.1016/j.jde.2018.08.003},
   url = {https://dx.doi.org/10.1016/j.jde.2018.08.003},
   year = {2019},
   type = {Journal Article}
}

@incollection{canadell2016newton,
  title={A Newton-like method for computing normally hyperbolic invariant tori},
  author={Canadell, Marta and Haro, {\`A}lex},
  booktitle={The Parameterization Method for Invariant Manifolds: From Rigorous Results to Effective Computations},
  pages={187--238},
  year={2016},
  publisher={Springer}
}

@incollection{figueras2016parameterization,
  title={The Parameterization Method for Quasi-Periodic Systems: From Rigorous Results to Validated Numerics},
  author={Figueras, Jordi-Llu{\'\i}s and Haro, {\`A}lex},
  booktitle={The Parameterization Method for Invariant Manifolds: From Rigorous Results to Effective Computations},
  pages={75--117},
  year={2016},
  publisher={Springer}
}

@article{HdlLL1,
   author = {Haro, A. and De La Llave, R.},
   title = {A parameterization method for the computation of invariant tori and their whiskers in quasi-periodic maps: Rigorous results},
   journal = {Journal of Differential Equations},
   volume = {228},
   number = {2},
   pages = {530–579},
   ISSN = {0022-0396},
   DOI = {10.1016/j.jde.2005.10.005},
   url = {https://dx.doi.org/10.1016/j.jde.2005.10.005},
   year = {2006},
   type = {Journal Article}
}

@article{HdlLL3,
   author = {Haro, A. and De La Llave, R.},
   title = {A Parameterization Method for the Computation of Invariant Tori and Their Whiskers in Quasi‐Periodic Maps: Explorations and Mechanisms for the Breakdown of Hyperbolicity},
   journal = {SIAM Journal on Applied Dynamical Systems},
   volume = {6},
   number = {1},
   pages = {142–207},
   ISSN = {1536-0040},
   DOI = {10.1137/050637327},
   url = {https://dx.doi.org/10.1137/050637327},
   year = {2007},
   type = {Journal Article}
}

@article{HdlLL2,
   author = {Haro, Àlex and De La Llave, Rafael},
   title = {A parameterization method for the computation of invariant tori and their whiskers in quasi-periodic maps: Numerical algorithms},
   journal = {Discrete \& Continuous Dynamical Systems - B},
   volume = {6},
   number = {6},
   pages = {1261–1300},
   ISSN = {1553-524X},
   DOI = {10.3934/dcdsb.2006.6.1261},
   url = {https://dx.doi.org/10.3934/dcdsb.2006.6.1261
https://www.aimsciences.org/article/doi/10.3934/dcdsb.2006.6.1261},
   year = {2006},
   type = {Journal Article}
}

@book{BroerHuitemaSevryuk1996,
    AUTHOR = {Broer, Hendrik W. and Huitema, George B. and Sevryuk, Mikhail
              B.},
     TITLE = {Quasi-periodic motions in families of dynamical systems},
    SERIES = {Lecture Notes in Mathematics},
    VOLUME = {1645},
      NOTE = {Order amidst chaos},
 PUBLISHER = {Springer-Verlag, Berlin},
      YEAR = {1996},
     PAGES = {xii+196},
      ISBN = {3-540-62025-7},
   MRCLASS = {58F27 (34C27 34C35 70H05)},
  MRNUMBER = {1484969},
MRREVIEWER = {Luigi\ Chierchia},
}

@incollection{Chierchia2003,
    AUTHOR = {Chierchia, Luigi},
     TITLE = {K{AM} lectures},
 BOOKTITLE = {Dynamical systems. {P}art {I}},
    SERIES = {Pubbl. Cent. Ric. Mat. Ennio Giorgi},
     PAGES = {1--55},
 PUBLISHER = {Scuola Norm. Sup., Pisa},
      YEAR = {2003},
   MRCLASS = {37J40 (70H08)},
  MRNUMBER = {2071231},
MRREVIEWER = {Vassili\ G.\ Gelfreich},
}

@article{CallejaCellettiGimenoLlave2024,
    AUTHOR = {Calleja, Renato and Celletti, Alessandra and Gimeno, Joan and
              de la Llave, Rafael},
     TITLE = {Accurate computations up to breakdown of quasi-periodic
              attractors in the dissipative spin-orbit problem},
   JOURNAL = {J. Nonlinear Sci.},
  FJOURNAL = {Journal of Nonlinear Science},
    VOLUME = {34},
      YEAR = {2024},
    NUMBER = {1},
     PAGES = {Paper No. 12, 38},
      ISSN = {0938-8974,1432-1467},
   MRCLASS = {70K43 (37J40 37N05 70F15)},
  MRNUMBER = {4665059},
       DOI = {10.1007/s00332-023-09988-w},
       URL = {https://doi.org/10.1007/s00332-023-09988-w},
}

@article{CallejaCellettiGimenoLlave2022,
    AUTHOR = {Calleja, Renato and Celletti, Alessandra and Gimeno, Joan and
              de la Llave, Rafael},
     TITLE = {Efficient and accurate {KAM} tori construction for the
              dissipative spin-orbit problem using a map reduction},
   JOURNAL = {J. Nonlinear Sci.},
  FJOURNAL = {Journal of Nonlinear Science},
    VOLUME = {32},
      YEAR = {2022},
    NUMBER = {1},
     PAGES = {Paper No. 4, 40},
      ISSN = {0938-8974,1432-1467},
   MRCLASS = {70K43 (37J40 37N05 70F15)},
  MRNUMBER = {4346779},
MRREVIEWER = {Jessica\ Elisa\ Massetti},
       DOI = {10.1007/s00332-021-09767-5},
       URL = {https://doi.org/10.1007/s00332-021-09767-5},
}

@article{BroerHanssmannYou2005,
    AUTHOR = {Broer, Henk W. and Han\ss mann, Heinz and You, Jiangong},
     TITLE = {Bifurcations of normally parabolic tori in {H}amiltonian
              systems},
   JOURNAL = {Nonlinearity},
  FJOURNAL = {Nonlinearity},
    VOLUME = {18},
      YEAR = {2005},
    NUMBER = {4},
     PAGES = {1735--1769},
      ISSN = {0951-7715,1361-6544},
   MRCLASS = {37J20 (37G99 37J40 70E17 70E40)},
  MRNUMBER = {2150353},
MRREVIEWER = {Thomas\ Wagenknecht},
       DOI = {10.1088/0951-7715/18/4/018},
       URL = {https://doi.org/10.1088/0951-7715/18/4/018},
}

@incollection {HanssmannVanDerMeer2005,
    AUTHOR = {Han\ss mann, H. and van der Meer, J. C.},
     TITLE = {On non-degenerate {H}amiltonian {H}opf bifurcations in 3{DOF}
              systems},
 BOOKTITLE = {E{QUADIFF} 2003},
     PAGES = {476--481},
 PUBLISHER = {World Sci. Publ., Hackensack, NJ},
      YEAR = {2005},
      ISBN = {981-256-169-2},
   MRCLASS = {37J20 (34C23)},
  MRNUMBER = {2185075},
       DOI = {10.1142/9789812702067\_0077},
       URL = {https://doi.org/10.1142/9789812702067_0077},
}

@incollection {Hanssmann2005a,
    AUTHOR = {Han\ss mann, Heinz},
     TITLE = {Hamiltonian bifurcations of invariant tori with a vanishing
              {F}loquet exponent},
 BOOKTITLE = {E{QUADIFF} 2003},
     PAGES = {732--737},
 PUBLISHER = {World Sci. Publ., Hackensack, NJ},
      YEAR = {2005},
      ISBN = {981-256-169-2},
   MRCLASS = {37J20 (34C23 37J40)},
  MRNUMBER = {2185120},
       DOI = {10.1142/9789812702067\_0122},
       URL = {https://doi.org/10.1142/9789812702067_0122},
}

@article {Hanssmann2006,
    AUTHOR = {Han\ss mann, Heinz},
     TITLE = {On {H}amiltonian bifurcations of invariant tori with a
              {F}loquet multiplier {$-1$}},
   JOURNAL = {Dyn. Syst.},
  FJOURNAL = {Dynamical Systems. An International Journal},
    VOLUME = {21},
      YEAR = {2006},
    NUMBER = {2},
     PAGES = {115--145},
      ISSN = {1468-9367,1468-9375},
   MRCLASS = {37J20 (34C23 37J40)},
  MRNUMBER = {2241606},
MRREVIEWER = {Vincent\ Naudot},
       DOI = {10.1080/14689360500321440},
       URL = {https://doi.org/10.1080/14689360500321440},
}

@article {Hanssmann1998,
    AUTHOR = {Han\ss mann, Heinz},
     TITLE = {The quasi-periodic centre-saddle bifurcation},
   JOURNAL = {J. Differential Equations},
  FJOURNAL = {Journal of Differential Equations},
    VOLUME = {142},
      YEAR = {1998},
    NUMBER = {2},
     PAGES = {305--370},
      ISSN = {0022-0396,1090-2732},
   MRCLASS = {58F14 (34C23 58F05 58F27 70E15 70H05)},
  MRNUMBER = {1601868},
MRREVIEWER = {Ugo\ C.\ Bessi},
       DOI = {10.1006/jdeq.1997.3365},
       URL = {https://doi.org/10.1006/jdeq.1997.3365},
}

@incollection {Hanssmann2004,
    AUTHOR = {Han\ss mann, Heinz},
     TITLE = {A survey on bifurcations of invariant tori},
 BOOKTITLE = {New advances in celestial mechanics and {H}amiltonian systems},
     PAGES = {109--121},
 PUBLISHER = {Kluwer/Plenum, New York},
      YEAR = {2004},
      ISBN = {0-306-48117-0},
   MRCLASS = {37J20 (37J40 70H08)},
  MRNUMBER = {2083008},
}

@article {IoossLos1988,
    AUTHOR = {Iooss, G. and Los, J. E.},
     TITLE = {Quasi-genericity of bifurcations to high-dimensional invariant
              tori for maps},
   JOURNAL = {Comm. Math. Phys.},
  FJOURNAL = {Communications in Mathematical Physics},
    VOLUME = {119},
      YEAR = {1988},
    NUMBER = {3},
     PAGES = {453--500},
      ISSN = {0010-3616,1432-0916},
   MRCLASS = {58F14 (58F27)},
  MRNUMBER = {969212},
MRREVIEWER = {Dietrich\ Flockerzi},
       URL = {http://projecteuclid.org/euclid.cmp/1104162499},
}

@article {Los1988,
    AUTHOR = {Los, J\'er\^ome E.},
     TITLE = {D\'edoublement de courbes invariantes sur le cylindre: petits
              diviseurs},
   JOURNAL = {Ann. Inst. H. Poincar\'e{} Anal. Non Lin\'eaire},
  FJOURNAL = {Annales de l'Institut Henri Poincar\'e. Analyse Non
              Lin\'eaire},
    VOLUME = {5},
      YEAR = {1988},
    NUMBER = {1},
     PAGES = {37--95},
      ISSN = {0294-1449},
   MRCLASS = {58F14 (58F27)},
  MRNUMBER = {936889},
MRREVIEWER = {Henk\ Broer},
       URL = {http://www.numdam.org/item?id=AIHPC_1988__5_1_37_0},
}

@article {SekikawaInaba2016,
    AUTHOR = {Sekikawa, Munehisa and Inaba, Naohiko},
     TITLE = {Doubly twisted {N}eimark-{S}acker bifurcation and two
              coexisting two-dimensional tori},
   JOURNAL = {Phys. Lett. A},
  FJOURNAL = {Physics Letters. A},
    VOLUME = {380},
      YEAR = {2016},
    NUMBER = {1-2},
     PAGES = {171--176},
      ISSN = {0375-9601,1873-2429},
   MRCLASS = {37G35 (39A12)},
  MRNUMBER = {3421562},
       DOI = {10.1016/j.physleta.2015.10.040},
       URL = {https://doi.org/10.1016/j.physleta.2015.10.040},
}

@article {KamiyamaInabaSekikawaEndo2014,
    AUTHOR = {Kamiyama, Kyohei and Inaba, Naohiko and Sekikawa, Munehisa and
              Endo, Tetsuro},
     TITLE = {Bifurcation boundaries of three-frequency quasi-periodic
              oscillations in discrete-time dynamical system},
   JOURNAL = {Phys. D},
  FJOURNAL = {Physica D. Nonlinear Phenomena},
    VOLUME = {289},
      YEAR = {2014},
     PAGES = {12--17},
      ISSN = {0167-2789,1872-8022},
   MRCLASS = {37M20},
  MRNUMBER = {3270944},
       DOI = {10.1016/j.physd.2014.09.001},
       URL = {https://doi.org/10.1016/j.physd.2014.09.001},
}

@article {KuznetsovSedova2016,
    AUTHOR = {Kuznetsov, Alexander P. and Sedova, Yuliya V.},
     TITLE = {The simplest map with three-frequency quasi-periodicity and
              quasi-periodic bifurcations},
   JOURNAL = {Internat. J. Bifur. Chaos Appl. Sci. Engrg.},
  FJOURNAL = {International Journal of Bifurcation and Chaos in Applied
              Sciences and Engineering},
    VOLUME = {26},
      YEAR = {2016},
    NUMBER = {8},
     PAGES = {1630019, 12},
      ISSN = {0218-1274,1793-6551},
   MRCLASS = {37G99},
  MRNUMBER = {3533660},
       DOI = {10.1142/S0218127416300196},
       URL = {https://doi.org/10.1142/S0218127416300196},
}

@incollection {Chenciner1985,
    AUTHOR = {Chenciner, Alain},
     TITLE = {Hamiltonian-like phenomena in saddle-node bifurcations of
              invariant curves for plane diffeomorphisms},
 BOOKTITLE = {Singularities and dynamical systems ({I}r\'aklion, 1983)},
    SERIES = {North-Holland Math. Stud.},
    VOLUME = {103},
     PAGES = {7--14},
 PUBLISHER = {North-Holland, Amsterdam},
      YEAR = {1985},
      ISBN = {0-444-87641-3},
   MRCLASS = {58F14},
  MRNUMBER = {806175},
       DOI = {10.1016/S0304-0208(08)72111-X},
       URL = {https://doi.org/10.1016/S0304-0208(08)72111-X},
}

@article {SchilderOsingaVogt2005,
    AUTHOR = {Schilder, Frank and Osinga, Hinke M. and Vogt, Werner},
     TITLE = {Continuation of quasi-periodic invariant tori},
   JOURNAL = {SIAM J. Appl. Dyn. Syst.},
  FJOURNAL = {SIAM Journal on Applied Dynamical Systems},
    VOLUME = {4},
      YEAR = {2005},
    NUMBER = {3},
     PAGES = {459--488},
      ISSN = {1536-0040},
   MRCLASS = {37M99 (37D10 37M20 65L99 65P20)},
  MRNUMBER = {2173544},
MRREVIEWER = {Gerd\ P\"onisch},
       DOI = {10.1137/040611240},
       URL = {https://doi.org/10.1137/040611240},
}

@article {FiguerasHaro2025,
    AUTHOR = {Figueras, Jordi-Llu\'is and Haro, Alex},
     TITLE = {Sun-{J}upiter-{S}aturn system may exist: a verified
              computation of quasiperiodic solutions for the planar
              three-body problem},
   JOURNAL = {J. Nonlinear Sci.},
  FJOURNAL = {Journal of Nonlinear Science},
    VOLUME = {35},
      YEAR = {2025},
    NUMBER = {1},
     PAGES = {Paper No. 13, 20},
      ISSN = {0938-8974,1432-1467},
   MRCLASS = {70F07 (37C55 37N05 70-08 70G60)},
  MRNUMBER = {4827105},
MRREVIEWER = {Thomas\ Kotoulas},
       DOI = {10.1007/s00332-024-10109-4},
       URL = {https://doi.org/10.1007/s00332-024-10109-4},
}

@article {FiguerasHaroLuque2017,
    AUTHOR = {Figueras, J.-Ll. and Haro, A. and Luque, A.},
     TITLE = {Rigorous computer-assisted application of {KAM} theory: a
              modern approach},
   JOURNAL = {Found. Comput. Math.},
  FJOURNAL = {Foundations of Computational Mathematics. The Journal of the
              Society for the Foundations of Computational Mathematics},
    VOLUME = {17},
      YEAR = {2017},
    NUMBER = {5},
     PAGES = {1123--1193},
      ISSN = {1615-3375,1615-3383},
   MRCLASS = {37J40 (65G20 65G40 65T50)},
  MRNUMBER = {3709329},
MRREVIEWER = {Maciej\ J.\ Capi\'nski},
       DOI = {10.1007/s10208-016-9339-3},
       URL = {https://doi.org/10.1007/s10208-016-9339-3},
}

@article{kuptsov2012theory,
  title={Theory and computation of covariant {Lyapunov} vectors},
  author={Kuptsov, Pavel V and Parlitz, Ulrich},
  journal={Journal of nonlinear science},
  volume={22},
  number={5},
  pages={727--762},
  year={2012},
  publisher={Springer}
}

@article{ginelli2013covariant,
  title={Covariant {Lyapunov} vectors},
  author={Ginelli, Francesco and Chat{\'e}, Hugues and Livi, Roberto and Politi, Antonio},
  journal={Journal of Physics A: Mathematical and Theoretical},
  volume={46},
  number={25},
  pages={254005},
  year={2013},
  publisher={IOP Publishing}
}

@article {JorbaO09,
    AUTHOR = {Jorba, {\`A}. and Olmedo, E.},
     TITLE = {On the computation of reducible invariant tori on a parallel
              computer},
   JOURNAL = SIADS,
    VOLUME = {8},
      YEAR = {2009},
    NUMBER = {4},
     PAGES = {1382--1404}
}

@article{chenciner1979bifurcations,
  title={Bifurcations de tores invariants},
  author={Chenciner, A and Iooss, G},
  journal={Archive for Rational Mechanics and Analysis},
  volume={69},
  number={2},
  pages={109--198},
  year={1979},
  publisher={Springer}
}

@article{vitolo2011quasi,
  title={Quasi-periodic bifurcations of invariant circles in low-dimensional dissipative dynamical systems},
  author={Vitolo, Renato and Broer, Henk and Sim{\'o}, Carles},
  journal={Regular and chaotic dynamics},
  volume={16},
  number={1},
  pages={154--184},
  year={2011},
  publisher={Springer}
}

@article{chenciner1979persistance,
  title={Persistance et bifurcation de tores invariants},
  author={Chenciner, A and Iooss, G},
  journal={Archive for Rational Mechanics and Analysis},
  volume={71},
  number={4},
  pages={301--306},
  year={1979},
  publisher={Springer}
}

@article{kaas1985computation,
  title={Computation of quasi-periodic solutions of forced dissipative systems},
  author={Kaas-Petersen, Chr},
  journal={Journal of Computational Physics},
  volume={58},
  number={3},
  pages={395--408},
  year={1985},
  publisher={Elsevier}
}

@article{gonzalez2022efficient,
  title={Efficient and reliable algorithms for the computation of non-twist invariant circles},
  author={Gonz{\'a}lez, Alejandra and Haro, {\`A}lex and de la Llave, Rafael},
  journal={Foundations of Computational Mathematics},
  volume={22},
  number={3},
  pages={791--847},
  year={2022},
  publisher={Springer}
}

@book{gonzalez2014singularity,
  title={Singularity theory for non-twist KAM tori},
  author={Gonz{\'a}lez-Enr{\'\i}quez, Alejandra and Haro, Alex and De la Llave, Rafael},
  volume={227},
  number={1067},
  year={2014},
  publisher={American mathematical society}
}

@unpublished{BarcelonaGJ2026,
    author = {Barcelona, Miquel and Gimeno, Joan and Jorba-Cusc\'o, Marc},
    title = {An Explicit Graph Transform Approach to Reducible Whiskered Tori in Poincar\'e Maps},
    note = {On progress}
}

@unpublished{FiguerasGLP2027,
    author = {Figueras, Jord\'i-Llu\'is and Gimeno, Joan and de la Llave, Rafael and Parker, Jeremy},
    title = {Analytical validation of invariant tori bifurcations},
    note = {work in progress}
}

@incollection {Cvitanovic91,
    AUTHOR = {Cvitanovi\'c, Predrag},
     TITLE = {Periodic orbits as the skeleton of classical and quantum
              chaos},
      NOTE = {Nonlinear science: the next decade (Los Alamos, NM, 1990)},
   JOURNAL = {Phys. D},
  FJOURNAL = {Physica D. Nonlinear Phenomena},
    VOLUME = {51},
      YEAR = {1991},
    NUMBER = {1-3},
     PAGES = {138--151},
      ISSN = {0167-2789,1872-8022},
   MRCLASS = {58F20 (58F03 58F13 58F19 81Q50)},
  MRNUMBER = {1128807},
MRREVIEWER = {Alfredo\ M.\ Ozorio de Almeida},
       DOI = {10.1016/0167-2789(91)90227-Z},
       URL = {https://doi.org/10.1016/0167-2789(91)90227-Z},
}

\appendix
\section{Algorithms}\label{sec:algorithms appendix}
Below is a step-by-step description of the two algorithms described in Section \ref{sec:computation}. 
Both start from an approximation of the torus, its normal bundle, and the 
dynamics on it, and refine these objects and 
the parameters so that the torus is invariant. The difference between them is 
that the first uses prescribed quasiperiodic inner dynamics $\omega$, while the 
second treats one component of $\omega$ also as a parameter.

\begin{alg}[Steps to correct $(\kko,\prmo)$ and $(\nf, \lamb _{\nf})$] \label{alg.Kmu}\small \
\begin{enumerate}
 {\setlength{\itemsep}{0pt}
 \item [$\star$] \texttt{Input:} ODE like \eqref{eq.model}, ergodic frequency $\omg\in \R ^{d}$. Initial guesses of embedding $\kko\colon \T ^{d}\to \R ^n$, $\prmo \in \R ^{d}$, normal bundle $\nfo \colon \T ^{d} \to \R ^{n\times (n-d)}$, and matrix $\lamb _{\nfo} = \diag (\lamb ^s, \lamb ^u) \in \R ^{(n-d)\times (n-d)}$ with $\lamb ^s = \diag(\lambda ^s _i) \in \R ^{n_s \times n _s}$ and $\lamb ^u = \diag(\lambda ^u _j) \in \R ^{n _u \times n _u}$ such that $n _s + n _u = n - d$
 \item [$\star$] \texttt{Output:} $\kk$, $\prm $ such that \eqref{eq.inv} holds and $\nf$, $\lamb _{\nf}$ such that \eqref{eq.normalredee} holds up to a given tolerance
 \item [$\star$] \texttt{Notation:} $\li _\omg[f](\th)\bydef -\dop f(\th) \omg$ and  $\langle \eta \rangle \bydef \int _{\T ^{d}} \eta (\th) \, d\th$
 }
 \item\label{alg.Kmuiterstep} $\ee(\th) \gets \li _\omg[\kko](\th) + \vf (\kko(\th); \prmo)$\hfill {\footnotesize $\rhd \ee \colon \T ^{d} \to \R ^n$}
 % \item $\tf(\th) \gets \dop \kko(\th)$ \hfill $\tf\colon \T ^{d} \to \R ^{n\times d}$
 \item $\fo(\th) \gets 
 \begin{pmatrix}
  \dop \kko(\th) & \nfo(\th)
 \end{pmatrix}$\hfill {\footnotesize $\rhd \fo \colon \T ^{d} \to \R ^{(n\times(p +1) + n \times (n-d))}$}
 \item $(\eta ^{\tfo}, \eta ^{\nfo}) \gets \fo (\th)^{-1} \ee(\th)$\hfill {\footnotesize $\rhd \eta ^{\tfo} \colon \T ^{d} \to \R ^{d}$ and $\eta ^{\nfo} \colon \T ^{d} \to \R ^{n-d}$}
 \item $(b ^{\tfo}, b ^{\nfo}) \gets \fo (\th)^{-1}\dop _\prm \vf(\kko(\th); \prmo)$\hfill {\footnotesize $\rhd (b ^{\tfo}, b ^{\nfo}) \colon \T ^{d} \to \R ^{((d)+ (n-d))\times (d)}$}
 \item Fourier step to solve $\hprm$, $\homg _0$, and $\xi ^{\tfo}(\th)$: for all $k \in \Z ^{d}$
\begin{equation*}
 \begin{split}
  \fou \xi _0^{\tfo} &=0, \qquad \text{normalization condition,}\\
  \fou \eta _0^{\tfo} + \fou b _0 ^{\tfo} \hprm &=0, \qquad \text{for }|k| = 0\text{, } \hprm \text{ is solved}, \\
  -\I (k \cdot \omg) \fou \xi _k^{\tfo} + \fou \eta _k^{\tfo} + \fou b _k ^{\tfo} \hprm &=0, \qquad \text{for }|k| \ne 0\text{, } \fou \xi _k^{\tfo} \text{ is solved}
 \end{split}
\end{equation*}
 \item Fourier step to solve $\xi ^{\nfo}(\th)$: for all $k \in \Z ^{d}$,
\begin{equation*} 
  (\lamb _{\nfo} - \I (k \cdot \omg) )\fou \xi _k^{\nfo} + \fou \eta _k^{\nfo} + \fou b _k ^{\nfo} \hprm =0
\end{equation*}
 \item $\kko(\th) \gets \kko(\th) +  \fo(\th)
 \begin{pmatrix}
  \xi ^{\tfo}(\th) \\ \xi ^{\nfo}(\th)
 \end{pmatrix}$ and $\prmo \gets \prmo + \hprm$%\hfill [correction of $(\kk,\prm)$]
 \item $\eered(\th) \gets \li _\omg [\nfo](\th) + \dop _\x \vf(\kko(\th); \prmo) \nfo(\th) - \nfo(\th) \lamb _{\nfo}  $\hfill {\small $\rhd \eered\colon \T ^{d} \to \R ^{n\times(n-d)}$}
 \item $\fo(\th) \gets 
 \begin{pmatrix}
  \dop \kko(\th) & \nfo(\th)
 \end{pmatrix}$ \hfill {\footnotesize $\rhd $ updated frame}
 \item $
 \begin{pmatrix}
  \eta _{\texttt{red}} ^{\tfo}(\th) \\ \eta _{\texttt{red}} ^{\nfo}(\th)
 \end{pmatrix} \gets \fo(\th)^{-1}\eered(\th)$\hfill {\footnotesize $\rhd \eta _{red }^{\tfo}\colon \T ^{d} \to \R ^{d\times(n-d)}$ and $\eta _{\texttt{red}}^{\nfo} \colon \T ^{d} \to \R ^{(n-d)\times(n-d)}$}
 \item $\eta _{\texttt{red}}^{\nfo} = 
 \begin{pmatrix}
  \eta ^{ss} _{\texttt{red}} & \eta ^{su} _{\texttt{red}} \\ \eta ^{us} _{\texttt{red}} & \eta ^{uu} _{\texttt{red}} 
 \end{pmatrix}$\hfill {\footnotesize $\rhd$ block view of }
 \item Solve $(\fou Q ^{ss}_{i,j}) _k$ $(\fou Q ^{uu}_{i,j}) _k$ $(\fou Q ^{su}_{i,j}) _k$, and $(\fou Q ^{us}_{i,j}) _k$, for all $k \in \Z ^{d}$,
 \begin{align*}
  (\fou Q ^{ss}_{i,i})_0 = (\fou Q ^{uu}_{i,i})_0 &=0 & &\text{normalization condition}\\
  (\lambda _i^s - \lambda _j^s)(\fou Q ^{ss}_{i,j}) _0& =-((\fou \eta ^{ss}_{\texttt{red}})_{i,j})_0 & |k| &= 0\text{, } i\ne j\\
  (\lambda _i^u - \lambda _j^u)(\fou Q ^{uu}_{i,j}) _0& =-((\fou \eta ^{uu}_{\texttt{red}})_{i,j})_0  & |k| &= 0\text{, } i\ne j\\
  (\lambda _i^s - \lambda _j^s - \I k \cdot \omg)(\fou Q ^{ss}_{i,j}) _k& =-((\fou \eta ^{ss}_{\texttt{red}})_{i,j})_k & |k| &\ne 0 \\
  (\lambda _i^u - \lambda _j^u - \I k \cdot \omg)(\fou Q ^{uu}_{i,j}) _k& =-((\fou \eta ^{uu}_{\texttt{red}})_{i,j})_k  & |k| &\ne 0 \\
  (\lambda _i^s - \lambda _j^u - \I k \cdot \omg)(\fou Q ^{su}_{i,j}) _k& =-((\fou \eta ^{su}_{\texttt{red}})_{i,j})_k  \\
  (\lambda _i^u - \lambda _j^s - \I k \cdot \omg)(\fou Q ^{us}_{i,j}) _k& =-((\fou \eta ^{us}_{\texttt{red}})_{i,j})_k 
\end{align*}
 \item Fourier step to solve $(\fou Q _{i,j}^{\tfo}) _k$ for all $k \in \Z ^{d}$,
\begin{equation*} 
  (-\lambda _j - \I k \cdot \omg)(\fou Q ^{\tfo}_{i,j}) _k =-((\fou \eta ^{\tfo}_{\texttt{red}})_{i,j})_k
\end{equation*}
 \item $\lamb _{\nfo} \gets 
 \begin{pmatrix}
  \lambda _i^s + ((\fou \eta _{\texttt{red}}^{ss})_{i,i})_0& 0 \\ 
  0 & \lambda _i^u + ((\fou \eta _{\texttt{red}}^{uu})_{i,i})_0
 \end{pmatrix}$ and $\nfo(\th) \gets \nfo (\th) + \dop \kko(\th) Q ^{\tfo}(\th) + \nfo(\th)
 \begin{pmatrix}
  Q ^{ss}(\th) & Q ^{su}(\th) \\ Q^{us}(\th) & Q ^{uu}(\th)
 \end{pmatrix}$
 \item Iterate from step \ref{alg.Kmuiterstep} until convergence of $\ee$ and $\eered$\hfill \algoendsymbol
\end{enumerate} 
\end{alg}

\begin{alg}[Steps to correct $(\kko,\prmo, \omg _0)$ and $(\nfo, \lamb _{\nfo})$] \label{alg.Kmuomg} \small \
\begin{enumerate} 
 {\setlength{\itemsep}{0pt}
 \item [$\star$] \texttt{Input:} ODE like \eqref{eq.model}, ergodic frequency $\omg = (\omg _0, \omg _1)\in \R \times \R ^{d-1}$. Initial guesses $\omg _0 \in \R$, embedding $\kko\colon \T ^{d}\to \R ^n$, $\prmo \in \R ^{d-1}$, normal bundle $\nfo \colon \T ^{d} \to \R ^{n\times (n-d)}$, and matrix $\lamb _{\nfo} = \diag (\lamb ^s, \lamb ^u) \in \R ^{(n-d)\times (n-d)}$ with $\lamb ^s = \diag(\lambda ^s _i) \in \R ^{n_s \times n _s}$ and $\lamb ^u = \diag(\lambda ^u _j) \in \R ^{n _u \times n _u}$ such that $n _s + n _u = n - d$
 \item [$\star$] \texttt{Output:} $\kk$, $\prm $, and $\omg _0$ such that $\|\seed \ee \| \leq \tol $, $\nf$ and $\lamb _{\nf}$ such that $\|\seed \ee ^{red}\| < \tol$
 \item [$\star$] \texttt{Notation:} $\li _{(\omg _0, \omg _1)}[f](\alpha, \beta)\bydef - \partial _\alpha f(\alpha, \beta)\omg _0 - \partial _\beta f (\alpha, \beta)\omg _1 = -\dop f (\alpha, \beta) \omg$ and $\langle \eta \rangle \bydef \int _{\T ^{d}} \eta (\alpha, \beta) \, d\th$
 }
%  \item $\bar \omg \gets (\omg _0, \omg _1)$
 \item \label{alg.Kmuomgiterstep}$\seed\ee(\alpha, \beta) \gets \li _{(\omg _0, \omg _1)}[\kko](\alpha, \beta) + \vf (\kko(\alpha, \beta); \prmo)$\hfill {\footnotesize $\rhd \seed\ee \colon \T ^{d} \to \R ^n$}
 % \item $\tf(\alpha, \beta) \gets \dop \kko(\alpha, \beta)$ \hfill $\tf\colon \T ^{d} \to \R ^{n\times d}$
 \item $\fo(\alpha, \beta) \gets 
 \begin{pmatrix}
  \dop \kko(\alpha, \beta) & \nfo(\alpha, \beta)
 \end{pmatrix}$\hfill {\footnotesize $\rhd \fo \colon \T ^{d} \to \R ^{(n\times d + n \times (n-d))}$}
 \item $(\eta ^{\tfo}, \eta ^{\nfo}) \gets \fo (\alpha, \beta)^{-1} \seed\ee(\alpha, \beta)$\hfill {\footnotesize $\rhd \eta ^{\tfo} \colon \T ^{d} \to \R ^{d}$ and $\eta ^{\nfo} \colon \T ^{d} \to \R ^{n-d}$}
 \item $(b ^{\tfo}, b ^{\nfo}) \gets \fo (\alpha, \beta)^{-1}
 \begin{pmatrix}
  \dop _\prm \vf(\kko(\alpha, \beta); \prmo) & \dop _{\omg _0} \vf(\kko(\alpha, \beta); \prmo)-\partial _\alpha \kko(\alpha, \beta)
 \end{pmatrix}$ {\footnotesize $\rhd (b ^{\tfo}, b ^{\nfo}) \colon \T ^{d} \to \R ^{(d+ (n-d))\times d}$}
 \item Fourier step to solve $\hprm$, $\homg _0$, and $\xi ^{\tf}(\alpha, \beta)$: for all $k \in \Z ^{d}$
\begin{equation*}
 \begin{split}
  \fou \xi _0^{\tfo} &=0, \qquad \text{normalization condition,}\\
  \fou \eta _0^{\tfo} + \fou b _0 ^{\tfo} 
  \begin{pmatrix}
   \hprm \\ \homg _0
  \end{pmatrix} &=0, \qquad \text{for }|k| = 0\text{, } (\hprm, \homg _0) \text{ is solved}, \\
  -\I (k \cdot (\omg _0, \omg _1)) \fou \xi _k^{\tfo} + \fou \eta _k^{\tfo} + \fou b _k ^{\tfo} 
  \begin{pmatrix}
   \hprm \\ \homg _0
  \end{pmatrix} &=0, \qquad \text{for }|k| \ne 0\text{, } \fou \xi _k^{\tfo} \text{ is solved}
 \end{split}
\end{equation*}
 \item Fourier step to solve $\xi ^{\nfo}(\alpha, \beta)$: for all $k \in \Z ^{d}$,
\begin{equation*} 
  (\lamb _{\nfo} - \I (k \cdot (\omg _0, \omg _1)) )\fou \xi _k^{\nfo} + \fou \eta _k^{\nfo} + \fou b _k ^{\nfo} 
  \begin{pmatrix}
   \hprm \\ \homg _0
  \end{pmatrix} =0
\end{equation*}
 \item $\kko(\alpha, \beta) \gets \kko(\alpha, \beta) +  \fo(\alpha, \beta)
 \begin{pmatrix}
  \xi ^{\tfo}(\alpha, \beta) \\ \xi ^{\nfo}(\alpha, \beta)
 \end{pmatrix}$, $\prmo \gets \prmo + \hprm$, and $ \omg _0 \gets \omg _0 + \homg _0$%\hfill [correction of $(\kk,\prm)$]
 \item $\seed\ee ^{red}(\alpha, \beta) \gets \li _{(\omg _0, \omg _1)} [\nfo](\alpha, \beta) + \dop _\x \vf(\kko(\alpha, \beta); \prmo) \nfo(\alpha, \beta) - \nfo(\alpha, \beta) \lamb _{\nfo}  $\hfill {\footnotesize $\rhd\seed\ee ^{red}\colon \T ^{d} \to \R ^{n\times(n-d)}$}
 
 \item $\fo(\alpha, \beta) \gets 
 \begin{pmatrix}
  \dop \kko(\alpha, \beta) & \nfo(\alpha, \beta)
 \end{pmatrix}$ \hfill {\footnotesize $\rhd$ updated frame}
 \item $
 \begin{pmatrix}
  \eta _{\texttt{red}} ^{\tfo}(\alpha, \beta) \\ \eta _{\texttt{red}} ^{\nfo}(\alpha, \beta)
 \end{pmatrix} \gets \fo(\alpha, \beta)^{-1}\seed\ee ^{red}(\alpha, \beta)$ \hfill {\footnotesize $\rhd \eta _{red }^{\tfo}\colon \T ^{d} \to \R ^{d\times(n-d)}$ and $\eta _{\texttt{red}}^{\nfo} \colon \T ^{d} \to \R ^{(n-d)\times(n-d)}$ }
 \item $\eta _{\texttt{red}}^{\nfo} = 
 \begin{pmatrix}
  \eta ^{ss} _{\texttt{red}} & \eta ^{su} _{\texttt{red}} \\ \eta ^{us} _{\texttt{red}} & \eta ^{uu} _{\texttt{red}} 
 \end{pmatrix}$ \hfill {\footnotesize $\rhd$ block view of $\eta _{\texttt{red}}^{\nfo}$}
 \item Solve $(\fou Q ^{ss}_{i,j}) _k$, $(\fou Q ^{uu}_{i,j}) _k$, $(\fou Q ^{su}_{i,j}) _k$, and $(\fou Q ^{us}_{i,j}) _k$, for all $k \in \Z ^{d}$,
 \begin{align*}
  (\fou Q ^{ss}_{i,i})_0 = (\fou Q ^{uu}_{i,i})_0 &=0 & &\text{normalization condition} \\
  (\lambda _i^s - \lambda _j^s)(\fou Q ^{ss}_{i,j}) _0& =-((\fou \eta ^{ss}_{\texttt{red}})_{i,j})_0 & |k| &= 0\text{, } i\ne j\\
  (\lambda _i^u - \lambda _j^u)(\fou Q ^{uu}_{i,j}) _0& =-((\fou \eta ^{uu}_{\texttt{red}})_{i,j})_0  & |k| &= 0\text{, } i\ne j\\
  (\lambda _i^s - \lambda _j^s - \I k \cdot (\omg _0, \omg _1))(\fou Q ^{ss}_{i,j}) _k& =-((\fou \eta ^{ss}_{\texttt{red}})_{i,j})_k & |k| &\ne 0 \\
  (\lambda _i^u - \lambda _j^u - \I k \cdot (\omg _0, \omg _1))(\fou Q ^{uu}_{i,j}) _k& =-((\fou \eta ^{uu}_{\texttt{red}})_{i,j})_k  & |k| &\ne 0 \\
  (\lambda _i^s - \lambda _j^u - \I k \cdot (\omg _0, \omg _1))(\fou Q ^{su}_{i,j}) _k& =-((\fou \eta ^{su}_{\texttt{red}})_{i,j})_k  \\
  (\lambda _i^u - \lambda _j^s - \I k \cdot (\omg _0, \omg _1))(\fou Q ^{us}_{i,j}) _k& =-((\fou \eta ^{us}_{\texttt{red}})_{i,j})_k 
\end{align*}
 \item Fourier step to solve $(\fou Q _{i,j}^{\tfo}) _k$ for all $k \in \Z ^{d}$,
\begin{equation*} 
  (-\lambda _j - \I k \cdot (\omg _0, \omg _1))(\fou Q ^{\tfo}_{i,j}) _k =-((\fou \eta ^{\tfo}_{\texttt{red}})_{i,j})_k
\end{equation*}
 \item $\lamb _{\nfo} \gets 
 \begin{pmatrix}
  \lambda _i^s + ((\fou \eta _{\texttt{red}}^{ss})_{i,i})_0& 0 \\ 
  0 & \lambda _j^u + ((\fou \eta _{\texttt{red}}^{uu})_{j,j})_0
 \end{pmatrix}$ and \\
 $\nfo(\alpha, \beta) \gets \nfo (\alpha, \beta) + \dop \kko(\alpha, \beta) Q ^{\tfo}(\alpha, \beta) + \nfo(\alpha, \beta)
 \begin{pmatrix}
  Q ^{ss}(\alpha, \beta) & Q ^{su}(\alpha, \beta) \\ Q^{us}(\alpha, \beta) & Q ^{uu}(\alpha, \beta)
 \end{pmatrix}$
 \item Iterate from step \ref{alg.Kmuomgiterstep} until convergence \hfill \algoendsymbol % of $\seed\ee$ and $\seed\ee ^{red}$
\end{enumerate} 
\end{alg}

\subsection{Synthetic example}
To assess the accuracy and convergence of the proposed Algorithms~\ref{alg.Kmubif} and \ref{alg.Kmuomg}, we consider a synthetic model similar to \eqref{eq.toy-test}. 
% The corresponding code has been written from scratch in C/C++ and relies on \verb|mpfr| for multiprecision arithmetic. 
\begin{equation} \label{eq.toy-test2}
\begin{aligned}
 \dot h   &= -3 h + \varepsilon (x _1 + x _3) \\
 \dot x _1&= -7(1-r_{1,2}) x _1 + \mu _1  x _2 \tilde \omega _1 + \varepsilon \cos(h) & r _{i,j} \bydef \sqrt{x_i^2+x_j^2}\\
 \dot x _2&= -7(1-r_{1,2}) x _2 - \mu _1  x _1 \tilde \omega _1 + \varepsilon \sin(h) \\
 \dot x _3&= -5(1-r_{3,4}) x _3 + \mu _2  x _4 \tilde \omega _2 + \varepsilon \sin(h) \\
 \dot x _4&= -5(1-r_{3,4}) x _4 - \mu _2  x _3 \tilde \omega _2 + \varepsilon \cos(h).
\end{aligned}
\end{equation}
Vector-field evaluations and its derivatives required in Algorithms~\ref{alg.Kmubif} and \ref{alg.Kmuomg} are generated via the automatic parser of the \verb|taylor| package~\citep{GimenoJZ22}, which provides
a suitable output for parallelization and supports the generation of source code targeting different data
types like the \verb|mpfr| or jet type.

For $\varepsilon=0$, \eqref{eq.toy-test2} admits the explicit solution.
% The parameterization currently written in \red{red} is wrong; the \textcolor{blue}{blue} one is the correct parameterization used in the code:
\begin{equation} \label{eq.toy-init-guess2}
\begin{gathered}
 % {\color{red}
 % \kko(\th _1, \th _2) =
 % \begin{pmatrix}
  3%  0 \\ \cos (\th_1) \\ \sin(\th _1) \\ \cos (\th_2) \\ \sin(\th _2)
 % \end{pmatrix}, \qquad
 % \nfo(\th _1, \th _2) =
 % \begin{pmatrix}
 %  1 & 0 & 0 \\
 %  0 &\cos \th _1 &  0 \\
 %  0 &\sin \th _1 &  0 \\
 %  0  & 0& \cos \th _2 \\
 %  0  & 0& \sin \th _2 \\
 % \end{pmatrix},
 % }\\[0.4em]
 % {\color{blue}
 \kko(\th _1, \th _2) =
 \begin{pmatrix}
  1 \\ \cos (\th_1) \\ -\sin(\th _1) \\ \cos (\th_2) \\ -\sin(\th _2)
 \end{pmatrix}, \qquad
 \nfo(\th _1, \th _2) =
 \begin{pmatrix}
  1 & 0 & 0 \\
  0 &\cos \th _1 &  0 \\
  0 &-\sin \th _1 &  0 \\
  0  & 0& \cos \th _2 \\
  0  & 0& -\sin \th _2 \\
 \end{pmatrix},
 % }\\[0.4em]
 \lambo =
 \begin{pmatrix}
  -3 & 0 & 0\\
  0  & 7 & 0 \\
  0  & 0 & 5
 \end{pmatrix}.
\end{gathered}
\end{equation}
These functions satisfy the invariance equations \eqref{eq.inv} and \eqref{eq.normalinveq} with frequency vector
$(\omega_1,\omega_2)=(\tilde\omega_1,\tilde\omega_2)$. Throughout the tests we fix
\[
\tilde\omega_1=2,
\qquad
\tilde\omega_2=\frac{1}{\tilde \omega _1}\frac{\sqrt{5}-1}{2},
\qquad
\mu_1=\mu_2=1.
\]

% All experiments for \eqref{eq.toy-test} were performed on a standard laptop (Intel i5 @ 1.80\,GHz, 4 CPUs, 8\,GB RAM). 
We illustrate Algorithms~\ref{alg.Kmu} and \ref{alg.Kmuomg} in two stages: first, we compute Newton
corrections for small values of $\varepsilon$ starting from the same initial guess \eqref{eq.toy-init-guess};
second, we perform continuation w.r.t. $\varepsilon$ and compare the resulting continuation paths.

\subsubsection{Computation using Algorithms~\ref{alg.Kmu} and~\ref{alg.Kmuomg}}

We test Algorithms~\ref{alg.Kmu} and \ref{alg.Kmuomg} in parallel for the case $n=5$ and $d=2$.
Algorithm~\ref{alg.Kmu} corrects the parameters $(\mu_1,\mu_2)$ while keeping the ergodic frequency
$(\omega_1,\omega_2)$ fixed. In contrast, Algorithm~\ref{alg.Kmuomg} corrects the parameters
$(\omega_1,\mu_2)$ while enforcing the normalized ergodic frequency $(1,\omega_2)$ 
(note that the parameters $\tilde \omega _1$ and $\tilde \omega _2$ are fixed).
In both cases, the output is a torus tuple, consisting of an embedding $\kk$, a normal bundle $\nf$,
a matrix of normal eigenvalues $\lamb$, and the associated parameter vector $\prm$.
As initial data we use the explicit solution at $\varepsilon=0$ given in \eqref{eq.toy-init-guess}.

All runs use $60$ digits of working precision and a Newton tolerance of $10^{-25}$
(in practice, the final residual is approximately $10^{-39}$ after $5$ Newton iterations). We fix
$\varepsilon=0.01$ and discretize $(\th_1,\th_2)$ on a uniform $64\times 64$ mesh.
Table~\ref{tab.toy-sol} reports the resulting eigenvalues and parameter values obtained with both algorithms.
\begin{table}[ht]
    \centering
    \caption{Solutions for parameter and eigenvalues for the perturbation value $\varepsilon=0.01$}
    \label{tab.toy-sol}
    \[
    \begin{array}{|c||c|r|}
    \hline
      \multirow{5}{*}{\rotatebox{90}{\text{Algorithm~\ref{alg.Kmu}}}} 
      & \mu _1 & \mathtt{1.00052697210630033254839728493756871231811405034174304289635}  \\
      &\mu _2 & \mathtt{1.00001732534809581567761988705188895022301721109568914593080}  \\ \cline{2-3} \cline{2-3}
      &\lambda _1 & \mathtt{-3.00001407507960698892934006117740754266855025972259232633777} \\
      &\lambda _2 & \mathtt{6.99999461263855808906267969183936397919201057951920582992174} \\
      &\lambda _3 & \mathtt{4.99994337324595685767471272768096176370143257808820548402771} \\
    \hline \hline
      \multirow{5}{*}{\rotatebox{90}{\text{Algorithm~\ref{alg.Kmuomg}}}} 
      & \omega _1 & \mathtt{1.00052697138302398739590392210119322699701341330905208244312}  \\
      &\mu _2 & \mathtt{0.99993599939111782244076921568248481013805636329007258339164}  \\ \cline{2-3} \cline{2-3}
      &\lambda _1 & \mathtt{-3.00001723383624588439532799249972571787550377228585058677492} \\
      &\lambda _2 & \mathtt{ 6.99998850562657732923210262219104744126912949598907046750395} \\
      &\lambda _3 & \mathtt{ 4.99994337293893088365014315626024825770010129793540870984207} \\
    \hline
    \end{array}
    \]
\end{table}

For Algorithm~\ref{alg.Kmu}, Fig.~\ref{fig.toy} shows the computed invariant torus, including the two
sections $\{\theta_1=0\}$ and $\{\theta_2=0\}$. The figure also shows the projections onto the
$r_{1,2}$ and $r_{3,4}$ plane at the bottom of the plot, with points colored according to the value
of $h$.

Table~\ref{tab.toy-speed-up} reports the execution times obtained with OpenMP parallelization for the main
routines that scale with the $(\theta_1,\theta_2)$ mesh, such as vector-field evaluation, linear system
solvers, and the cohomological solver. In our current implementation, the FFT routines remain the main
bottleneck and have not been parallelized; this limits the achievable speed-up. We stress that the OpenMP
parallelization was not tuned for optimal performance. Moreover, the runs were carried out on a standard
laptop, with a relatively small mesh, and using an imbalanced number of threads (np=3), which is reflected
in Table~\ref{tab.toy-speed-up}. The results for Algorithm~\ref{alg.Kmuomg} are visually identical.

\begin{minipage}[t]{.45\textwidth}
% \begin{figure}[ht]
    \centering
    \captionof{figure}{Illustration of the torus solution \eqref{eq.toy-test} using  Algorithm~\ref{alg.Kmu}}
    \label{fig.toy}
    \includegraphics[scale=.8]{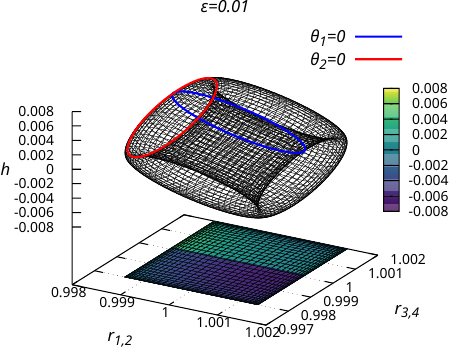}
% \end{figure}
\end{minipage} \hfill 
\begin{minipage}[t]{.45\textwidth}
% \begin{table}[ht]
    \centering
    \captionof{table}{Wall times and speed-up depending on the number of threads (np)}
    \label{tab.toy-speed-up}
    \begin{tabular}{ccc}
np & wall-time &  speed-up \\ \hline
1 & {\tt 28.374080000} & 1.00 \\
2 & {\tt 22.070379677} & 1.27 \\
3 & {\tt 23.536360025} & 1.16 \\
4 & {\tt 21.916848580} & 1.28 \\
    \end{tabular}
% \end{table}
\end{minipage}

\subsubsection{Continuation with respect to \texorpdfstring{$\varepsilon$}{epsilon}}

We perform standard continuation w.r.t. the parameter $\varepsilon$ in \eqref{eq.toy-test} using
Algorithms~\ref{alg.Kmu} and \ref{alg.Kmuomg}. We increase the $(\theta_1,\theta_2)$ mesh to $128\times 128$,
resulting in an average wall time of $85$ seconds per continuation step. Computations are carried out with
a working precision of $65$ digits for \verb+mpfr+. The Newton convergence tolerance is set to $10^{-16}$ for all continuation
steps, except for the final continuation value, where Newton is solved to a tolerance of $10^{-25}$.

The continuation step size is initialized as $\Delta\varepsilon = 2^{-6}$ and is constrained to remain in
the interval $(10^{-10},\,0.1)$. If Newton converges in fewer than $3$ iterations, we increase the step size
by $10\%$. If Newton requires at least $6$ iterations or fails to converge, we decrease the step size by $40\%$.
If Newton fails to converge for three consecutive attempts, the continuation procedure is terminated.

Figures~\ref{fig.toycontsol} and \ref{fig.toycontsolomg} illustrate the continuation runs obtained with
Algorithm~\ref{alg.Kmu} and Algorithm~\ref{alg.Kmuomg}, respectively, showing two closed continuation paths up to a common final $\varepsilon$.
 
\begin{figure}[ht]
  \centering
  \caption{Continuation w.r.t. $\varepsilon$ using Algorithm~\ref{alg.Kmu}. (1,1): Illustration of the solution at the last continuation solution; (1,2): Evolution of eigenvalues $\lambda _{2,3}$ ($\lambda _1$ stays close to $-3$); (1,3): Evolution of the parameters $\mu _{1,2}$; and (1,4): Illustration of the first and last solution and evolution sectioned on $\{\theta _2 = 0\}$}
  \label{fig.toycontsol}
  \setlength{\tabcolsep}{.5em} % horizontal padding between columns  
  \begin{tabular}{@{}cccc@{}}
    \parbox[t]{.24\textwidth}{\includegraphics[width=\linewidth,page=1]{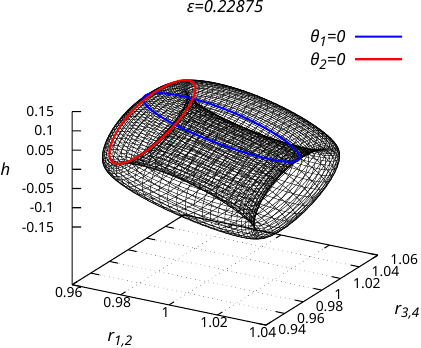}} &
    \parbox[t]{.24\textwidth}{\includegraphics[width=\linewidth,page=2]{toycontsol-crop.pdf}} &
    \parbox[t]{.225\textwidth}{\includegraphics[width=\linewidth,page=3]{toycontsol-crop.pdf}} &
    \parbox[t]{.24\textwidth}{\includegraphics[width=\linewidth,page=4]{toycontsol-crop.pdf}}
  \end{tabular}
\end{figure}

\begin{figure}[ht]
  \centering
  \caption{Same as Figure~\ref{fig.toycontsol} but using Algorithm~\ref{alg.Kmuomg} which corrects $\omega _1$ and $\mu _2$}
  \label{fig.toycontsolomg}
  \setlength{\tabcolsep}{.5em} % horizontal padding between columns  
  \begin{tabular}{@{}cccc@{}}
    \parbox[t]{.24\textwidth}{\includegraphics[width=\linewidth,page=1]{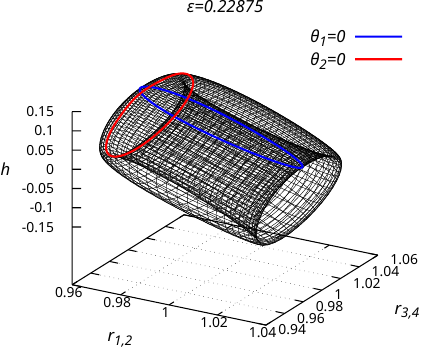}} &
    \parbox[t]{.24\textwidth}{\includegraphics[width=\linewidth,page=2]{toy2contsol-crop.pdf}} &
    \parbox[t]{.225\textwidth}{\includegraphics[width=\linewidth,page=3]{toy2contsol-crop.pdf}} &
    \parbox[t]{.24\textwidth}{\includegraphics[width=\linewidth,page=4]{toy2contsol-crop.pdf}}
  \end{tabular}
\end{figure}

\end{document}